\documentclass[12pt,oneside,reqno]{amsart}
\usepackage{mathrsfs}
\usepackage{amssymb,amsmath,amsthm,color}
\usepackage{graphicx, cite}
\usepackage{hyperref}
\usepackage{url}
\usepackage{setspace}
\usepackage{enumerate}
\usepackage[margin=1in]{geometry}
\usepackage{mathtools}
\usepackage{verbatim}

\usepackage{footnote}
\usepackage{cite}
\usepackage{amscd}
\usepackage{color}
\usepackage{euscript}
\usepackage{calc}                   

\usepackage{tikz}
\usepackage{graphicx} 

\newcommand{\norm}[1]{\left\| #1 \right\|}
\newtheorem{theorem}{Theorem}[section]
\newtheorem{lemma}[theorem]{Lemma}
\newtheorem{proof of lemma}[theorem]{Proof of Lemma}
\newtheorem{proposition}[theorem]{Proposition}

\newtheorem{corollary}[theorem]{Corollary}

\theoremstyle{definition}

\newtheorem{remark}[theorem]{Remark}

\numberwithin{equation}{section}

\newcommand{\abs}[1]{\lvert#1\rvert}
\newcommand{\M}{\mathcal{M}}
\newcommand{\cH}{\mathcal{H}}
\newcommand{\hM}{\widehat{\mathcal{M}}}
\newcommand{\Sp}{\mathfrak{S}}
\newcounter{cases}
\newcounter{subcases}

\begin{document}
	
	\title[Orthonormal Strichartz estimates]{orthonormal Strichartz estimates on torus and waveguide manifold  and applications}
	
	\author{Divyang G. Bhimani and  Subhash.  R.  Choudhary }
	
	\address{Department of Mathematics, Indian Institute of Science Education and Research-Pune, Homi Bhabha Road, Pune 411008, India}
	
	\email{divyang.bhimani@iiserpune.ac.in}
	\email{subhashranjan.choudhary@students.iiserpune.ac.in}
	
	\date{}
	
	\begin{abstract}
    We establish new orthonormal Strichartz estimates for the fractional Schr\"odinger equations on torus $\mathbb T$ and waveguide manifold $\mathbb R^n\times \mathbb T^m$.  We generalizes the  result of Nakamura \cite{nakamura2020orthonormal} on  torus; while this is the first result on the waveguide manifold. The main novelty  in this paper is the derivation of  various  kernel estimates associated to  the fractional Schr\"odinger equations.   Our kernel estimate generalizes the classical  dispersive estimate on torus  due to Kenig-Ponce-Vega  \cite{Kenig_1991_OsciDis}. On the other hand, we obtain new   $\ell^2$ decoupling inequality  for degeneracy type surfaces  to treat the case of  waveguide manifold; which maybe of independent interest and complements several known results.  As an application, we establish local  and small data global well-posednes for the Hartree equation with infinitely many particles with non-trace class initial data.
	\end{abstract}
	
	\maketitle
	\tableofcontents
	
\section{Introduction}
\subsection{Classical Strichartz estimates}\label{cst}
Consider the nonlinear Schr\"odinger equation (NLS):
\begin{equation}\label{NLS}
\begin{cases}
    i\partial_tu -\Delta u = F(u)\\
    u(t=0)=f
\end{cases} (t,z) \in \mathbb R \times \mathcal{M}.  
\end{equation}
Here $\mathcal{M}= \mathbb R^d$ (Euclidean space) or $\mathbb T^d$ (torus) or $\mathbb R^n \times \mathbb T^m$ (waveguide manifold) and $F: \mathbb C \to \mathbb C$ is a nonlinearity. 
In order to develop well-posed theory for \eqref{NLS}, Strichartz estimates of the following type
\begin{eqnarray}\label{gst}
    \|e^{-it\Delta} f\|_{L^p_t (\mathbb R, L^q_z (\mathcal{M}))} \lesssim \|f\|_{H^s_z(\mathcal{M})}
\end{eqnarray}
played a fundamental role.  It is known that \cite{KeelTao, ORS} the estimate
\begin{eqnarray}\label{ost}
    \|e^{-it\Delta} f\|_{L^p_t (\mathbb R, L^q_z (\mathbb R^d)))} \lesssim \|f\|_{L_{z}^2(\mathbb R^d)}
\end{eqnarray}
holds if and only if $(p,q)$ satisfies $\frac{2}{p}+ \frac{d}{q}= \frac{d}{2}$ with $2 \leq p, q \leq \infty$ and $(p, q, d) \neq (2, \infty,2).$
  The case $p=q$ dates back  Strichartz \cite{ORS} via the Fourier restriction method. Later, it was  generalized  combining  the duality argument and  the following dispersive  estimate:
\begin{equation}\label{cec}
    \|e^{-it\Delta} f\|_{L^{\infty}_z(\mathbb R^d)} \lesssim |t|^{-\frac{d}{2}} \|f\|_{L^1_z (\mathbb R^d)}.
\end{equation}
The end point  $(p,q)= \left( 2, \frac{2d}{d-2} \right), d\neq 2,$ was then proven in \cite{KeelTao}.
Let $f\in L_{z}^2(\mathbb R^d)$
with $ \text{supp} \hat{f} \subset [-N, N]^d.$ Then,  by \eqref{ost}, we have  
\begin{eqnarray}
    \|e^{-it\Delta} f\|_{L^{q}_{t,z}(\mathbb R \times \mathbb R^d)} \lesssim  N^{\frac{d}{2}- \frac{d+2}{q}} \|f\|_{L_{z}^2(\mathbb R^d)}
\end{eqnarray}
for $\frac{2(d+2)}{d} \leq q \leq \infty$. On torus $\mathbb T^d,$ we do not have  dispersive estimates \eqref{cec} and so Strichartz estimates of the form \eqref{gst} is  difficult. However, one would like to have the following Strichartz estimates on torus:
\begin{eqnarray}\label{ft}
    \|e^{-it\Delta} f\|_{L_{t,z}^{q}(I \times \mathbb T^d)} \lesssim_{|I|}  N^{ d/2- (d+2)/q}\|f\|_{L_{z}^2(\mathbb T^d)}
\end{eqnarray}
for all $f\in L_{z}^2(\mathbb T^d)$ and $\text{supp} \hat{f} \subset [-N, N]^d,$ where $I$ is a compact interval. In contrast with the Euclidean case, an estimate of the form  \eqref{ft} does not hold with $I=\mathbb R,$ unless $q= \infty.$
In order to develop well-posedness theory for classical  NLS  on $d-$dimensional torus, Bourgain initiated the study of \eqref{ft} in his seminal paper \cite{Bourgainrestriction} and later it is improved by Bourgain and  Demeter \cite{bourgain2015} by employing the  machinery from harmonic analysis (decoupling theorem). While Barron in  \cite[Proposition 3.4]{BarronAlex}  established similar result on waveguide manifold, where he employed the discrete restriction theorem from \cite{bourgain2015}. To state their results, we introduce 
 frequency  cut-off  operator $P_{\leq N},$   which is defined by $$P_{\leq N} \phi=(\textbf{1}_{[-N,N]} \widehat{\phi})^{\lor},$$ where $\phi$ is Schwartz function on $\mathcal{M}$,     $\widehat{\cdot}$ and ${\cdot}^{\lor}$ denotes the Fourier and inverse Fourier transform respectively (see Section \ref{pre}).
Specifically, they proved the following result. 
\begin{theorem}[\cite{bourgain2015,Bourgainrestriction, BarronAlex}]\label{TIN1}
Assume that for $\mathcal{M}= \mathbb  T^d$ or $\mathbb R^n \times \mathbb T^m$ with $d=n+m,$
\begin{eqnarray*}
    \sigma_1= \begin{cases}
        0& if  \quad 2 \leq q \leq \frac{2(d+2)}{d}\\
        \frac{d}{2}- \frac{d+2}{q} &  if \quad  \frac{2(d+2)}{d} \leq q \leq \infty
    \end{cases}.
\end{eqnarray*}
Then, for any $\epsilon >0$ and $N>1,$ we have 
    \begin{equation*}\label{nsise}
    \| e^{-it \Delta} P_{\leq N} f \|_{L^q_{t,z}(I \times \mathcal{M})} \lesssim_{\epsilon} N^{\sigma_1+\epsilon} \| f \|_{L_{z}^2 (\mathcal{M})}.
\end{equation*}
\end{theorem}
We recall that 
\begin{align*}
\widehat{\mathcal{M}}= \begin{cases}
     \mathbb R^d & if \quad \mathcal{M}=\mathbb R^d\\
     \mathbb Z^d & if \quad \mathcal{M}=\mathbb T^d\\
     \mathbb R^n \times \mathbb Z^m & if \quad   \mathcal{M}=\mathbb R^n \times \mathbb T^m  
 \end{cases},   
\end{align*}
where $\widehat{\mathcal{M}}$ denotes the Pontryagin dual of $\mathcal{M}.$
Here 
$e^{it (-\Delta)^{\theta/2}_\mathcal{M}}f(z)$ denotes the fractional  Schr\"odinger propagator and defined for $\xi \in \hM,$
\[
\mathcal{F}\left[e^{it (-\Delta)^{\theta/2}_\mathcal{M}}f(z)\right](\xi) 
   = \begin{cases}
     e^{2 \pi it \abs{\xi}^\theta } \mathcal{F}f (\xi) & if \quad \xi \in \mathbb{R}^d \text{ or } \mathbb{Z}^d \\
     e^{2 \pi it (\abs{\xi_{1}}^\theta +\abs{{\xi_{2}}}^{\theta}) } \mathcal{F}f (\xi) & if \quad \xi=(\xi_{1}, \xi_{2}) \in \mathbb{R}^n \times \mathbb{Z}^m
   \end{cases} , 
\]
 where $\theta>0.$  In \cite{laskin2002fractional}  fractional Laplacian have been applied to model physical phenomena. It was
formulated by Laskin  \cite{laskin2002fractional} as a result of extending the Feynman path integral from the
Brownian-like to L\'evy-like quantum mechanical paths.
For convenience, we denote 
$
(-\Delta)^{\theta/2}_\mathcal{M} := (-\Delta)^{\theta/2},
$
where the underlying manifold $\mathcal{M}$ will be clear from the context.
Recently Dinh  \cite[Theorems 1.1 and 1.2]{Dinh} established the analogue of Theorem \ref{TIN1} for fractional  Laplacian on torus and Euclidean spaces: 
\begin{theorem}[\cite{Dinh}]\label{HT2}
  Let $2 \leq p \leq \infty, 2 \leq q < \infty,(p,q,d) \neq (2,\infty,2)$ and $\frac{2}{p}+\frac{d}{q} \leq \frac{d}{2}.$ Assume that $\mathcal{M}= \mathbb T^d$ or  $\mathbb R^d$ and  
  \begin{equation*}
\sigma_{1} = \begin{cases}
\frac{1}{p} &  if \quad \theta \in (1,\infty) \\
\frac{2-\theta}{p} & if \quad  \theta \in (0,1)
\end{cases}.
\end{equation*} Then, for all \(N>1,\)  we have 
\begin{equation*}
\|e^{it(-\Delta)^{\frac{\theta}{2}}}P_{\leq N}f\|_{L^{p}_{t}(I, L_{z}^{q} (\mathcal{M}) )} \lesssim_{|I|} N^{\sigma_{1}}\|f\|_{L_{z}^{2}(\mathcal{M} )}.
\end{equation*}
\end{theorem}
However, to the best of authors knowledge  so far the analogue of Theorem \ref{HT2} on  waveguide manifold $\mathbb{R}^{n} \times \mathbb{T}^{m}$ remains open in the literature (cf. Theorem \ref{TIN1}).
\subsection{Known orthonormal Strichartz estimates}
In this paper we study ``orthonormal Strichartz estimates" (OSE for short)  of the following form
\begin{eqnarray}\label{onstr}
    \left\| \sum_{j} \lambda_j  |e^{it (-\Delta)^{\theta/2}_\mathcal{M}} P_{\leq N} f_j|^2\right\|_{L^p_tL^q_z (I\times \mathcal{M})} \lesssim_{|I|,\theta}  N^{\sigma} \|\lambda\|_{\ell^{\alpha'}},\quad (t, z)\in I \times \mathcal{M}
\end{eqnarray}
whenever $\M=\mathbb{T}^d  \ \text{or} \  \mathbb R^d \    \text{ or } \mathbb{R}^{n} \times \mathbb{T}^{m}.$ Here, $I$ is a bounded time interval in $\mathbb R, 1\leq p, q \leq \infty, \alpha'>0$, $1/\alpha + 1/ \alpha'=1, \theta>0,$ and $\sigma \geq 0.$ Here $(p,q)$ satisfies either the admissible condition 
\begin{eqnarray}\label{adc}
    \frac{2}{p}+\frac{d}{q}=d, \quad d= \dim (\mathcal{M}),
\end{eqnarray}
or the $\theta-$admissible condition
\begin{eqnarray}\label{theta admissible}
    \frac{\theta}{p}+\frac{d}{q}=d .
\end{eqnarray}
The initial data $(f_j)_{j}$ are a family of orthonormal functions in the homogeneous Sobolev space $\dot{H}^{s}(\mathcal{M})$. Indeed, for fermionic particles, the orthogonality assumption is natural (see Section \ref{Hinf} below).  The triangle inequality  and   classical Strichartz estimates (Subsection \ref{cst}) gives \eqref{onstr}  for $\theta=2$  and  $\alpha'=1$; and 
at this stage we do not require any use of orthogonality of the  initial data. It would be now interesting to ask  how much gain (if any) can be sought from the orthogonality by raising $\alpha' \geq 1$ as far as possible.\\

The idea of  transition from classical inequalities  for single function  to  the orthonormal family goes back to Lieb–Thirring  \cite{lieb1975bound} and found 
applications to the stability of matter. 
The OSE  of the form \eqref{onstr} goes back to Frank-Lewin-Lieb-Seiringer  \cite{frank2014strichartz} and Frank-Sabin \cite{frank2017restriction, frank2016stein}. There are many reasons  that one would like to see transition from   \eqref{gst} to  \eqref{onstr}. 
In fact, Chen, Hong, Pavlovic \cite{chen2017global, chen2018scattering} and  Frank and Sabin \cite{sabin2014hartree, frank2014strichartz} could successfully  employed  estimates of the form \eqref{onstr}   to understand dynamics of infinitely many fermions in  a quantum system. See Section \ref{Hinf} below and  \cite{nakamura2020orthonormal, Hoshiya2024, HoshiyaJMP2025, HadamaJFA2025, neal2019, nealbez2021, chen2017global,chen2018scattering,lewin2014hartree,lewin2015hartree,Hoshiya2024,Hong2019}.

For future  convenience,  we categorize  sharp Schr\"odinger admissible pair $(p,q)$ \eqref{adc} into four groups. 
See Figure \ref{fig:enter-label}.
\begin{enumerate}
    \item [(i)] Subcritical regime: $d\geq1,$ $1 \leq q < \frac{d+1}{d-1}$
    \item[(ii)] Critical point: $d \geq 2,$ $q=\frac{d+1}{d-1}$
    \item[(iii)]  Supercritical regime: $d=2,\frac{d+1}{d-1}<q<\infty$ or $d \geq 3, \frac{d+1}{d-1}<q<\frac{d}{d-2}$ 
    \item[(iv)] Keel-Tao endpoint: $d \geq 3,$ $q=\frac{d}{d-2}$
 \end{enumerate}
 
\begin{figure}[h]
		\centering
		\begin{tikzpicture}[scale=1]
			\draw[->] (0,0) -- (9,0) node[right]{$\frac{1}{q}$};
			\draw[->] (0,0) -- (0,7) node[above]{$\frac{1}{p}$};
			\coordinate (O) at (0,0);
			\coordinate (B) at (8,0);
			\coordinate (A) at (4.8,4);
			\coordinate (C) at (3.5,6);
			\draw[thick] (C) -- (A);
			\draw[thick] (A) -- (B);
			\draw[dashed] (A) -- (4.8,0);
			\draw[dashed] (C) -- (3.5,0);
			\draw[dashed] (A) -- (0,2);
			\draw[dashed] (0,4) -- (A);
            \draw[dashed] (0,6) -- (C);
			\node at (O) [left] {$O$};
			\node at (B) [above right] {$B$};
			\node at (O) [below] {$0$};
			\node at (A) [right] {$A$};
			\node at (C) [above] {$C$};
			\node at (3.5,0) [below ] {$\frac{d-2}{d}$};
			\node at (4.8,0) [below] {$\frac{d-1}{d+1}$};
			\node at (0,2) [left] {$\frac{1}{2}$};
			\node at (0,4) [left] {$\frac{d}{d+1}$};
			\node at (0,6) [left] {$1$};
			\node at (8,0) [below] {$1$};
           \node[font=\small] at (7.8,2.5) {subcritical regime};
           \node[font=\small] at (5.7,5.5) {supercritical regime};
           \fill (O) circle (2pt);
           \fill (A) circle (2pt);
           \fill (B) circle (2pt);
           \fill (C) circle (2pt);
           \fill (A) circle (2pt);
           \fill (0,6) circle (2pt);
           \fill (0,2) circle (2pt);
           \fill (0,4) circle (2pt);
		\end{tikzpicture}
		\caption{The line joining  $B$ to $C$ represents equation \eqref{adc} for $d>2.$ The point $A$ is a critical point. The segment $BA$ represents subcritical regime. The segment $AC$ represents supercritical regime.}
		\label{fig:enter-label}
	\end{figure}
For any two points $X,Y\in [0,1]^2,$ we use notation 
$$(X,Y)=\{(1-\tau)X+\tau Y: \tau \in (0,1)\}$$
$$[X,Y)=\{(1-\tau)X+\tau Y: \tau \in [0,1)\}$$
$$(X,Y]=\{(1-\tau)X+\tau Y: \tau \in (0,1]\}$$
$$[X,Y]=\{(1-\tau)X+\tau Y: \tau \in [0,1]\}$$
to represent the line segment connecting $X$ and $Y.$ We denote $XYZ$ for the convex hull of points $X,Y,Z \in [0,1]^2.$ And we also denote $XYZ \setminus X$ for the convex hull of points $X,Y,Z \in [0,1]^2$ excluding the point $X.$
\begin{theorem}[OSE on $\mathbb R^d$ \cite{frank2014strichartz,frank2016stein,neal2019}]\label{onsE} Let $\mathcal{M}=\mathbb{R}^{d},$ $I=\mathbb{R}$, $\theta=2, (f_{j})_{j}$ is an  orthonormal system in $L_{z}^2(\mathbb{R}^{d})$ and  $d \geq 1$ and $\lambda =(\lambda_{j})_{j} \in \ell^{\alpha'},$ $\alpha'\geq 1$ and $(p,q)$ satisfies \eqref{adc}.
    \begin{enumerate}
        \item \label{onsE1} (Subcritical regime) Assume that  $1 \leq q <\frac{d+1}{d-1},$  i.e. $\left( \frac{1}{q}, \frac{1}{p}\right) \in (A, B]$ in Figure \ref{fig:enter-label} and $\sigma =0.$  Then the estimate \eqref{onstr} holds if and only if  $\alpha' \leq \frac{2q}{q+1}.$ 
        \item (Critical point) On the point $(\frac{1}{q},\frac{1}{p})=(\frac{d-1}{d+1},\frac{d}{d+1})=A$ in Figure \ref{fig:enter-label}, the estimate \eqref{onstr} with $\alpha'= \frac{2q}{q+1}=p=\frac{d}{d+1}$ fails.
        \item (Supercritical regime) On the region $\frac{d+1}{d-1}<q<\frac{d}{d-2},$ i.e. $(\frac{1}{q},\frac{1}{p}) \in (A, C)$ in Figure \ref{fig:enter-label},   the estimate \eqref{onstr} hold as long as $\alpha' <p$ and this is sharp upto $\epsilon-\text{loss} $ in the sense that \eqref{onstr} fails if $\alpha' >p.$ 
        \item (Keel-Tao endpoint) On the point $(\frac{1}{q},\frac{1}{p})=(\frac{d-2}{d},1)=C,$ and $d \geq 3$ in Figure \ref{fig:enter-label}, the estimate \eqref{onstr} holds iff  $\alpha'= 1.$  
    \end{enumerate}
\end{theorem}
For further results on \eqref{onstr} with $\mathcal{M}=\mathbb R^d$ for the Laplacian in the presence of external potential, we refer to \cite{Hoshiya2024}.   Nakamura \cite{nakamura2020orthonormal} established the analogue of Theorem \ref{onsE} on torus.
\begin{theorem}[OSE on $\mathbb {T}^{d}$ \cite{nakamura2020orthonormal}]\label{onsT} 
Let $\mathcal{M}=\mathbb{T}^{d},$ $I=\mathbb{T}$, $\theta=2, (f_{j})_{j}$ is an  orthonormal system in $L_{z}^2(\mathbb{T}^{d})$ and  $d \geq 1$ and $\lambda =(\lambda_{j})_{j} \in \ell^{\alpha'}(\mathbb Z^d)$ and $\alpha'\geq 1.$ Assume that $(p,q)$ satisfies \eqref{adc} and $(p,q,d) \neq (1, \infty, 2)$.
    \begin{enumerate}
        \item \label{onsT1} Assume that   $\sigma\in(0,\frac{1}{p_*}],$ where $p_*=\frac{d+2}{d}=p=q$. Then the estimate \eqref{onstr} holds true iff $\alpha' < \frac{d}{d-\sigma}.$ 
        \item \label{onsT1} Let  $1 \leq q <\frac{d+1}{d-1},$ i.e. $(\frac{1}{q},\frac{1}{p}) \in (A, B]$ in Figure \ref{fig:enter-label} and $\sigma=\frac{1}{p}.$  Then the estimate \eqref{onstr} holds true iff $\alpha' \leq \frac{2q}{q+1}$.
        \item Let  $d=1$ and  $(\frac1q,\frac1p)=A=(0,\frac12)$ in Figure \ref{fig:enter-label} and $\sigma= \frac{1}{2}$. Then the estimate \eqref{onstr} holds true iff $\alpha' \leq 2$.
    \end{enumerate}
\end{theorem}
 Recently  Wang, Zhang and  Zhang  in \cite{wang2025strichartz} extended Theorem \ref{onsT}  to compact manifold without boundary and also consider fractional Laplacian on it.
\begin{theorem}[fractional OSE on  $\mathbb T^d$ \cite{wang2025strichartz}]\label{onsM}
    Let $\mathcal{M}=\mathbb{T}^{d},$ $I \text{ (bounded subset of }  \mathbb{R})$, $\theta \in (0,\infty)\setminus \{1\}, (f_{j})_{j}$ be an  orthonormal system in $L_{z}^2(\mathbb{T}^{d})$ and $\lambda =(\lambda_{j})_{j} \in \ell^{\alpha'}$ and $\alpha'\geq 1.$
    Let  $(p,q)$ satisfies \eqref{adc} and \begin{eqnarray*}
    \sigma_{1}= \begin{cases}
        \frac{1}{p}& if  \quad \theta >1\\
        \frac{2-\theta}{p} &  if \quad  \theta \in (0,1)
    \end{cases}.
    \end{eqnarray*} 
    \begin{enumerate}
        \item \label{onsM1} Let $\theta>1$ and $\sigma\in(0,d]$  and  $1 \leq q \leq \frac{d+2}{d}.$ Then the estimate \eqref{onstr} holds true if $\alpha' < \frac{d}{d-\sigma}.$ 
        \item \label{onsM2} Let   $1 \leq q <\frac{d+1}{d-1},$ i.e. $(\frac{1}{q},\frac{1}{p}) \in (A, B]$ in Figure \ref{fig:enter-label} and $\sigma=\sigma_{1}.$
        Then the estimate \eqref{onstr} holds true if $\alpha' \leq \frac{2q}{q+1}$.
        \item (Critical point) Let $d \geq 2.$ On the point $(\frac{1}{q},\frac{1}{p})=(\frac{d-1}{d+1},\frac{d}{d+1})=A$ in Figure \ref{fig:enter-label} and 
         $\sigma=\sigma_{1}.$
Then the estimate \eqref{onstr} holds true if $\alpha' < \frac{p}{2}$.
\item Let $d=2,\frac{(d+1)}{d-1}<q<\infty$ or $d \geq 3, \frac{(d+1)}{d-1}<q<\frac{d}{d-2},$ i.e $(\frac{1}{q},\frac{1}{p}) \in (A, C)$ in Figure \ref{fig:enter-label} and $\sigma=\sigma_{1}.$ Then the estimate \eqref{onstr} holds true if $\alpha' \leq \frac{p}{2}.$ Moreover, if $d \geq 5,\theta>1, \frac{(d+1)}{d-1}<q<\frac{d}{d-2},$ i.e $(\frac{1}{q},\frac{1}{p}) \in (A, C),$ and $\sigma=
\frac{1}{p}.$ Then the estimate \eqref{onstr} holds true if $\alpha' < \frac{pd(d-3)}{8+pd(d-4)}.$
\item (Keel-Tao endpoint) On the point $(\frac{1}{q},\frac{1}{p})=(\frac{d-2}{d},1)=C,$ and $d \geq 3$ in Figure \ref{fig:enter-label} the estimate \eqref{onstr} holds iff  $\alpha'= 1.$ 
    \end{enumerate}
\end{theorem}

\subsection{Statement of the main results}
\subsubsection{OSE on torus}
 Our first main result gives  the fractional orthonormal  Strichartz estimates on 1D torus. Specifically, we have the following theorem.
 \begin{figure}[h]
		\centering
		\begin{tikzpicture}
			\draw[->] (0,0) -- (5,0) node[right]{$\frac{1}{q}$};
			\draw[->] (0,0) -- (0,4) node[above]{$\frac{1}{p}$};
			\coordinate (O) at (0,0);
			\coordinate (B) at (4,0);
			\coordinate (C) at (0,0.7);
			\draw[thick] (C) -- (B);
			\draw[dashed] (0,3) -- (4,3);
			\draw[dashed] (4,3) -- (4,0);
            \draw[thick] (4,0) -- (0,2);
            \draw[thick] (0,0.7) -- (1.4,1.3);
            \draw[dashed] (0,1.3) -- (1.4,1.3);
            \draw[dashed] (1.4,1.3) -- (1.4,0);
			\node at (O) [left] {$O$};
			\node at (B) [above right] {$B$};
			\node at (O) [below] {$0$};
			\node at (C) [below right] {$C$};
            \node at (0,2) [above right] {$A$};
			\node at (0,0.7) [left] {$\frac{1}{\theta}$};
            \node at (0,2) [left] {$\frac{1}{2}$};
            \node at (0,1.3) [left] {$\frac{1}{3}$};
			\node at (0,3) [left] {$1$};
			\node at (4,0) [below] {$1$};
            \node at (1.4,1.3) [above] {$D$};
            \node at (1.4,0) [below] {$\frac{1}{3}$};
            \fill (O) circle (2pt);
            \fill (B) circle (2pt);
            \fill (C) circle (2pt);
            \fill (0,2) circle (2pt);
            \fill (0,1.3) circle (2pt);
            \fill (1.4,1.3) circle (2pt);
            \fill (1.4,0) circle (2pt);
        \end{tikzpicture}
		\caption{The points $A$ to $B$ in one dimension for $\theta >3.$}
		\label{fig:2}
	\end{figure}

\begin{theorem}\label{ose}
    Let $\theta >2$ and $p,q \in [1, \infty]$ satisfies the $\theta$-admissible condition \eqref{theta admissible}, that is, $(\frac{1}{q},\frac{1}{p}) \in [C,B]$ in Figure \ref{fig:2}. 
  Let  $N > 1$, $\lambda \in \ell^{\alpha'} (\mathbb Z) $ and $(f_j)$ are a family of orthonormal functions in $L_{z}^{2}(\mathbb{T}).$  
  Then  the following estimate 
		\begin{equation}\label{IN12}
			\left\| \sum_{j} \lambda_{j} |e^{it(-\Delta)^{\frac{\theta}{2}}} P_{\leq N} f_{j}|^{2} \right\|_{L_{t}^{p}L_{z}^{q}(I \times \mathbb{T})} \lesssim_{|I|,\theta}N^{\frac{\theta-1}{p}} \| \lambda \|_{\ell^{\alpha'}}
		\end{equation}
	holds if  $$\alpha' \leq \frac{2 q}{q+1}.$$ 
  \end{theorem}
Theorem \ref{ose} generalizes  Theorem \ref{onsT}\eqref{onsT1} in the fractional setting.   We should  also note that   the range of $\theta$-admissible pairs \((p, q)\) established in Theorem \ref{ose} is not included in Theorem \ref{onsM} \eqref{onsM2} and we could cover the entire line segment \([C, B]\) as  in Figure \ref{fig:2}.   
\begin{corollary}\label{region strichartz}
    Let $\theta >2$ and $(\frac{1}{q},\frac{1}{p}) \in ABC \setminus A,$ in Figure \ref{fig:2}. Then there exist two points $(\frac{1}{q_{0}},\frac{1}{p_{0}}) \in (A,B]$ and $(\frac{1}{q_{1}},\frac{1}{p_{1}}) \in [C,B]$ such that $$(\frac{1}{q},\frac{1}{p})=(1-\tau)(\frac{1}{q_{0}},\frac{1}{p_{0}})+ \tau (\frac{1}{q_{1}},\frac{1}{p_{1}})$$ for some $\tau \in [0,1].$ Then
    \begin{equation}
			\left\| \sum_{j} \lambda_{j} |e^{it(-\Delta)^{\frac{\theta}{2}}} P_{\leq N} f_{j}|^{2} \right\|_{L_{t}^{p}L_{z}^{q}(I \times \mathbb{T})} \lesssim_{|I|,\sigma}N^{\sigma} \| \lambda \|_{\ell^{\alpha'}}
		\end{equation}
	holds for all orthonormal system $(f_{j})_{j} \subset L_{z}^{2}(\mathbb{T})$ and all sequences $\lambda \in \ell^{\alpha'}(\mathbb Z),$ and $ \sigma= (1-\tau) \frac{1}{p_{0}}+\frac{\tau(\theta-1)}{p_{1}}$ and $\alpha' \leq (1-\tau)\alpha'_{0}+\tau \alpha'_{1},$ where $\alpha'_{k}= \frac{2q_{k}}{q_{k}+1},$ for $k=0,1.$
    \end{corollary}
    A notable aspect of  Corollaries \ref{region strichartz} and \ref{region arbitrary strichartz} is that it provides an OSE for the fractional Schrödinger equation in 1D over a region, rather than only a line segment. To the best autors  knowledge, such a result has not been previously established on torus.
    Corollary \ref{region strichartz} is derived from Theorems \ref{ose} and  \ref{onsM} \eqref{onsM2}  through the process of complex interpolation.
\begin{corollary}\label{region arbitrary strichartz}
    Let $\theta >2$ and $(\frac{1}{q},\frac{1}{p}) \in CDB,$ in Figure \ref{fig:2}. Then there exist two points $(\frac{1}{q_{0}},\frac{1}{p_{0}}) \in [D,B]$ and $(\frac{1}{q_{1}},\frac{1}{p_{1}}) \in [C,B]$ such that $$(\frac{1}{q},\frac{1}{p})=(1-\tau)(\frac{1}{q_{0}},\frac{1}{p_{0}})+ \tau (\frac{1}{q_{1}},\frac{1}{p_{1}})$$ for some $\tau \in [0,1].$ Then
    \begin{equation}
			\left\| \sum_{j} \lambda_{j} |e^{it(-\Delta)^{\frac{\theta}{2}}} P_{\leq N} f_{j}|^{2} \right\|_{L_{t}^{p}L_{z}^{q}(I \times \mathbb{T})} \lesssim_{|I|,\sigma}N^{\sigma} \| \lambda \|_{\ell^{\alpha'}}
		\end{equation}
	holds for all orthonormal system $(f_{j})_{j} \subset L_{z}^{2}(\mathbb{T})$ and all sequences $\lambda \in \ell^{\alpha'}(\mathbb Z),$ and all $ \sigma \in (\frac{\tau(\theta-1)}{p_{1}}, \frac{\theta-1}{p_{1}}]$ and $\alpha' < \frac{2(\theta-1)}{2(\theta-1)-\sigma \theta}.$
\end{corollary}
Corollary \ref{region arbitrary strichartz} differs from Corollary \ref{region strichartz} in several ways. For instance, the region in Corollary \ref{region strichartz} is larger than that in Corollary \ref{region arbitrary strichartz}. To compare them within a common region, consider any point \((\frac{1}{q},\frac{1}{p}) \in CDB\) as  in Figure \ref{fig:2}. There exist two points, \((\frac{1}{q_{0}},\frac{1}{p_{0}}) \in [D,B]\) and \((\frac{1}{q_{1}},\frac{1}{p_{1}}) \in [C,B]\), such that $$(\frac{1}{q},\frac{1}{p}) = (1-\tau)(\frac{1}{q_{0}},\frac{1}{p_{0}}) + \tau (\frac{1}{q_{1}},\frac{1}{p_{1}}),$$ 
for some \(\tau \in [0,1].\) Furthermore, we can always find a \(\sigma\) such that
$$\frac{\tau(\theta-1)}{p_{1}} < \sigma < (1-\tau)(\frac{1}{q_{0}},\frac{1}{p_{0}}) + \tau (\frac{1}{q_{1}},\frac{1}{p_{1}}),
$$ and this \(\sigma\) will satisfy Corollary \ref{region arbitrary strichartz}. This indicates that the derivative loss in Corollary \ref{region arbitrary strichartz} can be smaller than that in Corollary \ref{region strichartz}, i.e,  we can choose derivative loss arbitrary close to $\frac{\tau(\theta-1)}{p_{1}},$ which we cannot do in Corollary \ref{region strichartz}. Corollary \ref{region arbitrary strichartz} is derived from Theorem \ref{ose} and  Theorem  \ref{onsM} \eqref{onsM1} through the process of complex interpolation.
Thus Theorem \ref{ose} and Corollaries \ref{region strichartz} and \ref{region arbitrary strichartz} complements Theorems \ref{onsT} and \ref{onsM}.

\begin{remark}  The key tool to prove Theorem \ref{ose} is the dispersive estimate established in Proposition \ref{l8}. This  generalizes  the classical result due to Kenig-Ponce-Vega  \cite[(5.9)]{Kenig_1991_OsciDis}. The dimension restriction in  Theorem \ref{ose} comes due to this proposition.
\end{remark}

\subsubsection{Decoupling inequalities and fractional Strichatz estimates on waveguide manifold}\label{cstwg} \

 In next   Theorem \ref{diagonal or nondiagonal for single stri} we establish  Strichartz estimates for a single function, which will be useful to obtain OSE in   Corollary \ref{arbitrary loss ons}. In order  to prove  Theorem \ref{diagonal or nondiagonal for single stri},  we  shall prove new  decoupling inequalities  for  phase function $\phi(\xi_1, \xi_2)= |\xi_1|^{\theta} + |\xi_{1}|^{\theta}$ (degeneracy type) on $\mathbb R^d$ and on waveguide manifold.  See Sections \ref{decoupling sec on R} and  \ref{decoupling sec on wave} below for a precise statements and motivation. This maybe of independent interest to the reader.

\begin{theorem}\label{diagonal or nondiagonal for single stri}
Let \(d=m+n \geq 1\), \(N >1\). Suppose \(p \geq 2\) and \(\frac{1}{p} = \frac{d}{2}\left(\frac{1}{2} - \frac{1}{q}\right)\) and \(f \in L_z^{2}(\mathbb{R}^n \times \mathbb{T}^m)\). Moreover,
\begin{itemize}
    \item for  \(\theta \geq 2\), assume that 
    \begin{equation}\label{derivative loss for arbitrary}
\sigma_{1} = \begin{cases}
0 & if \quad  2 \leq q \leq \frac{2(d+2)}{d} \\
\frac{d}{2} - \frac{d+2}{q} & if \quad \frac{2(d+2)}{d} \leq q \leq \infty
\end{cases},
\end{equation}
\item for $1< \theta<2,$ assume that $n\geq2, $  $f$  is radial in $\mathbb{R}^n$ and \begin{equation}\label{derivative loss for arbitrary 1< theta <2}
\sigma_{1} = \begin{cases}
(2-\theta) \frac{d}{2}(\frac{1}{2}-\frac{1}{q}) & if \quad  2 \leq q \leq \frac{2(d+2)}{d} \\
\frac{d}{2} - \frac{d+\theta}{q} & if \quad \frac{2(d+2)}{d} \leq q \leq \infty
\end{cases}.
\end{equation}
\end{itemize}
Then,  for \(\varepsilon > 0\), we have 
\begin{equation}\label{arbitrary loss for wave guide for q}
\|e^{it(-\Delta)^{\frac{\theta}{2}}} P_{\leq N}f\|_{L^{q}_{t,z}(I \times \mathbb{R}^n \times \mathbb{T}^m)} \lesssim_{|I|, \varepsilon} N^{\sigma_{1}+\varepsilon}\|f\|_{L_z^{2}(\mathbb{R}^n \times \mathbb{T}^m)}
\end{equation}
and \begin{equation}\label{arbitrary loss for wave guide for pq}
\|e^{it(-\Delta)^{\frac{\theta}{2}}}P_{\leq N}f\|_{L^{p}_{t}L^{q}_{z}(I \times \mathbb{R}^n \times \mathbb{T}^m)} \lesssim_{|I|,\varepsilon} N^{\sigma_{1}+\varepsilon}\|f\|_{L_z^{2}(\mathbb{R}^n \times \mathbb{T}^m)}.
\end{equation}
\end{theorem}
The case \(\theta = 2\) was previously established by Barron \cite[Proposition 3.4]{BarronAlex}. In contrast, the more general non-degenerate phase function, which is away from the origin, has been studied by Sire-Yu-Yue-Zhao in \cite{sire2023}. In Theorem~\ref{diagonal or nondiagonal for single stri}, a radiality assumption on the spatial variable is imposed to prevent additional derivative loss (see \cite[Theorem 1.5(b)]{guo2014improved}). Theorem \ref{diagonal or nondiagonal for single stri} complements Theorem \ref{TIN1}.

\subsubsection{Orthonormal Strichatz estimates on waveguide manifold}
We are now ready to  state OSE  on  the waveguide manifold.

\begin{theorem}\label{Strichartz}
    Let $d=n+m \geq 1, \theta \in (0,\infty) \setminus \{1, 2\}, N> 1.$ Suppose $\frac{d+1}{d} < p \leq \infty,1 \leq q < \frac{d+1}{d-1}$ and $\frac{2}{p}+\frac{d}{q}=d,$ i.e $(\frac{1}{q},\frac{1}{p}) \in (A,B]$ in Figure \ref{fig:enter-label}. Let 
\begin{eqnarray}\label{hn}
    \sigma=\begin{cases}
(1+\frac{n(\theta-2)}{m+n})\frac{1}{p} & if \quad \theta>3+\frac{m}{n} \\
\frac{1}{p} & if  \quad  \theta =2\\
    \frac{2}{p} & if \quad 1< \theta \leq 3+\frac{m}{n}, \theta \neq 2 \\
    \frac{2(2-\theta)}{p} & if \quad    \theta \in (0,1)
\end{cases}.
\end{eqnarray}

Then 
\begin{equation}\label{stri}
		\left\| \sum_j \lambda_j |e^{it(-\Delta)^\frac{\theta}{2}} P_{\leq N}f_j|^2 \right\|_{L_t^p(I,L_z^q(\mathbb{R}^n \times \mathbb{T}^m))} \lesssim_{|I|} N^{\sigma} \| \lambda \|_{\ell^{\alpha'}},
	\end{equation}
   holds for all orthonormal system $(f_{j})_{j}$ in $L^{2}_{z}(\mathbb{R}^n \times \mathbb{T}^m)$ if $\alpha' \leq \frac{2q}{q+1}.$ 
\end{theorem}

Theorem~\ref{Strichartz} is new in the setting of the waveguide manifold 
 and complements Subsection \ref{cstwg} and Theorems \ref{onsE} and \ref{onsT}. In  this setting  the derivative loss $N^{\sigma}$ depends    on the order of dispersion $\theta$ (except certain range); while it is  $N^{1/p}$ in the  setting of torus, see Theorem~\ref{onsM}\eqref{onsM2}. The range of $\alpha$ and $(A,B]$ as shown in Figure~\ref{fig:enter-label} remain same as in Theorems~\ref{onsE}\eqref{onsE1} and \ref{onsT}\eqref{onsT1}.  The point $A$ which corresponds to the endpoint of the desired inequality is  excluded due to   Hardy-Littlewood-Sobolev inequality.

\begin{remark} The key step in the proof is to obtain the  kernel estimate. For $\theta \neq 2,$ we use  localization procedure via the Littlewood-Paley decomposition on both $\mathbb{R}^n$ and $\mathbb{T}^m$  and by subdividing the time interval into smaller subintervals.  Then we invoke Lemma \ref{dispersive on R}   to obtain the necessary kernel bounds (see \eqref{kernel for waveguide}-\eqref{kernel for waveguide 3}).

While for $\theta =2,$ we use the fact that the higher-dimensional kernel can be expressed as a product of 1D kernels (Proposition \ref{l8} and Lemma \ref{gle} \eqref{l1}).
 Here we note that the improvement in derivative loss arises from the fact that the proof does not rely on the Littlewood-Paley decomposition.
\end{remark}

\begin{remark}  Theorem~\ref{Strichartz} also holds true for the supercritical regime and critical point. See Figure \ref{fig:enter-label}. In fact,  the  inequality \eqref{stri} holds true by using Lemma \ref{dispersive on R} for small time interval at the Keel–Tao endpoint (see \cite{KeelTao}), given by $(p,q) = (1, \tfrac{d}{d-2})$ for $d \ge 3$, with $\alpha' = 1$.
Then in  the supercritical regime and critical point, the inequality \eqref{stri} follows ($\alpha'<p$) from Theorem~\ref{Strichartz} (subcritical regime) and is obtained by interpolation between the Keel-Tao endpoint $(p,q) = (1, \tfrac{d}{d-2})$ (with $\alpha' = 1$) and points arbitrarily close to $(p,q) = (\tfrac{d+1}{d}, \tfrac{d+1}{d-1})$ (with $\alpha' = \tfrac{2q}{q+1}$). See \cite{feng2024orthonormal} and \cite{frank2016stein}. 
\end{remark}

\begin{corollary}\label{arbitrary loss ons}
    Let \(d=m+n \geq 1\), \(N > 1\). Suppose \(1 \leq q \leq \frac{d+2}{d}\) and $\frac{2}{p}+\frac{d}{q}=d.$ Let 
$$\sigma_{1}=\begin{cases}
(1+\frac{n(\theta-2)}{m+n})\frac{1}{p} & if \quad \theta>3+\frac{m}{n} \\
    \frac{2}{p} & if \quad 2 < \theta \leq 3+\frac{m}{n}
\end{cases},$$
and $$\kappa=\begin{cases}
1+\frac{n(\theta-2)}{m+n} & if \quad \theta>3+\frac{m}{n} \\
    2 & if \quad 2 < \theta \leq 3+\frac{m}{n}
\end{cases}.$$
Then
\begin{equation}
\Big\|\sum_{j}\lambda_{j}|e^{it(-\Delta)^{\frac{\theta}{2}}}P_{\leq N}f_{j}|^{2}\Big\|_{L^{p}_{t}L^{q}_{z}(I \times \mathbb{R}^n \times \mathbb{T}^m)} \lesssim_{|I|, \sigma} N^{\sigma}\|\lambda\|_{\ell^{\alpha'}}
\end{equation}
holds for all orthonormal systems \((f_{j})_{j} \subset L_z^{2}(\mathbb{R}^n \times \mathbb{T}^m)\) and all sequences \(\lambda = (\lambda_{j})_{j} \in \ell^{\alpha'}\), and all \(\sigma \in (0,\sigma_{1}]\) and \(\alpha' < \frac{d}{d-\frac{\sigma}{\kappa}}.\)
\end{corollary}

In Corollary~\ref{arbitrary loss ons}, the exponent \(\alpha'\) is dependent on \(\theta\). As \(\theta\) increases, the corresponding value of \(\alpha'\) decreases. Additionally, as derivative loss $\sigma$ increases, the value of \(\alpha'\) also increases.

\section{Hartree Eqs. with infinitely many  particles}\label{Hinf}
\subsection{Motivation and description of the problem}
  There are two fundamental types of elementary particles in nature,  bosons and fermions. The  finite system of $M$ fermions (bosons) with mean field interactions involving a pair potential $w: \M \to \mathbb R$  is described by $M$ \textbf{coupled Hartree equations}:
\begin{equation}\label{HN}
\begin{cases} i\partial_t u_j = ( ( - \Delta )^{\frac{\theta}{2}} + w \ast \rho) u_j\\
u_{j|t=0}=f_{j}
\end{cases} \quad (t, z)\in I \times \mathcal{M}
\end{equation}
where $j=1,2,..., M.$ To account for the Pauli principle, the family  $(f_{j})_{1}^{M}$  assumed to be an  orthonormal system (ONS for short) in $L_{z}^2(\M)$.   Here,  $\rho$ is the total density  functions of particles defined by
\begin{eqnarray}
\rho(t,z)= \sum_{k=1}^{M}|u_k(t,z)|^2.
\end{eqnarray}
We introduce the  one-particle density matrix corresponding to  \eqref{HN} as
\begin{eqnarray*}
    \gamma_M (t) = \sum_{j=1}^M | u_j(t)\rangle \langle u_j(t) |,
\end{eqnarray*}
where  $|u \rangle \langle v|$ denotes the operator $f\mapsto \langle v, f\rangle u$ from $L^2$ into itself.  This  corresponds to the rank$-M$ orthogonal projection onto the span of the orthonormal family $\{f_{j}\}_{j=1}^M.$ Then system \eqref{HN} is  equivalent  to  a single operator-valued equation \footnote{This means that the solution $\gamma(t)$ to \eqref{oHN} stays of  finite rank for all times, with $\gamma(t)=  \sum_{j=1}^M|u_j(t)\rangle \langle u_j(t)|,$ where $u_j$ satisfies \eqref{HN}. We note that  taking $M=1$ or rank-one operator $\gamma_1(t=0)=|f_1\rangle \langle f_1|$ in   \eqref{HN}  and \ref{oHN} respectively, we recover   classical  Hartree  equation.}
\begin{equation}\label{oHN}
    \begin{cases}
        i \partial_t\gamma_{M} = [(-\Delta)^{\frac{\theta}{2}}+ w \ast \rho_ {\gamma_M}, \gamma_M]\\
 \gamma_M(t=0) = \sum_{j=1}^M | f_{j}\rangle \langle f_{j} |
    \end{cases},
\end{equation}
where  $[X, Y]=XY-YX$ is the operator commutator.  The  density function is given by 
\[\rho_{\gamma_M} (t,z)= \gamma_{M} (t, z,z). \]
The convolution $w \ast \rho_ {\gamma_M}$ represents the self-generated potential  which acts  a multiplication operator  in \eqref{oHN}.
More generally, we consider trace class operator $\gamma_0$, i.e. $\operatorname{Tr} |\gamma_0|= \sum_j |\lambda_j|< \infty;$
in this  case spectral decomposition  gives
\[\gamma_0 = \sum_{j} \lambda_j |f_{j}\rangle \langle f_{j}|.\]
This corresponds to a dilute gas with a finite density (for e.g, with $\rho(t,z)= \sum_{j=1}^M |u_j(t,z)|^2$ as $M\to \infty$ or $\rho(t,z)= \sum_{j=1}^{\infty} \lambda_j |u_j (t,z)|^2, \lambda_j >0$ and $\sum \lambda_j=1)$) and  \eqref{HN} is equivalent to  the following system of coupled equations:
\begin{equation}\label{fds}
     i\partial_t u_j = \left( ( - \Delta )^{\frac{\theta}{2}} + w \ast (\sum_{k} \lambda_k |u_k|^2)\right) u_j, \quad 
u_{j|t=0}=f_{j}, \ 
 j \in \mathbb N.
\end{equation}

In fact, the derivation of \eqref{fds} arises  from a many-body quantum system in the mean-field limit
(e.g.,  \cite{BenjaminCMP2014} \cite[Theorem 3.4]{Jacky}), and the global-in-time Cauchy problem  \eqref{fds}  has been resolved for trace class
initial data with finite energy in \cite{Bove, ChenAmieTori,BhiHaq,BhiGriOko}

On the other hand,  it is possible that  the one-particle density matrix $\gamma = \sum_{j=1}^{\infty} | u_j \rangle \langle u_j|$ is not of trace class; however, it has a bounded operator norm  $L^{2}_{z}(\M) \to L^{2}_{z}(\M).$  We note that  the analysis of the dynamics of \eqref{oHN} is much more difficult when $\gamma_0$ is not of trace class.  While this situation arises in various physical phenomena, e.g. Femi gas at zero/positive temperature, Bose
gas at positive temperature, and Boltzmann gas at positive temperature, see \cite{lewin2014hartree,lewin2015hartree}. Taking this consideration into account, Lewin and Sabin \cite{lewin2014hartree,lewin2015hartree} introduced suitable Schatten-Sobolev spaces  $\Sp^{p,s}$(to be define in Section \ref{schp}) to treat the non-trace class data $\gamma_0$. See also \cite{chen2017global,chen2018scattering}. We shall see that key ingredient to treat non trace class data is  the OSE of  the form \eqref{onstr}.
\subsection{Application to  well-posedness} We consider fractional Hartree  equation 
\begin{equation}\label{HEInP}
 \begin{cases}
i \partial_t\gamma = [(-\Delta)_{\M}^{\frac{\theta}{2}}+ w \ast \rho_\gamma, \gamma]\\
\gamma|_{t=0}=\gamma_0
\end{cases},
\end{equation}
describing the mean-field dynamics of an interacting gas containing infinitely many fermions in the $\mathcal{M}= \mathbb T$ (tori) or $\mathbb R^n \times \mathbb T^m$ (waveguide manifold).  
 Here,  the unknown $\gamma=\gamma(t)$ is a bounded self-adjoint operator on $L^2(\M)$
representing the one-particle density matrix and $\rho_\gamma:\mathbb R \times \M \to \mathbb R$ is the  scalar density function  associated to $\gamma(t),$ formally defined by 
$$\rho_\gamma(t, z)=\gamma (t, z,z),$$
where $\gamma (t, x, x')$ denote (with the abuse of notation) the integral kernel of operator $\gamma(t),$ i.e.,
\[\gamma (t) \phi (x) = \int_{\mathcal{M}} \gamma (t, x, x') \phi (x') dx'.\]The interaction potential  $w$ is in  Besov space $ B^{s}_{q', \infty}(\mathcal{M}) $  (to be define in  Subsection \ref{Bs} below). Some typical examples of $w:$
\[w(x) = |x|^{-a} \in B^{s}_{q', \infty}(\mathbb T^d), \quad  (0<a <d,  \ a\leq \frac{d}{q'}-s),\]
where $|x|^{-a}$ should be understood as  a periodic polynomial.
Similarly \[w(z) = |z|^{-a} \in B^{s}_{q', \infty}(\mathbb R^n \times \mathbb T^m), \quad  ( \frac {n}{q'}<a <n+m,  \ a\leq \frac{n+m}{q'}-s),\]

We are now ready to state our well-posedness result. To this end, assume that pair \(\left( \frac{1}{q}, \frac{1}{p} \right) \) satisfies the following conditions
\begin{eqnarray}\label{LWAss}
     \left( \frac{1}{q}, \frac{1}{p} \right) \in\begin{cases} 
    (C,B)  \   \text{in Figure \ref{fig:2}} & for \ \theta >2, \M= \mathbb T, \sigma= \frac{\theta-1}{2p}\\
    (A,B)  \   \text{in Figure \ref{fig:enter-label}} & for \ \theta \in (0, \infty)\setminus \{1\}, \M= \mathbb R^n \times  \mathbb T^m, \sigma \  as  \ in  \ \eqref{hn}.
    \end{cases}
\end{eqnarray}

\begin{theorem}\label{TIN7}
Let  pair \(\left( \frac{1}{q}, \frac{1}{p} \right) \) be as in \eqref{LWAss} with $1 <q< \infty,$  \(s > \frac{\sigma}{2}\), and  $w\in B^{s}_{q', \infty}(\mathcal{M}) $ (Besov space). Assume that \(\gamma_{0} \in \Sp^{\frac{2 q}{q+1}, s} (L^2_z (\M))\) (Schatten space). Then 

	\begin{enumerate}
	    \item \label{TIN71} (local well-posedness) there exist 
			\(
			{\small T = T\left(\| \gamma_{0} \|_{\Sp^{\frac{2 q}{q+1}, s}}, \| w \|_{B^{s}_{q', \infty}} \right) > 0}
			\)
			and a unique solution 
			\[
			\gamma \in C([0, T], \Sp^{\frac{2 q}{q+1}, s} (L^2 (\M)))
			\]
			satisfying \eqref{HEInP} on \([0, T] \times \M\) whose density 
			\(
			\rho_\gamma \in L_{t}^{p} L_{z}^{q} ([0, T] \times \M).\)
   \item \label{TIN72} (small data global well-posedness) For any $T >0,$ there exists a small ${\small R_{T}=R_{T}\left(\| w \|_{B^{s}_{q', \infty}}\right)}$ $>0$ such that if $\| \gamma_{0} \|_{\Sp^{\frac{2 q}{q+1}, s}} \leq R_{T},$ then there exists a unique  solution $$\gamma \in C([0, T], \Sp^{\frac{2 q}{q+1}, s} (L^2 (\M))),$$ satisfying \eqref{HEInP} on $[0,T] \times \M$ whose density $\rho_{\gamma} \in L_{t}^{p} L_{z}^{q} ([0, T] \times \M).$
   \end{enumerate}
	\end{theorem}
The key tool to prove Theorem \ref{TIN7} is our OSE established in  Theorems  \ref{ose} and \ref{Strichartz}. It is notable that though  there are  many authors have studied NLS and Hartree  equation on waveguide manifold (see e.g., \cite{Takaoka2001, Deng2024JFA}), Theorem \ref{TIN7} is the first result  for the  infinite system of Hartree equations (with non trace class data) on waveguide manifold.

Now we briefly mention the history.  Frank--Sabin \cite{frank2017restriction} and 
Lewin and Sabin \cite{lewin2015hartree, lewin2014hartree}  initiated the study of \eqref{HEInP} and obtained some  local and global existence in some suitable Schatten space   under  some strong assumption on  the interaction potential  $w$. Later, Chen, Hong and Pavlovi\'c in \cite{chen2017global, chen2018scattering} treated  the more singular delta interaction case  $w=\delta$. While Bez, Lee and Nakamura  \cite[Proposition 10]{nealbez2021} considered potential $w\in B^{s}_{q, \infty}(\mathbb R^n)$ for some $s,q$. Recently Hadama and Hong \cite[Theorem 1.2]{HadamaJFA2025} considered interaction potential $w$ as a finite measure on $\mathbb R^3$. For the global well-posedness (without any smallness assumption on initial data), we refer to \cite[Theorem 2.18]{HoshiyaJMP2025}.  On torus, 
Nakamura \cite[Theorem 1.7]{nakamura2020orthonormal} studied the case $\theta=2$ and $\omega(x) = |x|^{-a}$ ; and recently Wang--Zhang--Zhang \cite[Section 6]{wang2025strichartz} extended to other dispersion $\theta$ and on   compact Riemannian manifolds without boundary. We note that  waveguide manifold case is not covered in  the literature so far. Thus, Theorem \ref{TIN7} complement these ongoing investigation in this direction.  

\begin{remark}
    The point $B$ is excluded because of H\"ormander–Mikhlin Theorem $(1<q<\infty).$  See \eqref{WL2}.
\end{remark}
	
\section{Preliminaries}\label{pre}
\subsection{Notations}\label{notation}
We write \( A \lesssim B \) to denote that there exists a constant \( C > 0 \) such that \( A \leq C B \). Similarly, we use \( A \lesssim_u B \) to indicate that there exists a constant \( C(u) > 0 \), depending on \( u \), such that \( A \leq C(u) B \).
We write \( A \sim B \) if both \( A \lesssim B \) and \( B \lesssim A \) hold.
The notation \( A \approx B \), though less rigorous, will sometimes be used to indicate that \( A \) and \( B \) are essentially of the same size, up to small and insignificant error terms.
For any two functions $f$ and $g,$ when we write
\[
f(x) = O\big(g(x)\big) \quad \text{as } x \to a,
\]
it means there exist constants $\delta > 0$ and $M > 0$ (where \( f \) and \( g \) are defined) such that
\[
  \quad |f(x)| \leq M |g(x)|, \text{ if } 0 < |x - a| < \delta.
\]
We use usual Lebesgue spaces $L_{z}^p:=L^{p}_{z}{(\M)}$ and $L_{t}^{p}L_{z}^{q}(I \times \M)=L_{t}^{p}(I,L_{z}^{q}(\M)),$ where $I$ is a bounded subset of $\mathbb{R}$; if not, we will mention it. 
Let us denote \begin{equation*}
    z=\begin{cases}
      x \in \mathbb{R}^d & if \quad z \in \mathbb{R}^d \\
      x \in \mathbb{T}^d & if \quad z \in \mathbb{T}^d \\
      (x,y) \in \mathbb{R}^n \times \mathbb{T}^n & if \quad z \in \mathbb{R}^n \times \mathbb{T}^m
    \end{cases},
\end{equation*}
and if $\xi \in \mathbb{R}^n \times \mathbb{Z}^n$ then $\xi=(\xi_{1},\xi_{2}),$ where $\xi_{1} \in \mathbb{R}^n$ and $\xi_{2} \in \mathbb{Z}^m.$ \\
We define the Fourier transform on $\mathcal{M}$ as follows:
$$(\mathcal{F}f)(\xi)=\widehat{f}(\xi)=\int_{\mathcal{M}} f(z) e^{-2 \pi i z \cdot \xi}dz, \quad \xi \in \widehat{\mathcal{M}}.$$
Let $\eta_{1}: \mathbb{R} \to [0,1]$ be a smooth even function satisfying 
\begin{equation*}
   \eta_{1}(x)= \begin{cases}
        1 &  if \quad \abs{x} \leq 1 \\
       0  &  if \quad \abs{x} \geq 2
    \end{cases}.
\end{equation*}
Then we form $\eta^{n+m}$ on $\mathbb R^{n+m}= \mathbb R^d$ as follows:
 \[\eta^{d}(\xi)=\eta_{1}^{n}(\xi_1)\eta_{1}^{m}(\xi_2)=\eta_{1}(\xi'_{1})\eta_{1}(\xi'_{2}) \cdots \eta_{1}(\xi'_{n+m})\]
and for $\xi \in \mathbb{R}^n \times \mathbb{R}^m,$ we take 
$\eta(\xi) = \mathbf{1}_{S_{d,1}}(\xi)$ or $\eta^d(\xi)$; 
in the proof, we will specify which choice of $\eta$ is being used, where $S_{d,1}$ to be defined below in \eqref{rdp}.
Let $N \in 2^{\mathbb{N}}$ be a dyadic integer.
Now we will define Littlewood-Paley projectors on $\M$ by 
$$\widehat{P_{1}f}(\xi)= \eta(\xi) \widehat{f}(\xi), \quad \xi \in \hM,$$
$$\widehat{P_{\leq N}f}(\xi)= \eta\left(\frac{\xi}{N}\right) \widehat{f}(\xi), \quad \xi \in \hM,$$
and 
$$P_{N}f=P_{\leq N}f-P_{\leq \frac{N}{2}}f,$$
where $P_{\leq \frac{1}{2}}=0.$
Let $\varphi \in C_0^\infty(\mathbb{R} \setminus \{0\}).$ Then for any Schwartz function $f$ on  $\mathcal{M}$ the operator $\varphi(h\sqrt{\Delta})$ defined as follows:
$$\varphi(h\sqrt{-\Delta})f(x)={\left(\varphi(h \abs{\xi})\widehat{f}(\xi) \right)}^{\vee}(x) \quad (\xi \in \widehat{\mathcal{M}} , h \in [0,1) ).$$
\subsection{Besov spaces} \label{Bs}
Now for $s \in \mathbb{R}$ and $p \in [1, \infty]$,
\[
B^s_{p, \infty}(\M) = \left\{ f : \|f\|_{B^s_{p, \infty}(\M)} < \infty \right\},
\]
where
\[
\|f\|_{B^s_{p, \infty}(\M)} := \sup_{N \in 2^\mathbb{N} } N^{s} \|P_N f\|_{L^p(\M)}.
\] 
Properties of Besov space (denote $B^s_{p, \infty}:=B^s_{p, \infty}(\M)$):
\begin{enumerate}
    \item
Let $1 \leq p_1 \leq p_2 \leq \infty.$
Then, for any real number $s,$ the space $B^{s}_{p_1, \infty}$ is continuously embedded in 
$B^{s - d\left(\frac{1}{p_1} - \frac{1}{p_2}\right)}_{p_2, \infty}.$
\item A constant $C$ exists which satisfies the following properties. 
If $s_1$ and $s_2$ are real numbers such that $s_1 < s_2,$ 
$\theta \in \, ]0,1[,$ and $p \in [1,\infty],$ then we have
\[
\|u\|_{B^{\theta s_1 + (1-\theta)s_2}_{p,\infty}} 
\leq 
\|u\|^{\theta}_{B^{s_1}_{p,\infty}} 
\|u\|^{1-\theta}_{B^{s_2}_{p,\infty}}
\]
and
\[
\|u\|_{B^{\theta s_1 + (1-\theta)s_2}_{p,1}} 
\leq 
\frac{C}{s_2 - s_1} 
\left( 
\frac{1}{\theta} + \frac{1}{1-\theta} 
\right)
\|u\|^{\theta}_{B^{s_1}_{p,\infty}} 
\|u\|^{1-\theta}_{B^{s_2}_{p,\infty}}.
\]
\item Let $a$ be in $]0,d[.$ For any $p$ in $[1,\infty],$ the function $|\,\cdot\,|^{-a}$ 
belongs to $B^{\frac{d}{p}-a}_{p,\infty}.$
\end{enumerate}
The proofs of these properties for $\M = \mathbb{R}^d$ can be found in \cite[Chapter~2]{bahouri2011basic}, and with similar calculations, one can show that they also hold for $\M = \mathbb{T}^d$ and $\M = \mathbb{R}^n \times \mathbb{T}^m$.
\subsection{Schatten spaces $\Sp^{\alpha}$} \label{schp}
    Denote Hilbert space $$\mathcal{H}= L_{t}^2 L_{z}^2(I \times \mathcal{M})$$ where $I \subset \mathbb R$  is  a time  interval. Let $\mathcal{B}(\cH)$ denotes the space of all bounded operators on $\cH.$ The \textbf{Schatten space} $\Sp^{\alpha }(\cH) \ (\alpha \in [1, \infty))$ consists of all compact operators $A \in \mathcal{B}(\cH)$ such that $\operatorname{Tr}|A|^{\alpha} < \infty,$ where $|A| = \sqrt{A^* A}, A^* \text{ is adjoint of }A$.   The  Schatten space norm is  given by 
$$\|A\|_{\Sp^{\alpha}}=\|A\|_{\Sp^{\alpha}(\cH)} = (\operatorname{Tr}|A|^{\alpha})^{\frac{1}{\alpha}}.$$ 
 The Schatten space norm coincides with the operator norm for  $\alpha = \infty,$ 
	 thus we  define $$\|A\|_{\Sp^{\infty}}=\|A\|_{\Sp^{\infty}(\cH)} = \|A\|_{\cH\to \cH}.$$ We may introduce regularity index $s$ in the spirit of Sobolev spaces in the realm of Schatten spaces.  Specifically,  we introduce  Sobolev-Schatten space $ \Sp^{\alpha, s}(\cH) \ (s \in \mathbb{R})$, 
	  by the norm
	\[ \|A\|_{\Sp^{\alpha, s}}=
	\|A\|_{\Sp^{\alpha, s}(\cH)} = \|\langle D \rangle^s A \langle D \rangle^s\|_{\Sp^{\alpha}(\cH)},
	\]
	where $\langle D \rangle^s$ is the inhomogeneous derivative and $\langle D \rangle^s \phi = \left((1 + |\eta|^2)^{\frac{s}{2}} \hat{\phi}\right)^\vee$. 
     We also use the result (see in \cite[Theorem 2.7]{simon2005trace} ): If $T_{1},T_{2}$ are two bounded linear operators on Hilbert space $\cH$ and \(T_{3} \in \Sp^{\alpha} (\cH)\), then 
		\begin{equation}\label{composition of operators}
		    \|T_{1}T_{3}T_{2}  \|_{\Sp^{\alpha}(\cH)} \leq \| T_{1} \|_{\cH \to \cH} \| T_{3} \|_{\Sp^{\alpha}(\cH)} \|T_{2}\|_{\cH \to \cH}.
		\end{equation}
        And H\"older's inequality for the Schatten space (see \cite[Theorem 2.8]{simon2005trace}):
        \begin{equation}\label{Holder for Schatten}
            \|ST  \|_{\Sp^{\alpha}(\cH)} \leq  \| S \|_{\Sp^{\alpha_{1}}(\cH)} \|T\|_{\Sp^{\alpha_{2}}(\cH)},
        \end{equation}
        for all $S \in \Sp^{\alpha_{1}}(\cH)$ and $T \in \Sp^{\alpha_{2}}(\cH)$ and $\frac{1}{\alpha}=\frac{1}{\alpha_{1}}+\frac{1}{\alpha_{2}}.$
\subsubsection{\textbf{Duality principle in $\Sp^{\alpha}$}}
\begin{lemma}[Lemma 3 in \cite{frank2017restriction}]\label{PL1} Let $p, q \geq 1,$ and $\alpha \geq 1.$ Let $T$ be a bounded operator from a separable Hilbert space $\cH$ to $L_{t}^{2p}L_{z}^{2q}(I \times \M)$. Then the following are equivalent.
		\begin{enumerate}
			\item[(i)] There is a constant $C > 0$ such that
			\begin{equation*}
				\left\|WTT^*W\right\|_{\Sp^\alpha (\cH)} \leq C\left\|W\right\|^{2}_{L_t^{2p'}L_x^{2q'}(I \times \M)},
			\end{equation*} for all $ W \in {L_t^{2p'}L_z^{2q'}(I \times \M)}.$ Where the function $W$ act as multiplication operator. 
			\item[(ii)] There is a constant $C' > 0$ such that for any ONS $(f_j)_{j \in \mathbb{Z}} \subset L_{z}^{2}(\mathcal{M})$ and any sequence $(\lambda_j)_{j \in \mathbb{Z}} \subset \mathbb{C}$,
			\begin{equation*}
				\left\| \sum_{j \in \mathbb{Z}} \lambda_j |Tf_j|^2 \right\|_{L_t^{p}L_z^{q}(I \times \M)} \leq C' \left( \sum_{j \in \mathbb{Z}} |\lambda_j|^{\alpha'} \right)^{1/\alpha'}.
			\end{equation*}
	\end{enumerate} \end{lemma}

\subsubsection{\textbf{Complex interpolation in $\Sp^{\alpha}$}}
Let $a_0, a_1 \in \mathbb R$ such that $a_0<a_1$ and strip $S= \{ z' \in \mathbb C: a_0 \leq \operatorname{Re}(z') \leq a_1\}.$  Denote Hilbert space $\mathcal{H}= L^2_{t} L_{z}^2(I \times \mathcal{M})$ where $I \subset \mathbb R$  is  a time  interval.
We say a  family of operators  $\{T_{z'}\} \subset \mathcal{B}(\mathcal{H})$  defined in a strip  $S$ is analytic in the sense of Stein if it has the following properties:
\begin{itemize}
    \item[-] for each $z' \in S,$ $T_{z'}$ is a linear transformation  of simple functions on $I \times \mathcal{M}$
    \item[-] for all simple functions $F, G$ on $I \times \mathcal{M},$ the map $z' \mapsto \langle G, T_{z'} F \rangle$ is analytic in the interior of the strip  $\stackrel{\circ}{S}$ and continuous in $S$
\item[-] $\sup_{a_0 \leq \lambda \leq a_1} |\langle G, T_{\lambda +is} F \rangle| \leq C(s)$ for some $C(s)$ with at most a (double) exponential growth  in $s$.
\end{itemize}
\begin{lemma} \cite[Theorem 2.9]{simon2005trace}\label{s-interpolation} Let $\{T_z\} \subset \mathcal{B}(\mathcal{H})$ be an  analytic family of operators on $I \times  \mathcal {M}$ in the sense of Stein defined in the strip  $S$. If there exist $M_0, M_1, b_0, b_1>0,$ $1 \leq p_0, q_0, p_1, q_1$ and $1\leq r_0, r_1  \leq  \infty$ such that for all simple functions $W_1, W_2$ on $I \times \mathcal{M}$ and $s \in \mathbb R,$ we have 
\[ \|W_1 T_{a_0 +is} W_2 \|_{\Sp^{r_0}(\mathcal{H})} \leq M_0 e^{b_0 |s|}\|W_1\|_{L^{q_0} (I, L^{p_0} (\mathcal{M}))}\|W_2\|_{L^{q_0} (I, L^{p_0} (\mathcal{M}))}\]
and 
\[ \|W_1 T_{a_1 +is} W_2 \|_{\Sp^{r_1}(\mathcal{H})} \leq M_1 e^{b_1|s|} \|W_1\|_{L^{q_0} (I, L^{p_0} (\mathcal{M}))}\|W_2\|_{L^{q_0} (I, L^{p_0} (\mathcal{M}))}\]
Then for all $\tau \in [0,1]$ we have 
 \[ \|W_1 T_{a_1 +is} W_2 \|_{\Sp^{r_\tau}(\mathcal{H})} \leq M_0^{1-\tau}M_1^{\tau}\|W_1\|_{L^{q_0} (I, L^{p_0} (\mathcal{M}))}\|W_2\|_{L^{q_0} (I, L^{p_0} (\mathcal{M}))}\]
 where $a_\tau, r_\tau, p_\tau$ and $q_\tau$ are defined by 
 \[a_\tau = (1- \tau)a_0 + \tau a_1, \quad \frac{1}{r_\tau}= \frac{1-\tau}{r_0}+ \frac{\tau}{r_1}, \quad  \frac{1}{p_\tau} = \frac{1-\tau}{p_0}+ \frac{\tau}{p_1}, \quad \frac{1}{q_\tau} = \frac{1-\tau}{q_0}+ \frac{\tau}{q_1}.\]
\end{lemma}

\subsection{Fourier extension and restriction operators}
In order to restrict  the  frequency variable on $d-$dimensional cube, we introduce
\begin{eqnarray}\label{rdp}
  S_{d,N}  = \hM \cap [-N, N]^d,  \quad d= \dim (\M).  
\end{eqnarray}
Define
\begin{eqnarray*}
  \varphi (\xi)= \begin{cases}
      |\xi|^{\theta} & if \quad \xi \in \mathbb Z^m\\
      |\xi_1|^\theta + |\xi_2|^\theta & if \quad \xi=(\xi_1, \xi_2)\in \mathbb R^n \times \mathbb Z^m
  \end{cases}. 
\end{eqnarray*}
Let $\eta$ be the function mentioned earlier in \ref{notation}, and let $\varphi$ be defined as above. For $a  \in L_{\xi}^2(\hM),$ the  \textbf{Fourier extension operator} $\mathcal{E}_N$  is given by 
	\begin{equation}\label{FEP}
	   \mathcal{E}_N a (t,z) = \int_{\hM} a(\xi)  e^{2\pi i (z \cdot \xi + t\varphi (\xi))} \eta(\frac{\xi}{N}) d\xi, \quad(t, z) \in I \times \M. 
	\end{equation}
By Plancherel theorem,  it follows that   $\mathcal{E}_{N}:L_{\xi}^2(\hM) \to \cH $ is a bounded operator. In fact, we have 
	\begin{align} \label{d1}
		\|\mathcal{E}_{N}a\|_{\cH}^{2}&= \int_{I} \|\mathcal{E}_{N}a(t, \cdot)\|^2_{L_{z}^{2}(\M)} \, dt =\int_{I} \|\widehat{\mathcal{E}_{N}a(t, \cdot)}(\xi)\|^2_{L_{\xi}^{2}(\hM)} \, dt \lesssim \|a\|_{L_{\xi}^{2}(\hM)}^{2}.
	\end{align}    
The dual\footnote{the dual operator of $\mathcal{E}_N$ means that
	$\langle \mathcal{E}_N a, F \rangle_{\cH} = \langle a, \mathcal{E}_N^* F \rangle_{L^2(\hM)}
	$
	holds for all $a$ and $F$} of  the operator $\mathcal{E}_N$ is denoted by $\mathcal{E}^*_N-$so called the \textbf{Fourier restriction operator}. Specifically, $\mathcal{E}_N^*: \cH \to L_{\xi}^2(\hM)$ is
	given by
 \begin{equation}\label{FRO}
   F \mapsto \mathcal{E}_N^* F (\xi) =
	      \int_{I \times \M} F(t, z) e^{-2\pi i (z \cdot \xi + t\varphi (\xi))} \eta(\frac{\xi}{N}) \, dz dt 
.  
 \end{equation} 
We note that a composition of $\mathcal{E}_N$ and $\mathcal{E}^*_N$  gives us a convolution operator $\cH$. In fact, we have 
  \begin{flalign}\label{P1}
		\mathcal{E}_N \circ \mathcal{E}_N^* F(t,z) &= \int_{\hM } \mathcal{E}_N^*F(\xi) e^{2\pi i (z \cdot \xi + t\varphi (\xi))} \eta(\frac{\xi}{N}) \, d\xi  \nonumber \\
		&=\int_{ \hM }\int_{I \times \M} F(t',z') e^{-2\pi i (z' \cdot \xi + t'\varphi (\xi))}e^{2\pi i (z \cdot \xi + t\varphi (\xi))} \left(\eta(\frac{\xi}{N})\right)^{2} \, d\xi \, dz' \, dt' \nonumber \\
		&=\int_{I \times \M} F(z',t')\int_{ \hM }e^{2\pi i [(z-z') \cdot \xi + (t-t')\varphi (\xi)]} \left(\eta(\frac{\xi}{N})\right)^{2} \, d\xi \, dz' \, dt' \nonumber \\
        & =  \int_{I \times \M} K_N (z - z', t - t') F(z', t') \, dz' dt'.
	\end{flalign}  
   where 
	\begin{equation}\label{kernel of extension}
	    K_N (t,z) = \int_{\hM} e^{2\pi i (z \cdot \xi + t \varphi (\xi))} \left(\eta(\frac{\xi}{N})\right)^{2} \, d\xi.
	\end{equation}
	\begin{remark}\label{usrev} We can restate the inequality \eqref{onstr} as follows: it holds for every $N > 1$, any $\lambda \in \ell^{\alpha'}$, and any orthonormal system $(f_j)_j \subset L_{z}^2(\M)$ if and only if
	\begin{equation}\label{P2}
		\left\| \sum_j \lambda_j |\mathcal{E}_N a_j|^2 \right\|_{L_t^p L_z^q} \leq C_\sigma N^{\sigma} \| \lambda \|_{\ell^{\alpha'}}
	\end{equation}
	holds for any $N > 1$, $\lambda \in \ell^{\alpha'}$, and any ONS $(a_j)_j \subset L_{\xi}^{2}(\hM)$. This is
	because if we let $a_j = \widehat{f_j}$, then the orthonormality of $(f_j)_j$ in $L_{z}^2(\M)$ is equivalent
	to the one of $(a_j)_j$ in $L_{\xi}^{2}(\hM)$ and $$e^{it (-\Delta)^{\frac{\theta}{2}}} P_{\le N}f_j = \mathcal{E}_N a_j.$$ 
    \end{remark}
	\begin{lemma}[Hardy-Littlewood-Sobolev inequality \cite{frank2014strichartz}]\label{PL3}
		Let $p,q > 1$ \text{and} $\lambda \in [0,1).$ Let $\frac{1}{p}+\frac{1}{q}+\lambda=2 \text{ and } \frac{1}{p}+\frac{1}{q} \geq 1.$ Let $f \in L^{p}(\mathbb{R})$ and $ g \in L^{q}(\mathbb{R}).$ Then we have 
		$$\left|\int_{\mathbb{R}}\int_{\mathbb{R}} \frac{f(x) g(y)}{|x-y|^{\lambda}} \,dx \, dy\right| \leq C \|f\|_{L^{p}(\mathbb{R})}\|g\|_{L^{q}(\mathbb{R})}.$$
	\end{lemma}
     \begin{lemma}[see p. 11 \cite{bahouri2011basic}] \label{PL5}
     Let $1 \leq p_{0},p_{1},q_{0},q_{1},\alpha_{0}, \alpha_{1} \leq \infty.$ And let $T: L_{t}^{p_{0}}L_{z}^{q_{0}} \to \ell^{\alpha_{0}}$ and $T: L_{t}^{p_{1}}L_{z}^{q_{1}} \to \ell^{\alpha_{1}}$ be bounded operator with norm $A_{0}$ and $A_{1}$ respectively. Then $T: L_{t}^{p_{\tau}}L_{z}^{q_{\tau}} \to \ell^{\alpha_{\tau}}$ is bounded operator with norm $A_{\tau}$ for $\tau \in [0,1],$ where $$\left(\frac{1}{p_{\tau}},\frac{1}{q_{\tau}},\frac{1}{\alpha_{\tau}}\right)=(1-\tau)\left(\frac{1}{p_{0}},\frac{1}{q_{0}},\frac{1}{\alpha_{0}}\right)+\tau \left(\frac{1}{p_{1}},\frac{1}{q_{1}},\frac{1}{\alpha_{1}}\right) \text{ and } A_{\tau}=A_{0}^{1-\tau}A_{1}^{\tau}.$$
 \end{lemma}
 \begin{lemma}[Hilbert-Schmidt operators \cite{simon2005trace}] \label{PL6}
      Let $H = L^2(M, d\mu)$ for some separable measure space (i.e., one with $H$ separable). If $A \in \Sp^2(H)$, then there exists a unique function $K \in L^2(M \times M, d\mu \otimes d\mu)$ with
\begin{equation}\label{P4}
    (A\varphi)(x) = \int_{M} K(x,y) \varphi(y) d\mu(y)
\end{equation}
Conversely, any $K \in L^2(M \times M)$ defines an operator $A$ by \eqref{P4} which is in $\Sp^2(H)$ and
\[
\|A\|_{\Sp^{2}(H)} = \|K\|_{L^{2}(M \times M, d\mu \otimes d\mu)}.
\]
 \end{lemma}

\begin{lemma}{\cite[Propositions 2.5 and 3.5]{Dinh}}\label{dispersive on R}
        Let $\theta \in (0,\infty) \setminus \{1\}$. Let $\varphi \in C_0^\infty(\mathbb{R} \setminus \{0\})$ and $\mathcal{M}= \mathbb R^d $ or $\mathbb T^d.$ There exists $t_0 > 0$ and $C > 0$ such that for any $h \in (0,1],$
\begin{equation*}
\left\| e^{it{(-\Delta)}^{\frac{\theta}{2}}} \varphi(h\sqrt{\Delta}) f \right\|_{L_{z}^\infty(\mathcal{M})} \leq C h^{-d} (1 + |t| h^{-\theta})^{-d/2} \|f\|_{L_{z}^1(\mathcal{M})}
\end{equation*}
for each $t \in [-t_0 h^{\theta -1}, t_0 h^{\theta -1}].$
 \end{lemma}
\begin{lemma}\cite{cho2011}\label{fractional stri on Rd}
    If $\theta \in \mathbb{R} \setminus \{0,1\}$, for $(p, q) \in [2,\infty)^{2}$ such that $\frac{1}{p} \leq \frac{d}{2}(\frac{1}{2}-\frac{1}{q}),$ the classical Strichartz estimate
$$
\left\|e^{i t(-\Delta)^{\theta / 2}} f\right\|_{L_t^p L_z^q(\mathbb{R} \times \mathbb{R}^d)} \lesssim\|f\|_{\dot{H}^s(\mathbb{R}^d)}, \quad s=\frac{d}{2}-\frac{d}{q}-\frac{\theta}{p}
$$
holds.
\end{lemma}

\begin{lemma}\cite[Proposition 2.1]{chen2024}\label{Lp multiplier}
Let $1 \leq p \leq \infty$. A function $m(\xi_{1},\xi_{2})$ on 
$ \mathbb{R}^n \times \mathbb{Z}^m$ is a multiplier on 
$L^p( \mathbb{R}^n \times \mathbb{T}^m)$ if 
\[
T f(x,y) =  \int_{\mathbb{R}^n} \sum_{\xi_{2} \in \mathbb{Z}^m}
m(\xi_{1},\xi_{2})\,\widehat{f}(\xi_{1},\xi_{2})\,e^{2 \pi i(\xi_{1}\cdot x+\xi_{2}\cdot y)} \, d\xi_{1}
\]
defines a bounded operator (the operator norm of $T$ will be denoted by 
$\|m\|_{{M}^p(\mathbb{R}^n \times \mathbb{T}^m)}$). If $m(\eta,\xi)$ is a continuous multiplier on 
$L^p(\mathbb{R}^{n+m}),$ then its restriction 
$m|_{\mathbb{R}^n \times \mathbb{Z}^m}$ is a multiplier on 
$L^p(\mathbb{R}^n \times \mathbb{T}^m )$ with
\[
\|m\|_{{M}^p(\mathbb{R}^n \times \mathbb{T}^m )} 
\leq \|m\|_{{M}^p(\mathbb{R}^{n+m})}.
\]
\end{lemma}

\begin{proposition}\cite[Theorem 5.5.1]{grafakos2008classical}\label{l2 valued extension}
Let $0 < p,q < \infty$ and let $(X,\mu)$ and $(Y,\nu)$ be two $\sigma$-finite measure
spaces. Suppose that $T$ is a bounded linear operator from $L^p(X)$ to $L^q(Y)$ with norm $\|T \|.$
Then $T$ has an $\ell^2$-valued extension, that is, for all complex-valued functions $f_j$ in
$L^p(X)$ we have
\[
\Big\| \Big( \sum_j |T(f_j)|^2 \Big)^{\tfrac{1}{2}} \Big\|_{L^q(Y)} 
\leq C_{p,q} \|T \| \, \Big\| \Big( \sum_j |f_j|^2 \Big)^{\tfrac{1}{2}} \Big\|_{L^p(X)}
\]
for some constant $C_{p,q}$ that satisfies $C_{p,q}=1$ if $p \leq q$.
\end{proposition}
   \section{The principal lemmas}
\subsection{Kernel Estimates}   
In this section we shall prove the following kernel estimates, which will play a crucial role in our analysis.  
\begin{proposition}\label{l8}
		For $\theta \geq 2.$ Then 
		\begin{equation*}
			\left| \sum_{\xi=-N}^{N} e^{2 \pi i (x \cdot \xi +t |\xi|^{\theta})} \right| \leq C_{\theta} |t|^{-\frac{1}{\theta}},
		\end{equation*}
		for any $(t,x) \in  [-N^{-(\theta-1)},N^{-(\theta-1)}] \times \mathbb{T},$ where $N>1,$ and $N \in \mathbb{Z}.$ 
        \end{proposition}
In order to  prove this proposition, we first need following lemmas.    
\begin{lemma}\label{gle} \
\begin{enumerate}
    \item \label{l1}  \cite[Proposition 2.6.7, p. 153]{grafakos2008classical} Let $f:(a,b) \to \mathbb{R}$ be a map, $k \in \mathbb{Z}^{+}$ and $|f^{(k)}(x)| \geq \rho >0$ for any $x \in (a,b)$ with $f'$ monotonic on $(a,b).$ Then 
		\begin{equation*}
			\left|\int_{a}^{b} e^{2 \pi i f(x)}dx \right| \leq c_{k} \rho^{-\frac{1}{k}},
		\end{equation*}
		where $c_{k}$ is independent of $a,b.$
        \item \label{l3} \cite[p. 198, Lemma 4.4]{zygmund2002trigonometric}  If $f:[a,b] \to \mathbb{R}$ is a map such that $f'$ (derivative of f) is monotone and $|f'| \leq \frac{1}{2}$ in $(a,b),$ then 
		\begin{equation*}
			\left|\int_{a}^{b} e^{2 \pi i f(x)}dx - \sum_{a \leq \xi \leq b} e^{2 \pi i f(\xi)} \right| \leq A,
		\end{equation*}
		where $A$ is absolute constant and $\xi \in \mathbb{Z}$.
\end{enumerate}
\end{lemma}

	\begin{lemma}\label{l4}
		Let $f:[a,b] \to \mathbb{R}$ be a map where $b-a >1$. Let $f'$ be strictly monotonic function on $(a,b)$ such that 
		\begin{equation*}
			\left|\int_{a}^{b} e^{2 \pi i (f(x)-px)}dx \right| \leq C_{\theta} \rho^{-\frac{1}{\theta}} \text{ for all } p \in \mathbb{Z},
		\end{equation*}
		for some $\rho >0 ,\theta \geq 2$ and $C_{\theta}$ is independent of $a,b,p.$ Then 
		\begin{equation*}
			\left| \sum_{a \leq \xi \leq b} e^{2 \pi i f(\xi)} \right| \leq [|f'(b)-f'(a)|+2][C_{\theta}\rho ^{-\frac{1}{\theta}}+A],
		\end{equation*} where $\xi \in \mathbb{Z}.$
		
	\end{lemma}
	\begin{proof} Without loss of generality we assume that $f'$ is strictly increasing function. Let $\alpha_{p} \in (a,b)$ be a point (if any) such that $$f'(\alpha_{p})=p-\frac{1}{2}$$
		and  $$F_{p}(x)=e^{2 \pi i(f(x)-px)}\  \text{ for all } \ p \in \mathbb{Z}. $$
Since $f'$ is strictly monotonic function, we have  
\[ |f'(x)-p| \leq \frac{1}{2} \quad \text{for all} \  x\in (\alpha_{p},\alpha_{p+1}). \ \]
Let $\alpha_{r},\alpha_{r+1},\alpha_{r+2},...,\alpha_{r+m}$ be the points, if any such exists, belonging to the interval $a \leq x \leq b.$ Using hypothesis and Lemma \ref{gle} \eqref{l3}, we have \\

\begin{align*}
    \left| \sum_{\alpha_{p} \leq \xi \leq \alpha_{p+1}} e^{2 \pi i f(\xi)} \right| &=\left| \sum_{\alpha_{p} \leq \xi \leq \alpha_{p+1}} e^{2 \pi i (f(\xi)-p \xi)}\right|  \\
    &\leq \left| \int_{\alpha_{p}}^{\alpha_{p+1}}  e^{2 \pi i (f(x)-px)} dx  \right|+ \left|\sum_{\alpha_{p} \leq \xi \leq \alpha_{p+1}} e^{2 \pi i (f(\xi)-p \xi)}-\int_{\alpha_{p}}^{\alpha_{p+1}}  e^{2 \pi i (f(x)-px)} dx \right|\\
    & \leq C_{\theta} \rho ^{-\frac{1}{\theta}}+A.
\end{align*} 
Similarly,
		$$\left| \sum_{a \leq \xi \leq \alpha_{r}} e^{2 \pi i f(\xi)} \right| \leq C_{\theta} \rho ^{-\frac{1}{\theta}}+A, $$
		and 
		$$\left| \sum_{\alpha_{r+m} \leq \xi \leq b} e^{2 \pi i f(\xi)} \right| \leq C_{\theta} \rho ^{-\frac{1}{\theta}}+A .$$
		Thus 
		$$\left| \sum_{a \leq \xi \leq b} e^{2 \pi i f(\xi)} \right| \leq (m+2) (C_{\theta} \rho ^{-\frac{1}{\theta}}+A) .$$
		Note that  $m+2=f'(\alpha_{r+m})-f'(\alpha_{r})+2.$
		So,
		$$\left| \sum_{a \leq \xi \leq b} e^{2 \pi i f(\xi)} \right| \leq (f'(\alpha_{r+m})-f'(\alpha_{r})+2) (C_{\theta} \rho ^{-\frac{1}{\theta}}+A) .$$
		Since $f'$ is strictly increasing, we get
		$$\left| \sum_{a \leq \xi \leq b} e^{2 \pi i f(\xi)} \right| \leq (f'(b)-f'(a)+2) (C_{\theta} \rho ^{-\frac{1}{\theta}}+A) .$$ 
	\end{proof}

	\begin{lemma}\label{l6}
		Let $f:(0,b) \to \mathbb{R}$ such that $f(x)=x \cdot \xi +t x^{\theta}$ be a function for any nonintegers $\theta \geq 2, \xi \in \mathbb{R}.$ If $b>1,$ then 
		\begin{equation*}
			\left|\int_{0}^{b} e^{2 \pi i (f(x)-px)}dx \right| \leq C_{\theta} |t|^{-\frac{1}{\theta}}, \text{ for all } t \in \mathbb{R} \setminus \{0\}, p \in \mathbb{Z}.
		\end{equation*}
	\end{lemma}
	
	\begin{proof} Let $|t|^{-\frac{1}{\theta}} >0$ be any number  and  $F(x)=f(x)-px.$  The idea is to apply Lemma \ref{gle}\eqref{l1}. To this end, we consider two cases  $x \in (0,|t|^{-\frac{1}{\theta}})$ and $x \in (|t|^{-\frac{1}{\theta}},b).$
For $x \in (0,|t|^{-\frac{1}{\theta}}),$ take  $k=[\theta]+1,$ where $[\cdot] $ is the greatest integer function. Then we have 
\begin{align*}
 |F^{(k)}(x)| &=|t||C_{1\theta}||x|^{\theta-[\theta]-1}\\
 &\geq |t||C_{1\theta}|(|t|^{-\frac{1}{\theta}})^{\theta-[\theta]-1}=|C_{1\theta}||t|^{\frac{[\theta]+1}{\theta}}=|C_{1\theta}||t|^{\frac{k}{\theta}}.   
\end{align*} 
By Lemma \ref{gle}\eqref{l1},we get 
		\begin{equation*}
			\left|\int_{0}^{|t|^{-\frac{1}{\theta}}} e^{2 \pi i F(x)}dx \right| \leq C_{\theta}^{'} (|C_{1\theta}||t|^{\frac{k}{\theta}})^{-\frac{1}{k}}=C_{1\theta}^{''} |t|^{-\frac{1}{\theta}}.
		\end{equation*}
For $x \in (|t|^{-\frac{1}{\theta}},b),$ take $k=[\theta].$ Then we have 
\begin{align*}
    |F^{(k)}(x)| & =|t||C_{1\theta}||x|^{\theta-[\theta]}\\
    & \geq |t||C_{1\theta}|(|t|^{-\frac{1}{\theta}})^{\theta-[\theta]}=|C_{1\theta}||t|^{\frac{[\theta]}{\theta}}=|C_{1\theta}||t|^{\frac{k}{\theta}}
\end{align*}
By Lemma \ref{gle} \ref{l1},we get 
		\begin{equation*}
			\left|\int_{|t|^{-\frac{1}{\theta}}}^{b} e^{2 \pi i F(x)}dx \right| \leq C_{2\theta}^{'} (|C_{2\theta}||t|^{\frac{k}{\theta}})^{-\frac{1}{k}}=C_{2\theta}^{''} |t|^{-\frac{1}{\theta}}.
		\end{equation*}
Combining the above inequalities, we get 
		\begin{equation*}
			\left|\int_{0}^{b} e^{2 \pi i F(x)}dx \right| \leq C_{1\theta}^{''} |t|^{-\frac{1}{\theta}} + C_{2\theta}^{''} |t|^{-\frac{1}{\theta}}=C_{\theta} |t|^{-\frac{1}{\theta}}.
		\end{equation*}
		where $C_{\theta}=C_{1\theta}^{''}+C_{2\theta}^{''}.$
	\end{proof}
	\begin{proof}[\textbf{Proof of Proposition \ref{l8}}]		
		Using triangle inequality, we get
		\begin{equation*}
			\left| \sum_{\xi=-N}^{N} e^{2 \pi i (x \cdot \xi +t |\xi|^{\theta})} \right| \leq 
			\left| \sum_{\xi=-N}^{0} e^{2 \pi i (x \cdot \xi +t |\xi|^{\theta})} \right|+\left|\sum_{\xi=1}^{N} e^{2 \pi i (x \cdot \xi +t |\xi|^{\theta})} \right|
		\end{equation*}
By Lemmas \ref{l6} and \ref{l4}, we get 
		\begin{equation*}
			\begin{split}
				\left| \sum_{\xi=-N}^{0} e^{2 \pi i (x \cdot \xi +t |\xi|^{\theta})} \right| &\leq [|x-x-\theta tN^{(\theta-1)}|+2][C_{\theta}^{'}|t|^{-\frac{1}{\theta}}+A] \\
				=& [\theta|t|N^{(\theta-1)}+2][C_{\theta}^{'}|t|^{-\frac{1}{\theta}}+A]
			\end{split}
		\end{equation*}
		Similarly,
		\begin{equation*}
			\begin{split}
				\left| \sum_{\xi=1}^{N} e^{2 \pi i (x \cdot \xi +t |\xi|^{\theta})} \right| &\leq [|x+\theta tN^{(\theta-1)}-x- t \theta|+2][C_{\theta}^{''}|t|^{-\frac{1}{\theta}}+A] \\
				=& [\theta|tN^{(\theta-1)}-t|+2][C_{\theta}^{''}|t|^{-\frac{1}{\theta}}+A]
			\end{split}
		\end{equation*}
		If $(t,x) \in [-N^{-(\theta-1)},N^{-(\theta-1)}] \times \mathbb{T},$ then 
		\begin{equation*}
			\left| \sum_{\xi=-N}^{N} e^{2 \pi i (x \cdot \xi +t |\xi|^{\theta})} \right| \lesssim_{\theta} |t|^{-\frac{1}{\theta}}.
		\end{equation*}		
	\end{proof}
    
\subsection{Decoupling inequality on $\mathbb R^d$}\label{decoupling sec on R} In this  subsection we shall state and prove refine decoupling inequality  $\mathbb R^d$ (see Theorem \ref{decoupling inequality for surface on waveguided} below).   We shall see that this will play a crucial role in proving  decoupling inequality on waveguide manifold (Lemma \ref{l2 decoupling wave}).

The decoupling inequality was introduced by Wolff \cite{wolff2000} and was further developed by numerous mathematicians. In the breakthrough works \cite{bourgain2015, bourgain2017}, Bourgain and Demeter proved sharp decoupling inequalities for $C^2$ surfaces in $\mathbb{R}^d$ with nonzero Gaussian curvature. For recent work on decoupling inequalities for smooth hypersurfaces with vanishing Gaussian curvature, see Biswas-Gilula-Li-Schwend-Xi \cite{biswas2020}, Demeter \cite[Section 12.6]{demeter2020}, Yang \cite{yang2021}, Li-Yang \cite{li2022,li2021}, and Guth-Maldague-Oh \cite{guth2024}.

To state our results, we  briefly set notations. For $c_1, \ldots, c_{d+1}>0$, let rectangular box $$\mathcal{R}_{C}=\left[-c_1, c_1\right] \times \ldots \times\left[-c_{d+1}, c_{d+1}\right].$$
For $(x,y,t)=(x_1,\ldots,x_n,y_{n+1},\ldots,y_{d},t) \in \mathbb{R}^{d+1},$
we will utilize two weight functions that are associated with $\mathcal{R}_{C}:$
\begin{eqnarray}\label{w1}
    \omega_{\mathcal{R}_{C}}(x,y,t)=\left(1+\sum_{j=1}^{n}\left|x_j\right| / c_j+ \sum_{j=n+1}^{d}\left|y_j\right| / c_j + {|t|}/{c_{d+1}}\right)^{-12 d}
\end{eqnarray}
and 
\begin{eqnarray}\label{w2}
    \tilde{\omega}_{\mathcal{R}_{C}}(x,y,t)=\left(1+\sum_{j=1}^{n}\left|x_j\right| / c_j+ \sum_{j=n+1}^{d}\left|y_j\right| / c_j + {|t|}/{c_{d+1}}\right)^{-10 d}.
\end{eqnarray}
Here $\omega_{\mathcal{R}_{C}}$ is adapted to the cube $\mathcal{R}_{C},$ i.e $\omega_{\mathcal{R}_{C}}(x,y,t) \sim 1$ for all $(x,y,t) \in \mathcal{R}_{C}$ and for $(x,y,t) \notin \mathcal{R}_{C},$ $\omega_{\mathcal{R}_{C}}(x,y,t)$ decays rapidly as the distance from $(x,y,t)$ to $\mathcal{R}_{C}$ increases. If $B$ is some ball of radius $R$ centered at $c,$ then we apply a suitable affine transformation (see Biswas-Gilula-Li-Schwend-Xi \cite{biswas2020}).
We use Bourgain-Demeter's decoupling \cite[Theorem 1.1]{bourgain2015} in the weighted version \cite[Proposition 9.15]{demeter2020} the paraboloid \eqref{paraboloid}

We are now ready to state our main result.

\begin{theorem}\label{decoupling inequality for surface on waveguided}
     Let $d=n+m \geqslant 1$, $Q \subset \mathbb{R}^d,$  and $\theta>1.$ We define the extension operator
\begin{equation}\label{E_Q}
    E_Q g(x,y,t)=\int_Q g(\xi_{1},\xi_{2}) e^{2 \pi i \left(x\cdot \xi_{1}+y\cdot \xi_{2}+t(|\xi_{1}|^{\theta}+|\xi_{2}|^\theta) \right)} \, d\xi_{1} \, d\xi_{2}.
\end{equation}
 Let $2 \leqslant p \leqslant \frac{2(d+2)}{d}$ and $\operatorname{Part}_{\delta^{1 / 2}}\mathcal{Q}_{d}$ denote a partition of $\mathcal{Q}_{d}:=[-1,1]^{d}$ into cubes of side length $\delta^{1 / 2}<1.$  Then for all $\varepsilon>0,$ we have
\begin{equation}\label{dc0}
    \left\|E_{\mathcal{Q}_{d}} g\right\|_{L^p\left(\omega_{R_{\delta}}\right)} \lesssim_{\varepsilon} \delta^{-\varepsilon}\left(\sum_{\Delta \in \operatorname{Part}_{\delta^{1 / 2}}\mathcal{Q}_{d}}\left\|E_{\Delta} g\right\|_{L^p\left(\tilde{\omega}_{R_{\delta}}\right)}^2\right)^{\frac{1}{2}},
\end{equation}
where $R_{\delta}\subset \mathbb R^{d+1}$ is a rectangular box of size $\delta^{-1} \times \cdots \times \delta^{-1} \times \delta^{-\max \{1, \theta / 2\}},$  \( \omega_{R_{\delta}} \) is a weight adapted to $R_\delta \subset \mathbb R^{d+1},$ and $\|f\|^p_{L_{t,x,y}^{p}(\vartheta)}=\int_{\mathbb{R}^{n+m+1}}\abs{f}^{p}\vartheta \, dx \, dy \, dt.$
\end{theorem} 

To the best of authors knowledge Theorem \ref{decoupling inequality for surface on waveguided} is new in the sense that the phase function \[\phi(\xi_{1},\xi_{2})=|\xi_{1}|^\theta+|\xi_{2}|^\theta\] in \eqref{E_Q}, which is of importance in our applications,   has not been covered in the literature so far. Our method of proof is inspired from \cite{biswas2020, wang2025strichartz}, where the authors proved versions of the decoupling theorem. However, our choice of $\phi$ raises some difficulties.  In fact, the aforementioned  $\phi$ is not a $C^2$ function on  the sets $\{0\} \times [-1,1]^m$ and $[-1,1]^n \times \{0\}$; on the other hand, phase function $\abs{\xi}^{\theta}$ considered in \cite{wang2025strichartz} is  smooth function away from the origin. We overcame this difficulty by decomposing $\mathcal{Q}_{d}$ in an appropriate manner.
\subsubsection{Comments on the phase function}\label{cpf}
 An important tool in  the decoupling arguments is the behavior of the Hessian of the phase function $\phi,$ which is defined on $\mathcal{Q}_{n+m}$ for $\theta>1.$  The  degeneracy of the phase function $\phi$ is determined by the behavior of the eigenvalues of the Hessians of $\abs{\xi_{1}}^{\theta}$ and $\abs{\xi_{2}}^{\theta}$ at a given point $(\xi_{1},\xi_{2})$. Note that the convexity of the phase changes with $\theta$, particularly relative to the value $\theta=1$.   The key idea to prove Theorem \ref{decoupling inequality for surface on waveguided} is to invoke  Bourgain-Demeter's decoupling theorem. However,  for $1< \theta <2,$ we note that   $\phi$ is not $C^2$ function on $\mathcal{Q}_{n+m},$ specifically on the sets $\{0\} \times [-1,1]^m$ and $[-1,1]^n \times \{0\}$. So, the  Hessian is not well-defined; and we cannot apply  Bourgain-Demeter's decoupling theorem directly.  \\

 If $\theta \geq 2$ then the  $\phi$ is $C^2$ function on $\mathcal{Q}_{n+m}$ and the Hessian of $\phi$ is block-diagonal matrix. The first block is an $n \times n$ matrix corresponding to the Hessian of the function $\xi_{1} \to \abs{\xi_{1}}^{\theta}$; while the second block is an $m \times m$ matrix corresponding to the Hessian of $\xi_{2} \to \abs{\xi_{2}}^{\theta}$.
Let $\phi''(\xi_{1},\xi_{2})$ denote the Hessian of the function $\phi$ at $(\xi_{1},\xi_{2})$ and  $H_{\xi_{1}}\phi(\xi)$ (resp.\ $H_{\xi_{2}}\phi(\xi)$), denote the Hessian of $\phi$ at $\xi$ in the $\xi_{1}$–variables (resp.\ $\xi_{2}$–variables). Then
       \[
  \phi''(\xi)
  = \begin{pmatrix}
      H_{\xi_{1}}\phi(\xi) & 0\\[6pt]
      0            & H_{\xi_{2}}\phi(\xi)
    \end{pmatrix},
\]
  where $\xi=(\xi_{1},\xi_{2}).$ Thus, the degeneracy of the phase function $\phi$ is determined by the behavior of the eigenvalues of the Hessians $H_{\xi_{1}}\phi(\xi)$ and  $H_{\xi_{2}}\phi(\xi)$ at a given point $(\xi_{1},\xi_{2}).$ \\
    In particular if $\theta=2,$ then $ \phi''(\xi)=2\mathcal{I}_{n+m},$ where $\mathcal{I}_{n+m}$ is identity matrix of order $n+m.$ And so it is positive definite for all $\xi.$ In this case we can directly apply Bourgain-Demeter's decoupling theorem.\\

When  $\theta>2,$ we have: 
    \begin{enumerate}
        \item If $(\xi_1,\xi_{2})=(0, 0)$, then $\phi''(0,0)=\mathcal{O}_{n+m},$ where $\mathcal{O}_{n+m}$ is zero matrix of order $(n+m).$  
        \item If $\xi_{2}=0$ and $\xi_{1} \neq 0,$ then $$\phi''(\xi_{1},0)= \begin{pmatrix}
      H_{\xi_{1}}\phi(\xi) & 0\\[6pt]
      0            &\mathcal{O}_{m}
    \end{pmatrix},$$ 
    where $H_{\xi_{1}}\phi(\xi)=\theta \abs{\xi_{1}}^{\theta-2}(\mathcal{I}_{n}+(\theta-2) \frac{\xi_{1} {\xi_{1}}^{T}}{\abs{\xi_{1}}^2})$ and so $\phi$ is degenerate. In this case we cannot apply Bourgain-Demeter's decoupling theorem.
    \item If $\xi_{1}=0$ and $\xi_{2} \neq 0,$ then $$\phi''(0,\xi_{2})= \begin{pmatrix}
      \mathcal{O}_{n} & 0\\[6pt]
      0            & H_{\xi_{2}}\phi(\xi)
    \end{pmatrix},$$ 
    where $H_{\xi_{2}}\phi(\xi)=\theta \abs{\xi_{2}}^{\theta-2}(\mathcal{I}_{m}+(\theta-2) \frac{\xi_{2} {\xi_{2}}^{T}}{\abs{\xi_{2}}^2})$ and so $\phi$ is degenerate. In this case we cannot apply Bourgain-Demeter's decoupling theorem.
    \item If $\xi_{1} \neq 0$ and $\xi_{2} \neq 0,$ then $$ \phi''(\xi)
  = \begin{pmatrix}
      H_{\xi_{1}}\phi(\xi) & 0\\[6pt]
      0            & H_{\xi_{2}}\phi(\xi)
    \end{pmatrix},$$ 
    where $H_{\xi_{1}}\phi(\xi)=\theta \abs{\xi_{1}}^{\theta-2}(\mathcal{I}_{n}+(\theta-2) \frac{\xi_{1} {\xi_{1}}^{T}}{\abs{\xi_{1}}^2})$, $H_{\xi_{2}}\phi(\xi)=\theta \abs{\xi_{2}}^{\theta-2}(\mathcal{I}_{m}+(\theta-2) \frac{\xi_{2} {\xi_{2}}^{T}}{\abs{\xi_{2}}^2})$ and so $\phi$ is non-degenerate. In this case we can apply Bourgain-Demeter's decoupling theorem.
    \end{enumerate}

\subsubsection{Proof of Theorem \ref{decoupling inequality for surface on waveguided}}

\begin{remark}[proof strategy] The following steps are in order:
\begin{itemize}
    \item We decompose  $\mathcal{Q}_d$ dyadically into four parts (see Figure \ref{fig:3}). This reduce the analysis to deal the extension operator $E_{I}$ (see \eqref{E_Q}) on each parts $I's.$ Part $\mathrm{I}$ can be easily control by Minkowski and H\"older inequalities. The proof on  Parts $\mathrm{II}$ and $\mathrm{III}$ are similar. So we shall only  prove part $\mathrm{III}.$ For part $\mathrm{III},$ it enough to prove inequality \eqref{dc1}.
    \item  To prove inequality \eqref{dc1}, by covering the box $R_{\delta}$ of size $\delta^{-1} \times \cdots \times \delta^{-1} \times \delta^{-\max \{1, \theta / 2\}}$ with box \(R_{a,\delta}\) of size \(\delta^{-1} \times \cdots \times \delta^{-1} \times a^{2-\theta}\delta^{-1}\), it is enough to prove \eqref{dc2}.  
    \item To prove \eqref{dc2}, by scaling in the time variable of the extension operator $E_{Q_{d}},$ it is enough to prove \eqref{dc4}.
Then we used a piece of paraboloid to locally approximate the surface by using Taylor's expansion and then apply Bourgain-Demeter's decoupling theorem (See \cite[Theorem 1.1]{bourgain2015}).
\end{itemize}
\end{remark}

   \begin{theorem}[Bourgain-Demeter's decoupling theorem, see Theorem 1.1 in \cite{bourgain2015} and Proposition 9.15 in \cite{demeter2020}] \label{BD theorem}
       Let $S$ be a compact $C^2$ hypersurface in $\mathbb{R}^{d+1}$ with positive definite second fundamental form. If $\operatorname{supp}(\hat{f}) \subset \mathcal{N}_\delta$, then for $2 \leq p \leq \frac{2(d+2)}{d}$ and $\varepsilon>0$,
$$
\|f\|_{L^{p}(\omega_{\mathcal{R}_{C}})} \lesssim_{\varepsilon} \delta^{-\varepsilon}\left(\sum_{\Delta \in \operatorname{Part}_{\delta^{1 / 2} }\mathcal{N}_\delta}\left\|f_\Delta\right\|_{L^{p}(\omega_{\mathcal{R}_{C}})}^2\right)^{1 / 2},
$$
where $f_{\Delta}$ denote the Fourier restriction of $f$ to $\Delta$ and $\mathcal{N}_\delta$ is $\delta$ neighbourhood of the truncated paraboloid $\{(\xi, \abs{\xi}^2) \in \mathbb{R}^{d+1}: \xi=(\xi'_{1}, \cdots, \xi'_{d}) \in \mathbb{R}^{d}, \abs{\xi'_{i}} \leq 1/2 \}.$ 
   \end{theorem}

\begin{proof}[\textbf{Proof of Theorem \ref{decoupling inequality for surface on waveguided}}] 
\begin{figure}[h]
\centering
\begin{tikzpicture}[scale=0.5]
  \draw[thick] (-4,-4) rectangle (4,4);
\draw[thick] (-1,-1) rectangle (1,1);
\draw[] (1,-1) rectangle (2,1);
\draw[] (3,-1) rectangle (4,1);
\draw[] (-2,-1) rectangle (-1,1);
\draw[] (-4,-1) rectangle (-3,1);
\draw[] (-1,1) rectangle (1,2);
\draw[] (-1,3) rectangle (1,4);
\draw[] (-1,-2) rectangle (1,-1);
\draw[] (-1,-4) rectangle (1,-3);
\draw[dashed] (2.1,0.5) -- (2.9,0.5);
\draw[dashed] (-2.1,0.5) -- (-2.9,0.5);
\draw[dashed] (0.5,2.1) -- (0.5,2.9);
\draw[dashed] (0.5,-2.1) -- (0.5,-2.9);
  \draw[->] (-5.5,0) -- (5.5,0) node[right] {$\xi_{1}$};
  \draw[->] (0,-5.5) -- (0,5.5) node[above] {$\xi_{2}$};

  \node at (0.5,0.5) {$\mathrm{I}$};
\node at (1.5,0.5) {$\mathrm{II}$};
\node at (0.5,1.5) {$\mathrm{III}$};
\node at (2.5,2.5) {$\mathrm{IV}$};
  \fill (0,0) circle (2pt) node[below left] {$O$};
  \fill (4,0) circle (2pt) node[below right] {$1$};
\fill (-4,0) circle (2pt) node[below left] {$-1$};
\fill (0,4) circle (2pt) node[above left] {$1$};
\fill (0,-4) circle (2pt) node[below left] {$-1$};
\fill[red, opacity=0.3] (-1,-1) rectangle (1,1);
\fill[green, opacity=0.3] (3,-1) rectangle (4,1);
\fill[green, opacity=0.3] (-2,-1) rectangle (-1,1);
\fill[green, opacity=0.3] (-4,-1) rectangle (-3,1);
\fill[green, opacity=0.3]  (1,-1) rectangle (2,1);
\fill[green, opacity=0.3]  (2,-1) rectangle (3,1);
\fill[green, opacity=0.3] (-3,-1) rectangle (-2,1);
\fill[yellow, opacity=0.3] (-1,3) rectangle (1,4);
\fill[yellow, opacity=0.3] (-1,-2) rectangle (1,-1);
\fill[yellow, opacity=0.3] (-1,-4) rectangle (1,-3);
\fill[yellow, opacity=0.3] (-1,1) rectangle (1,2);
\fill[yellow, opacity=0.3] (-1,2) rectangle (1,3);
\fill[yellow, opacity=0.3] (-1,-3) rectangle (1,-2);
\end{tikzpicture}
\caption{Case $n=m=1$: red = $\mathrm{I}$, green = $\mathrm{II}$, yellow = $\mathrm{III}$, white = $\mathrm{IV}$.}
\label{fig:3}
\end{figure}
  Let us denote  
\[\mathrm{I}= \left\{(\xi_{1}, \xi_{2}) \in \mathcal{Q}_{n+m}:|\xi_{1}| \leq \delta^{\frac{1}{2}-\varepsilon},|\xi_{2}| \leq \delta^{\frac{1}{2}-\varepsilon}\right\},\]
\[\mathrm{II}= \bigcup_{k=1}^{K_1}\left\{(\xi_{1}, \xi_{2}) \in \mathcal{Q}_{n+m}:|\xi_{2}| \leq \delta^{\frac{1}{2}-\varepsilon}, 2^{k-1} \delta^{\frac{1}{2}-\varepsilon} \leq|\xi_{1}| \leq 2^k \delta^{\frac{1}{2}-\varepsilon}\right\},\]
\[ \mathrm{III}= \bigcup_{k=1}^{K_2}\left\{(\xi_{1}, \xi_{2}) \in \mathcal{Q}_{n+m}:|\xi_{1}| \leq \delta^{\frac{1}{2}-\varepsilon}, 2^{k-1} \delta^{\frac{1}{2}-\varepsilon} \leq|\xi_{2}| \leq 2^k \delta^{\frac{1}{2}-\varepsilon}\right\},\]
and 
\[\mathrm{IV}=\bigcup_{(k_{1},k_{2})=(1,1)}^{(K_1,K_{2})}\left\{(\xi_{1}, \xi_{2}) \in \mathcal{Q}_{n+m}:2^{k_{1}-1} \delta^{\frac{1}{2}-\varepsilon} \leq|\xi_{1}| \leq 2^{k_{1}} \delta^{\frac{1}{2}-\varepsilon}, 2^{k_{2}-1} \delta^{\frac{1}{2}-\varepsilon} \leq|\xi_{2}| \leq 2^{k_{2}} \delta^{\frac{1}{2}-\varepsilon}\right\}.\]
Let $K_{1} \approx \log(\delta^{-1})$ and $K_{2} \approx \log(\delta^{-1}).$ Then we can 
decompose the cube \(\mathcal{Q}_{d}\) dyadically as  
 \[\mathcal{Q}_{d}= \mathrm{I} \cup \mathrm{II} \cup \mathrm{III} \cup \mathrm{IV}.\]
By Minkowski and H\"older inequalities, we get 
\begin{align*}
    \left\|E_{\mathrm{I}} g\right\|_{L^p\left(\omega_{R_{\delta}}\right)}
    &\leq \sum_{\Delta \in \operatorname{Part}_{\delta^{1 / 2}}\mathrm{I}}\left\|E_{\Delta} g\right\|_{L^p\left(\tilde{\omega}_{R_{\delta}}\right)} \\
    &\lesssim_{\varepsilon} \delta^{-\varepsilon}\left(\sum_{\Delta \in \operatorname{Part}_{\delta^{1 / 2}}\mathrm{I}}\left\|E_{\Delta} g\right\|_{L^p\left(\tilde{\omega}_{R_{\delta}}\right)}^2\right)^{\frac{1}{2}}.
\end{align*}
For any $\delta^{\frac{1}{2}-\varepsilon} \leq a \leq \frac{1}{2},$ denote the annulus
\[A_{a}=\{\xi_{2} \in \mathcal{Q}_m: a \leq |\xi_{2}| \leq 2a\}.\]
Let $x \in \mathbb{R}^n$ be fixed. To prove \eqref{dc0} for Part $\mathrm{III}$, it is enough to prove that for any $\Delta_{\xi_{1}} \in \operatorname{Part}_{\delta^{\frac{1}{2}}}(\abs{\xi_{1}}\leq \delta^{\frac{1}{2}-\varepsilon}),$  we have  
\begin{align}\label{dc1}
&\norm{E_{\Delta_{\xi_{1}} \times A_a} g(x)}_{L^p_{t,y}(\omega_{R_\delta}(x))} \lesssim_{\varepsilon} \delta^{-\varepsilon}
\left( \sum_{\Delta_{\xi_{2}} \in \operatorname{Part}_{\delta^{\frac{1}{2}}}(A_a)} \norm{E_{\Delta_{\xi_{1}} \times \Delta_{\xi_{2}}} g(x)}_{L^p_{t,y}(\tilde{\omega}_{R_\delta}(x))}^2 \right)^{\frac{1}{2}}.
\end{align}
Let \( I_{k}= \left\{ \xi_{2} \in \mathcal{Q}_{m}: 2^{k-1} \delta^{\frac{1}{2}-\varepsilon} \leq|\xi_{2}| \leq 2^k \delta^{\frac{1}{2}-\varepsilon}\right\},\) for $k=1, \cdots,K_{2}.$
In fact for Part $\mathrm{III},$ by Minkowiski inequality
\begin{align*}
  \norm{E_{\mathrm{III}} g}_{L^p(\omega_{R_\delta})} & \leq \sum_{k=1}^{K_{2}} \sum_{\Delta_{\xi_{1}}} \norm{E_{\Delta_{\xi_{1}} \times I_{k}} g}_{L^p(\omega_{R_\delta})} \\
  &\lesssim_{\varepsilon} \delta^{-\varepsilon} \sum_{k=1}^{K_{2}} \sum_{\Delta_{\xi_{1}}}  \norm{ \left( \sum_{\Delta_{\xi_{2}} \in \operatorname{Part}_{\delta^{\frac{1}{2}}}(I_{k})} \norm{E_{\Delta_{\xi_{1}} \times \Delta_{\xi_{2}}} g(x)}_{L^p_{t,y}(\tilde{\omega}_{R_\delta}(x))}^2 \right)^{\frac{1}{2}}}_{L_{x}^p(\mathbb{R}^n)} \\
  &\lesssim_{\varepsilon} \delta^{-\varepsilon} \sum_{k=1}^{K_{2}} \sum_{\Delta_{\xi_{1}}}\left( \sum_{\Delta_{\xi_{2}} \in \operatorname{Part}_{\delta^{\frac{1}{2}}}(I_k)}  \norm{  E_{\Delta_{\xi_{1}} \times \Delta_{\xi_{2}}} g }_{L^p(\tilde{\omega}_{R_\delta})} ^2 \right)^{\frac{1}{2}} \\
  &\lesssim_{\varepsilon} \delta^{-\frac{3\varepsilon}{2}} (K_{2})^{\frac{1}{2}} \left( \sum_{k=1}^{K_{2}} \sum_{\Delta_{\xi_{1}}}\sum_{\Delta_{\xi_{2}} \in \operatorname{Part}_{\delta^{\frac{1}{2}}}(I_k)}  \norm{  E_{\Delta_{\xi_{1}} \times \Delta_{\xi_{2}}} g }_{L^p(\tilde{\omega}_{R_\delta})} ^2 \right)^{\frac{1}{2}} \\
  &\lesssim_{\varepsilon} \delta^{-{2\varepsilon}} \left( \sum_{\Delta_{\xi_{1}} \times \Delta_{\xi_{2} \in \operatorname{Part}_{\delta^{\frac{1}{2}}}(\mathrm{III}})}  \norm{  E_{\Delta_{\xi_{1}} \times \Delta_{\xi_{2}}} g }_{L^p(\tilde{\omega}_{R_\delta})} ^2 \right)^{\frac{1}{2}},
\end{align*}
where second inequality follows from equation \eqref{dc1}. The third inequality follows from the Minkowski integral inequality, while the fourth inequality results from H\"older's inequality. The fifth inequality is justified by the fact that $K_{2} \approx \log(\delta^{-1}),$ i.e $K_{2} \lesssim_{\varepsilon}\delta^{-\varepsilon}.$

Specifically, first we shall prove \eqref{dc2} implies \eqref{dc1}.  Our initial claim is that, for any \(\delta^{\frac{1}{2}-\varepsilon} \leq a \leq \frac{1}{2}\), we have 
\begin{align}\label{dc2}
\norm{E_{\Delta_{\xi_{1}} \times A_a} g(x)}_{L^p_{t,y}(\omega_{R_{a,\delta}}(x))} 
&\lesssim_{\varepsilon} \delta^{-\varepsilon}
\left( \sum_{\Delta_{\xi_{2}} \in \operatorname{Part}_{\delta^{\frac{1}{2}}}(A_a)} \norm{E_{\Delta_{\xi_{1}} \times \Delta_{\xi_{2}}} g(x)}_{L^p_{t,y}(\omega_{R_{a,\delta}}(x))}^2 \right)^{\frac{1}{2}},
\end{align}
where \(R_{a,\delta}\) is a rectangular box of size \(\delta^{-1} \times \cdots \times \delta^{-1} \times a^{2-\theta}\delta^{-1}.\) Since \(\delta^{-\max\{1, \theta/2\}} \geq a^{2-\theta}\delta^{-1}\), we can split \(R_\delta\) into copies of \(R_{a,\delta}\) and by \cite[Lemma 4]{wang2025strichartz},
\begin{equation}\label{dc3}
\omega_{R_\delta}(x,y,t) \lesssim_{n,m} \sum_j \omega_{R_{a,\delta}(j)}(x,y,t) \omega_{R_\delta}(j) \lesssim_{n,m} \tilde{\omega}_{R_{\delta}}(x,y,t),
\end{equation}
where each \(R_{a,\delta}(j)\) is a copy of \(R_{a,\delta}\) centered at \(j \in \delta^{-1}\mathbb{Z}^{n+m} \times a^{2-\theta}\delta^{-1}\mathbb{Z}.\) The inequality \eqref{dc3} is a weighted covering inequality that addresses the mismatch between the rectangles \( R_{a,\delta} \) and \( R_{\delta} \). Consequently, the expression \( \sum_j \omega_{R_{a,\delta}(j)}(x,y,t) \omega_{R_\delta}(j) \approx \omega_{R_\delta}(x,y,t) \) cannot hold. To resolve this issue, we utilize slightly different rectangular weight functions in \eqref{dc3}.
By  \eqref{dc3}, \eqref{dc2} and Minkowiski inequality for $p \geq 2,$ we have
\begin{align*}
 & \norm{E_{\Delta_{\xi_{1}} \times A_a} g(x)}_{L^p_{t,y}(\omega_{R_\delta}(x))}\\
 & \lesssim_{n,m}  \left( \sum_{j} \omega_{R_\delta}(j) \norm{E_{\Delta_{\xi_{1}} \times A_a} g(x)}_{L^p_{t,y}(\omega_{R_{a,\delta}(j)}(x))}^{p} \right)^{\frac{1}{p}} \\
  & \lesssim_{\varepsilon} \delta^{-\varepsilon} \left( \sum_{j} \omega_{R_\delta}(j) \left( \sum_{\Delta_{\xi_{2}} \in \operatorname{Part}_{\delta^{\frac{1}{2}}}(A_a)} \norm{E_{\Delta_{\xi_{1}} \times \Delta_{\xi_{2}}} g(x)}_{L^p_{t,y}(\omega_{R_{a,\delta}(j)}(x))}^2 \right)^{\frac{p}{2}} \right)^{\frac{1}{p}} \\
  &\lesssim_{\varepsilon} \delta^{-\varepsilon}
\left( \sum_{\Delta_{\xi_{2}} \in \operatorname{Part}_{\delta^{\frac{1}{2}}}(A_a)} \norm{E_{\Delta_{\xi_{1}} \times \Delta_{\xi_{2}}} g(x)}_{L^p_{t,y}(\tilde{\omega}_{R_\delta}(x))}^2 \right)^{\frac{1}{2}}.
\end{align*}
 This proves \eqref{dc2} implies \eqref{dc1}. 
Now we will prove \eqref{dc2}. Let $\phi(\xi_{2})=|\xi_{2}|^\theta$ and for
$\xi'_{2}\in A_a,$ the Taylor expansion gives
\[
  a^{2-\theta}\,\phi(\xi'_{2}+h)
  = a^{2-\theta}\Bigl(\phi(\xi'_{2}) + \phi'(\xi'_{2})\,h 
    + \tfrac12\,h^{T}\,\phi''(\xi'_{2})\,h\Bigr)
    + O\bigl(a^{-1}\lvert h\rvert^3\bigr),
\]
where 
$$\phi''(\xi'_{2})=\theta \abs{\xi'_{2}}^{\theta-2}(\mathcal{I}_{m}+(\theta-2) \frac{\xi'_{2}{\xi'_{2}}^{T}}{\abs{\xi'_{2}}^2}),$$
here $\mathcal{I}_{m}$ is identity matrix of order $m$ and ${\xi'_{2}}^{T}$ is transpose of column matrix $\xi'_{2}.$
Since $\theta |\xi'_2|^{\theta - 2}$ (with multiplicity $m - 1$) and $\theta(\theta - 1) |\xi'_2|^{\theta - 2}$ are the eigenvalues of the matrix $\phi''(\xi'_2)$, the matrix $\phi''(\xi'_2)$ is positive definite when $\theta > 1$.
So,  for $\theta>1,$ we have  $a^{2-\theta} \abs{h^{T}\,\phi''(\xi'_{2})\,h} \approx \abs{h}^2.$ 

When $\abs{h} \leq \delta^{\frac{1}{2}- \sigma}$ with $\sigma=\frac{\varepsilon}{3},$ the error term $O\bigl(a^{-1}\lvert h\rvert^3\bigr)=O(\delta),$ since $a \geq \delta^{\frac{1}{2}-\varepsilon}.$ Thus if $\abs{\xi_{2}-\xi'_{2}} \leq \delta^{\frac{1}{2}- \sigma},$ the surface $S_{a}(\xi_2)=(\xi_{2},a^{2-\theta} \abs{\xi_{2}}^{\theta})$ is in the $\delta$-neighborhood of the paraboloid
\begin{equation}\label{paraboloid}
    \tilde{S}_{a}(\xi_{2})=\left(\xi_{2},a^{2-\theta}\Bigl(\phi(\xi'_{2}) + \phi'(\xi'_{2})\,(\xi_{2}-\xi'_{2}) 
    + \tfrac12\,(\xi_{2}-\xi'_{2})^{T}\,\phi''(\xi'_{2})\,(\xi_{2}-\xi'_{2})\Bigr)\right).
\end{equation}
Now we rescale the time variable. Let 
\begin{equation*}
    E_{\Delta_{\xi_{1}} \times \Delta_{\xi_{2}}, S_{a}} g(x,y,t)= \int_{\Delta_{\xi_{2}}}\left( \int_{\Delta_{\xi_{1}}} g(\xi_{1},\xi_{2}) e^{2\pi i (x \cdot \xi_{1}+a^{2-\theta} t \abs{\xi_{1}}^{\theta})} \, d\xi_{1} \right) e^{2\pi i (y \cdot \xi_{2}+a^{2-\theta} t \abs{\xi_{2}}^{\theta})} \, d\xi_{2}.
\end{equation*}
Then
$$E_{\Delta_{\xi_{1}} \times \Delta_{\xi_{2}}} g(x,y,t)=E_{\Delta_{\xi_{1}} \times \Delta_{\xi_{2}}, S_{a}} g(x,y, a^{\theta-2}t)$$
and 
$$\norm{E_{\Delta_{\xi_{1}} \times \Delta_{\xi_{2}}} g(x)}_{L^{p}_{t,y}(R_{a,\delta})}=a^{\frac{2-\theta}{p}} \norm{E_{\Delta_{\xi_{1}} \times \Delta_{\xi_{2}}, S_{a}} g(x)}_{L^{p}_{t,y}(Q_{\delta})},$$
where $Q_{\delta}$ is a cube of side length $\delta^{-1}.$
In view of the scaling, to prove \eqref{dc2},  it suffices  to show 
\begin{align}
 \norm{E_{\Delta_{\xi_{1}} \times A_a, S_a} g(x)}_{L^p_{t,y}(\omega_{Q_{\delta}}(x))} 
&\lesssim_{\varepsilon} \delta^{-\varepsilon}
\left( \sum_{\Delta_{\xi_{2}} \in \operatorname{Part}_{\delta^{\frac{1}{2}}}(A_a)} \norm{E_{\Delta_{\xi_{1}} \times \Delta_{\xi_{2}}, S_a} g(x)}_{L^p_{t,y}(\omega_{Q_{\delta}}(x))}^2 \right)^{\frac{1}{2}}.
\end{align}
It is sufficient to demonstrate that the smallest constant \( K_{p}(\delta) \) that satisfies the following inequality fulfills the condition \( K_{p}(\delta) \lesssim_{\varepsilon} \delta^{-\varepsilon} ,\)
\begin{align}\label{dc4}
 \norm{E_{\Delta_{\xi_{1}} \times A_a, S_a} g(x)}_{L^p_{t,y}(\omega_{Q_{\delta}}(x))} 
&\lesssim  K_{p}(\delta)
\left( \sum_{\Delta_{\xi_{2}} \in \operatorname{Part}_{\delta^{\frac{1}{2}}}(A_a)} \norm{E_{\Delta_{\xi_{1}} \times \Delta_{\xi_{2}}, S_a} g(x)}_{L^p_{t,y}(\omega_{Q_{\delta}}(x))}^2 \right)^{\frac{1}{2}}.
\end{align}
Note that $$\sum_{j} \omega_{Q_{\delta^{1-2\sigma}}(j)}(x,y,t) \omega_{Q_{\delta}}(j) \approx \omega_{Q_{\delta}}(x,y,t),$$
where each $Q_{\delta^{1-2\sigma}}(j)$ is a copy of $Q_{\delta^{1-2\sigma}}$ centered at $j \in \delta^{-1+2\sigma} \mathbb{Z}^{n+m+1},$ and the implicit constants only depend on $m,n.$ By \eqref{dc4} and Minkowski inequality, we get 
\begin{align*}
    \norm{E_{\Delta_{\xi_{1}} \times A_a, S_a} g(x)}_{L^p_{t,y}(\omega_{Q_{\delta}}(x))} \leq C K_{p}(\delta^{1-2\sigma})\left( \sum_{\tau_{\xi_{2}} \in \operatorname{Part}_{\delta^{\frac{1}{2}-\sigma}}(A_a)} \norm{E_{\Delta_{\xi_{1}} \times \tau_{\xi_{2}}, S_a} g(x)}_{L^p_{t,y}(\omega_{Q_{\delta}}(x))}^2 \right)^{\frac{1}{2}}.
\end{align*}
Applying Theorem \ref{BD theorem}, for $2\leq p \leq \frac{2(m+2)}{m},$ we get 
\begin{equation*}
    \norm{E_{\Delta_{\xi_{1}} \times \tau_{\xi_{2}}, S_a} g(x)}_{L^p_{t,y}(\omega_{Q_{\delta}}(x))}  \leq D_{p}(\delta) \left( \sum_{\Delta_{\xi_{2}} \in \operatorname{Part}_{\delta^{\frac{1}{2}}}(A_a), \Delta_{\xi_{2}} \subset \tau_{\xi_{2}}} \norm{E_{\Delta_{\xi_{1}} \times \Delta_{\xi_{2}}, S_a} g(x)}_{L^p_{t,y}(\omega_{Q_{\delta}}(x))}^2 \right)^{\frac{1}{2}},
\end{equation*}
where the decoupling constant $D_{p}(\delta) \leq C_{\varepsilon}\delta^{-\frac{\varepsilon^{2}}{2}}$ for all $\varepsilon >0.$
Thus $$
K_p(\delta) \leqslant C D_p(\delta) K_p\left(\delta^{1-2 \sigma}\right) .
$$

Recall $\sigma=\varepsilon / 3$. We iterate to get
$$
\begin{aligned}
K_p(\delta) & \leqslant C^k D_p(\delta) D_p\left(\delta^{1-2 \sigma}\right) \ldots D_p\left(\delta^{(1-2 \sigma)^{k-1}}\right) K_p\left(\delta^{(1-2 \sigma)^k}\right) \\
& \leqslant C^k C_{\varepsilon}^k \delta^{-\frac{1}{2} \varepsilon^2\left(1+(1-2 \sigma)+\ldots+(1-2 \sigma)^{k-1}\right)} K_p\left(\delta^{(1-2 \sigma)^k}\right) \\
& =C^k C_{\varepsilon}^k \delta^{-\frac{3}{4} \varepsilon\left(1-(1-2 \sigma)^k\right)} K_p\left(\delta^{(1-2 \sigma)^k}\right)
\end{aligned}
$$
Recall that $\delta^{1 / 2-\varepsilon} \leqslant a \leqslant 1 / 2$. We choose $k$ such that $\delta^{(1-2 \sigma)^k} \approx a^2 \leqslant 1 / 4$, then $K_p\left(\delta^{(1-2 \sigma)^k}\right) \approx 1$ and $k \lesssim \log \log \left(\delta^{-1}\right)$. Thus (for detailed calculations, see \cite{biswas2020})
$$
K_p(\delta) \lesssim C^k C_{\varepsilon}^k \delta^{-\frac{3}{4} \varepsilon\left(1-(1-2 \sigma)^k\right)} \lesssim_{\varepsilon} \delta^{-\varepsilon}.
$$
Now we will prove for forth part. Since 
$$\resizebox{\textwidth}{!}{$\mathrm{IV}=\bigcup_{(k_{1},k_{2})=(1,1)}^{(K_1,K_{2})}\left\{(\xi_{1}, \xi_{2}) \in \mathcal{Q}_{n+m}:2^{k_{1}-1} \delta^{\frac{1}{2}-\varepsilon} \leq|\xi_{1}| \leq 2^{k_{1}} \delta^{\frac{1}{2}-\varepsilon}, 2^{k_{2}-1} \delta^{\frac{1}{2}-\varepsilon} \leq|\xi_{2}| \leq 2^{k_{2}} \delta^{\frac{1}{2}-\varepsilon}\right\}.$}$$
Again we will decompose the cube \(\mathcal{Q}_{d}\) dyadically into 
$$\resizebox{\textwidth}{!}{$\mathcal{Q}_{d}= \left\{\xi \in \mathcal{Q}_{d}:|\xi| \leq \delta^{\frac{1}{2}-\varepsilon}\right\} \cup \left(\bigcup_{k=1}^{K}\left\{\xi \in \mathcal{Q}_{d}: 2^{k-1} \delta^{\frac{1}{2}-\varepsilon} \leq|\xi | \leq 2^k \delta^{\frac{1}{2}-\varepsilon}\right\}\right).$}$$
Here $K \approx \log(\delta^{-1}).$ It is enough to prove that for any $\delta^{\frac{1}{2}-\varepsilon} \leq a \leq \frac{1}{2}$ and the set 
$$B_{a}=\{\xi \in \mathcal{Q}_{d}: a \leq |\xi| \leq 2a\} \cap \mathrm{IV},$$
we have  
\begin{equation}\label{dc7}
    \norm{E_{B_{a}} g}_{L^p_{t,x,y}(\omega_{R_\delta})} \lesssim_{\varepsilon} \delta^{-\varepsilon}
\left( \sum_{\Delta \in \operatorname{Part}_{\delta^{\frac{1}{2}}}(B_a)} \norm{E_{\Delta} g}_{L^p_{t,x,y}(\tilde{\omega}_{R_\delta})}^2 \right)^{\frac{1}{2}}.
\end{equation}
Let $\xi=(\xi_{1},\xi_{2})\in\mathbb{R}^{n+m}$ and $|\xi| = \sqrt{|\xi_{1}|^2 + |\xi_{2}|^2}.$
Let $\theta>1$ and define
\[
  S=\left\{\left(\xi_{1},\xi_{2},|\xi_{1}|^{\theta}+|\xi_{2}|^\theta\right) \in \mathbb{R}^{n} \times \mathbb{R}^m \times \mathbb{R}: (\xi_{1},\xi_{2}) \in\mathcal{Q}_{n+m}\right\}.
\]

Let $\phi(\xi)=\phi(\xi_{1},\xi_{2})=|\xi_{1}|^\theta+|\xi_{2}|^\theta.$
For $\xi'=(\xi'_{1},\xi'_{2}) \in B_a$, the Taylor expansion gives
\[
  a^{2-\theta}\,\phi(\xi'+h')
  = a^{2-\theta}\Bigl(\phi(\xi') + \phi'(\xi')\,h' 
    + \tfrac12\,{h'}^{T}\,\phi''(\xi')\,h'\Bigr)
    + O\bigl(a^{-1}\lvert h'\rvert^3\bigr),
\]
where
\[
  h' = \begin{pmatrix}h'_1\\[6pt]h'_2\end{pmatrix}
  \in\mathbb{R}^{n+m}, 
  \quad
  h'_1\in\mathbb{R}^n,\;
  h'_2\in\mathbb{R}^m,
\]
and
\[
  \phi''(\xi')
  = \begin{pmatrix}
      H_{\xi_{1}}\phi(\xi') & 0\\[6pt]
      0            & H_{\xi_{2}}\phi(\xi')
    \end{pmatrix},
\]
where $H_{\xi_{1}}\phi(\xi')$ (resp.\ $H_{\xi_{2}}\phi(\xi')$), the Hessian of $\phi$ at $\xi'$ in the $\xi_{1}$–variables (resp.\ $\xi_{2}$–variables)  with 
$$H_{\xi_{1}}\phi(\xi')=\theta \abs{\xi'_{1}}^{\theta-2}(\mathcal{I}_{n}+(\theta-2) \frac{\xi'_{1} {\xi'_{1}}^{T}}{\abs{\xi'_{1}}^2}),$$
and 
$$H_{\xi_{2}}\phi(\xi')=\theta \abs{\xi'_{2}}^{\theta-2}(\mathcal{I}_{m}+(\theta-2) \frac{\xi'_{2} {\xi'_{2}}^{T}}{\abs{\xi'_{2}}^2}).$$
Hence
\[
  {h'}^{T}\,\phi''(\xi')\,h'
  = {h'_1}^{T}\,H_x\phi(\xi')\,h'_1
  + {h'_2}^{T}\,H_y\phi(\xi')\,h'_2.
\]
If $\theta>1,$ then
\[
  a^{2-\theta}\bigl\lvert {h'}^T\phi''(\xi')\,h'\bigr\rvert
  \;\approx\;|h'_1|^2 + |h'_2|^2 \;=\; |h'|^2.
\]
And $a^{2-\theta} \abs{{h'}^{T}\,\phi''(\xi')\,h'} \approx \abs{h'}^2,$ whenever $\theta>1.$ When $\abs{h'} \leq \delta^{\frac{1}{2}- \sigma}$ with $\sigma=\frac{\varepsilon}{3},$ the error term $O\bigl(a^{-1}\lvert h'\rvert^3\bigr)=O(\delta),$ since $a \geq \delta^{\frac{1}{2}-\varepsilon}.$ Thus if $\abs{\xi-\xi'} \leq \delta^{\frac{1}{2}- \sigma},$ the surface $S_{a}=(\xi,a^{2-\theta} (\abs{\xi_{1}}^{\theta}+\abs{\xi_{2}}^{\theta}))$ is in the $\delta$-neighbourhood of the paraboloid
\begin{equation*}
    \tilde{S}_{a}(\xi)=\left(\xi,a^{2-\theta}\Bigl(\phi(\xi') + \phi'(\xi')\,(\xi-\xi') 
    + \tfrac12\,(\xi-\xi')^{T}\,\phi''(\xi')\,(\xi-\xi')\Bigr)\right).
\end{equation*}
Now the proof of the inequality \eqref{dc7} for \(2 \leq p \leq \frac{2(d+2)}{d},\) follows the same method as outlined in \cite[Lemma 4]{wang2025strichartz}.
\end{proof}
\subsection{Decoupling  inequality  on waveguide manifold}\label{decoupling sec on wave}
\begin{lemma}[Decoupling-type lemma]\label{l2 decoupling wave}
Suppose \( g, g_{l} \) are Schwartz functions on \(\mathbb{R}^n \times \mathbb{T}^m\) with \( g = P_{\leq N} g \) such that
\[
g(x, y) = \sum_{\substack{l \in \mathbb{Z}^m \\ |l| \leq N}} \int_{[-N, N]^n} \widehat{g}_{l}(\xi_{1}) e^{2\pi i (x \cdot \xi_{1} + y \cdot l)} \, d\xi_{1},
\]
where $\widehat{g}_{l}(\xi_{1})=\widehat{g}(\xi_{1},l),$ for all $(\xi_{1},l) \in \mathbb{R}^n \times \mathbb{Z}^m.$
Cover \([-N, N]^n\) by finitely overlapping cubes \( Q_k \) of side length \(\sim 1\). Let \(\{\phi_k\}\) be a partition of unity adapted to the \(\{Q_k\}\), and denote the component of $g$ adapted to $Q_k \times \left \{\nu \right\}$ by 
\[g_{\Theta_{\nu,k}} = e^{2\pi i y \cdot \nu} \mathcal{F}^{-1}_x (\hat{g}(\cdot, \nu) \phi_k), \quad (\nu \in \mathbb Z^m, \ |\nu|\leq N).\]
 Also \(2 \leq p \leq p^* = \frac{2(n+m+2)}{n+m} .\) Then, for any time interval $I$, when $1 < \theta < 2$, the interval $I$ has length $\sim N^{2-\theta}$, and for $\theta \ge 2$, the interval $I$ has length $\sim 1$.

\[
\|e^{i t {(-\Delta)}^{\frac{\theta}{2}}} g\|_{L^p(I \times \mathbb{R}^n \times \mathbb{T}^m )} \lesssim_\varepsilon N^\varepsilon \left( \sum_{\nu,k} \|e^{i t {(-\Delta)}^{\frac{\theta}{2}}} g_{\Theta_{\nu,k}} w_I \|^2_{L^p(\mathbb{R} \times \mathbb{R}^n \times  \mathbb{T}^m )} \right)^\frac{1}{2},
\]
where \( w_I \) is a bump function adapted to \( I \) and \( \theta >1.\)
\end{lemma}

We note that the decoupling inequality on waveguide manifold was first proved by  Sire-Yu-Yue-Zhao in  \cite[Lemma 4.2]{sire2023}, where the phase function  $\phi(\xi_1, \xi_2)=|\xi_1|^{\theta}+|\xi_2|^2$  or more generally nondegenerate type phase function  away from the origin.
And so their proof  could apply Bourgain-Demeter's decoupling theorem and employed the strategy used by  Schippa  in \cite{schippa2020}.
It is worth noting that in Lemma \ref{l2 decoupling wave}  the phase function  exhibits degeneracy on $(\{0\} \times [-1,1]^m) \cup ([-1,1]^n \times\{0\})$ and so we cannot apply  Bourgain-Demeter's decoupling theorem directly. See Subsection \ref{cpf}. To overcome this difficulties, we  have  had  established Theorem \ref{decoupling inequality for surface on waveguided}.

\begin{remark}[proof strategy] The following steps are in order:
\begin{itemize}
    \item First, we will reduce the problem to establishing an $L^p$-norm estimate for  
    \[ u(t, x, y) = e^{it(-\Delta)^{\frac{\theta}{2}}}g(x, y) \]  
    over \( I \times B_l \times \mathbb{T}^m \), by proving a key inequality \eqref{Bl}, through decomposing \( g \) into pieces \( g_{\Theta_{\nu,k}} \), where \( \Theta_{\nu,k} \) is a box of \( (n+m) \)-dimensions in a tiling of \( [-1,1]^{n+m} \).

    \item By rescaling \( u \) to adjust the frequency support and defining an extension operator \( Ef \), we adjust the norm via a change of variables and the periodicity in the \( y \)-variable, which gives a relation between the \( L^p \)-norms of \( u \) and \( Ef \).

    \item Then, we define an auxiliary operator \( \widetilde{E} \) with a smoothed function \( f^\delta \) to bound the \( L^p \)-norm of \( Ef \). Using the Lebesgue differentiation theorem and Fatou’s lemma, we obtain the inequality \eqref{discrete to continuous}.

    \item Next, we tile the domain with boxes \( \Theta_{\nu,k} \) and apply a decoupling inequality for the hypersurface i.e., Theorem \ref{decoupling inequality for surface on waveguided}, to bound the \( L^p \)-norm of \( \widetilde{E}f \). Then, specializing to \( f^\delta \) and taking a limit, we finally rescale and derive a uniform pointwise estimate for \( u^\delta \) to complete the \( L^p \)-norm estimate.
\end{itemize}    
\end{remark}
  \begin{proof}[\textbf{Proof of Lemma \ref{l2 decoupling wave}}]
    Let \( B_l \subset \mathbb{R}^n \) be a fixed ball of radius \( N \), and let \( w_l \) be a smooth weight adapted to \(I \times B_l  \). To prove Lemma \ref{l2 decoupling wave}, it will suffice to show that
\begin{equation}\label{Bl}
    \|e^{it{(-\Delta)}^{\frac{\theta}{2}}} g\|_{L^p( I \times B_l \times \mathbb{T}^m)} \lesssim_\varepsilon N^{\varepsilon} \left( \sum_{\nu,k} \|e^{it{(-\Delta)}^{\frac{\theta}{2}}} g_{\Theta_{\nu,k}} \|_{L^p(w_l)}^2 \right)^{\frac{1}{2}}.
\end{equation}

Then to prove the full estimate on \( \mathbb{R}^n \times \mathbb{T}^m \), we can choose a finitely overlapping collection of balls \( B_l \) of radius \( N \) that cover \( \mathbb{R}^n \), and then apply \eqref{Bl} in each \( B_l \) and use Minkowski's inequality to sum:
\begin{align*}
\|e^{it{(-\Delta)}^{\frac{\theta}{2}}} g\|_{L^p( I \times \mathbb{R}^n \times \mathbb{T}^m)}^p &\leq \sum_l \|e^{it{(-\Delta)}^{\frac{\theta}{2}}} g\|_{L^p( I \times B_l \times \mathbb{T}^m)}^p \\
&\lesssim_\varepsilon N^{\varepsilon} \sum_l \left( \sum_{\nu,k} \|e^{it{(-\Delta)}^{\frac{\theta}{2}}} g_{\Theta_{\nu,k}} \|_{L^p(w_l)}^2 \right)^{\frac{p}{2}} \\
&\lesssim_\varepsilon N^{\varepsilon} \left( \sum_{\nu,k} \|e^{it{(-\Delta)}^{\frac{\theta}{2}}} g_{\Theta_{\nu,k}} \|_{L^p(w)}^2 \right)^{\frac{p}{2}},
\end{align*}
where the \( w_l \) are weights adapted to \(I \times B_l  \) and \( w = \sum_l w_l \). Now if \( \vartheta \) is a smooth weight on \( \mathbb{R} \) that rapidly decays outside \(I\) it is easy to check that \( w \leq C \vartheta \), so Lemma \ref{l2 decoupling wave} will follow.  
Denote  \[ u(t, x, y) = e^{it{(-\Delta)}^{\frac{\theta}{2}}} g(x, y).\] We may rescale  \( u_0=u(t=0,x,y)\) to have frequency support in \([-1, 1]^{n+m}\). 
In fact, we have 
\begin{equation}\label{u function}
\resizebox{0.9\textwidth}{!}{$
 u(N^{-\theta}t, N^{-1}x, N^{-1}y) = N^n \int_{[-1,1]^{n}} \sum_{\nu \in N^{-1}\mathbb{Z}^m \cap [-1,1]^{m}} \hat{g}(N\xi, N\nu) e^{2 \pi i  (x\cdot \xi + y\cdot \nu + t(|\xi|^{\theta} + |\nu|^{\theta}))} d\xi.   
$}
\end{equation}
Denote
 \[ f(\xi, \nu) = \hat{g}(N\xi, N\nu)  \quad \text{and} \quad  \Lambda_N = N^{-1}\mathbb{Z}^m \cap [-1, 1]^m.\]  Then  \( f \) is supported on \( [-1,1]^{n} \times \Lambda_N \). Since  $P_{\leq N}g=g,$ \( f \) can be viewed as a function on \( n -\)dimensional cubes of size \(\sim 1\) that are \(\sim N^{-1}\)-separated in \( [-1,1]^{n+m} \). 
Let \( Ef \) denote the extension operator
\[
Ef(t,x,y) = \sum_{\nu \in \Lambda_N} \int_{[-1,1]^{n}} f(\xi, \nu) e^{2 \pi i  (x\cdot \xi + y\cdot \nu + t(|\xi|^{\theta} + |\nu|^{\theta}))} d\xi.
\]
After applying a change of variables on the spatial side and using periodicity in the \( y \)-variable in \eqref{u function}, we get
\begin{align*}
\|u\|_{L^p(I \times B_N \times \mathbb{T}^m )} &= N^n N^{-\frac{n+m+\theta}{p}} \|Ef\|_{L^p( N^{\theta}I \times B_{N^{2}} \times N\mathbb{T}^m)} \\
&= N^{n(1-\frac{1}{p})-\frac{1}{p}(\theta+2m)} \|Ef\|_{L^p( N^{\theta}I \times B_{N^{2}} \times [0,N^{2}]^{m} )}.
\end{align*}
We may introduce the continuous extension \(\widetilde{E}\)  of the operator $E$. Specifically,  the operator \(\widetilde{E}\) on \(C^\infty([-1, 1]^{n+m})\) is defined by
\[
\widetilde{E}f(t, x, y) = \int_{[-1,1]^{n}} \int_{[-1,1]^{m}} f(\xi_1, \xi_2) e^{2 \pi i (x\cdot \xi_1 + y\cdot \xi_2 + t(|\xi_1|^{\theta} + |\xi_2|^{\theta}))} d\xi_1 d\xi_2.
\]
We shall see that    $\widetilde{E}$ gives the transition from the waveguide to  usual extension operator associated to the fractional dispersion in  $\mathbb R^{n+m+1}$.
 Given a function \(f\) on \([-1, 1]^n \times \Lambda_N\), let
\[
f^{\delta}(\xi_1, \xi_2) = \sum_{\nu \in \Lambda_N} c_m \delta^{-m} \mathbf{1}_{\{|\xi_2 - \nu| \leq \delta\}} f(\xi_1, \nu), \quad \delta < \frac{1}{N},
\]
where \(c_m\) is a dimensional constant chosen for normalization. Then by Fatou lemma, we have
$$ \int_{N^{\theta}I} \int_{B_{N^2}} \int_{N^{2}\mathbb{T}^m}\liminf_{\delta \to 0} \abs{\widetilde{E}f^{\delta}}^{p} \, dx \, dy \, dt \leq \liminf_{\delta \to 0} \int_{N^{\theta}I} \int_{B_{N^2}} \int_{N^{2}\mathbb{T}^m} \abs{\widetilde{E}f^{\delta}}^{p} \, dx \, dy \, dt ,$$
and by Lebesgue differentiation theorem, we get
\[ \lim_{\delta \to 0} \widetilde{E}f^{\delta} (t,x, y)= Ef (t,x, y).\]
Thus we have
\begin{equation}\label{discrete to continuous}
    \|Ef\|_{L^p(N^{\theta}I \times B_{N^2} \times N^{2}\mathbb{T}^m )} \leq \liminf_{\delta \to 0} \|\widetilde{E}f^{\delta}\|_{L^p(N^{\theta}I \times B_{N^2} \times [0, N^{2}]^m )},
\end{equation}
where  \(N^{2}\mathbb{T}^m=[0, N^{2}]^m\). We will begin by estimating \(\widetilde{E}f\) for arbitrary \(f\) on \([-1,1]^{n+m}\) before specializing to \(f^{\delta}\) and passing to the limit later in the argument.\\
Choose a finitely overlapping collection of \((n+m)\)-dimensional boxes \(\Theta_{\nu,k} = I_k \times I_\nu \subset [-1,1]^{n} \times [-1,1]^{m}\) with the following properties:
\begin{itemize}
    \item[(i)] Each \(\Theta_{\nu,k}\) has side lengths \(\sim N^{-1}\).
    \item[(ii)] \(I_\nu\) is centered at \(\nu \in \Lambda_N\).
    \item[(iii)] \(\nu\) varies over an \(\sim N^{-1}\)-separated subset of \(\Lambda_N\) such that the cubes \(I_\nu\) cover \([-1, 1]^m\).
\end{itemize}

Such a collection of \(\Theta_{\nu,k}\) yields a finitely overlapping tiling of \([-1,1]^{n+m}\) by boxes of side length \(\sim N^{-1}.\) Let \(\{\varphi_{\nu,k}\}\) be a smooth partition of unity subordinate to this cover, and let \(f_{\nu,k} = f \cdot \varphi_{\nu,k}\). By covering the set $N^{\theta}I \times B_{N^2} \times [0, N^{2}]^m$ with finitely many cubes $R_{N}$ of size $N^{\max{\{2,\theta\}}} \times {N^2} \times \cdots \times {N^2},$ and applying Minkowski's inequality along with the decoupling inequality for the hypersurface, by Theorem \ref{decoupling inequality for surface on waveguided} for $2 \leq p \leq \frac{2(d+2)}{d},$ we get
\[
\|\widetilde{E}f\|_{L^p(N^{\theta}I \times B_{N^2} \times [0, N^{2}]^m )} \lesssim_\varepsilon N^{\varepsilon} \left( \sum_{\nu,k} \|\widetilde{E}f_{\nu,k}\|_{L^p(w_{R_{N}})}^2 \right)^{1/2},
\]
where \(w_{R_{N}}\) is a bump function adapted to \(N^{\theta}I \times B_{N^2} \times [0, N^{2}]^m \), with the \(f_{\nu,k}\) supported on finitely overlapping boxes of side lengths \(\sim 1/N\). Specializing to \(f^{\delta}\) and taking a limit as in \eqref{discrete to continuous}, we get
\[
\|Ef\|_{L^{p}(N^{\theta}I \times B_{N^2} \times N^{2}\mathbb{T}^m )} \lesssim_\varepsilon \liminf_{\delta \to 0} N^{\varepsilon} \left( \sum_{\nu,k} \|\widetilde{E}f_{\nu,k}^{\delta}\|_{L^{p}(w_{R_{N}})}^2 \right)^{1/2}.
\]
Then as a consequence of the scaling \eqref{Bl}, we get
\[
\|u\|_{L^{p}(I \times B_N \times \mathbb{T}^m )} \lesssim_\varepsilon \liminf_{\delta \to 0} N^{n(1-\frac{1}{p})-\frac{1}{p}(\theta+2m) + \varepsilon} \left( \sum_{\nu,k} \|\widetilde{E}f_{\nu,k}^{\delta}\|_{L^{p}(w_{R_{N}})}^2 \right)^{1/2}.
\]
where as above \( f(\xi, \nu) = \hat{g}(N \xi, N\nu) \) for \( \nu \in \Lambda_N \). Now by rescaling as before and writing out the definition of \( f^\delta \) we get
\[
N^{n(1-\frac{1}{p})-\frac{1}{p}(\theta+2m)} \|\widetilde{E}f_{\nu,k}^\delta \|_{L^{p}(w_{R_{N}})} \approx N^{-\frac{m}{p}} \| u_\delta \|_{L^{p}(\vartheta)},
\]
where
\[
u_\delta = \sum_{\substack{l \in \mathbb{Z}^m \\ |l-\nu| \lesssim 1}} \int_{\mathbb{R}^n} \int_{\mathbb{R}^m} c(N\delta)^{-m} \mathbf{1}_{\{|\xi_2-l| \leq N\delta\}} f_{\nu,k}(N^{-1}\xi_1, N^{-1}l)e^{2\pi i(x\cdot \xi_1 + y\cdot \xi_2 + t(|\xi_1|^{\theta} + |\xi_2|^{\theta}))} d\xi_2 \, d\xi_1
\]
and \( \vartheta \) is a suitable bump function adapted to \(I \times B_N \times [0, N]^m \). The desired result now follows from the pointwise estimate
\[
|u_\delta| \leq \sum_{\substack{I \in \mathbb{Z}^m \\ |l-\nu| \lesssim 1}} \left| \int_{\mathbb{R}^n} f_{\nu,k}(N^{-1}\xi_1, N^{-1}l)e^{2\pi i(x\cdot \xi_1 + t|\xi_1|^{\theta})} d\xi_1 \right|,
\]
which is uniform in \( \delta \).
\end{proof}

\section{Proof of main results}
\subsection{Proof of Theorem~\ref{ose} and of Theorem~\ref{Strichartz} for the case $\theta = 2$}

	\begin{proposition}\label{oseP}
		Let $(\frac{1}{q},\frac{1}{p}) \in [A,B]$ in Figure \ref{fig:2} and $\theta>2.$ Then for any $N > 1,$ any $\lambda \in \ell^{\frac{2q}{q+1}}(\mathbb Z), $ and any ONS $(a_{j})_{j}$ in $\ell^{2}(\mathbb Z),$ 
		\begin{equation}\label{PR1}
			\left\| \sum_{j} \lambda_{j} |\mathcal{E}_{N}a_{j} |^{2} \right\|_{L_{t}^{p} (I_N, L_{x}^{q}(\mathbb{T}))} \leq C_{\theta} \| \lambda \|_{\ell^{\alpha'}},
		\end{equation}
		where $I_{N}=[\frac{-1}{2N^{\theta-1}},\frac{1}{2N^{\theta-1}}],$ $\mathcal{E}_{N}$ is defined in \eqref{FEP} by taking $\eta(\xi)=\mathbf{1}_{S_{d,1}}(\xi)$ and $\alpha'=\frac{2 q}{q+1}.$
	\end{proposition}
\begin{remark}[proof strategy]
    We will not give direct proof of Proposition \ref{PR1}, instead we will prove the dual of the problem.
\begin{itemize}
    \item First, we will dualize Proposition \ref{PR1}, using Lemma \ref{PL1}.
    \item  Then we will prove the dualized inequality for \(\alpha = \infty\).
    \item  To get the inequality between \(\alpha = \frac{2(\theta + 1)}{\theta}\) and \(\alpha = \infty\), we will use Lemma \ref{PL5}.
    \item Then we will prove the dualized inequality for \(\frac{\theta}{2(\theta + 1)} \leq \frac{1}{\alpha} \leq \frac{1}{2}\) using Lemma \ref{s-interpolation} for \(\Sp^2\)-norm and \(\Sp^\infty\)-norm.
\end{itemize}
\end{remark}
	\begin{proof}[Proof of Proposition \ref{oseP}]  For any $x \in \M=\mathbb{T},$ we define $$\mathbf{1}_{I_{N}}(t,x)=\begin{cases}
       1 & if \quad  t \in  I_{N} \\
      0  & if \quad t \notin  I_{N} 
  \end{cases}.$$
Put $\cH= L^2(I, L^2(\mathbb T)).$	By duality principle Lemma \ref{PL1}, to prove  the desired estimate \eqref{PR1} for all $(\frac{1}{q},\frac{1}{p}) \in [A,B],$ it suffices to show that 
		\begin{equation}\label{PR2}
			\|W_{1}\mathbf{1}_{I_{N}}\mathcal{E}_{N} \mathcal{E}_{N}^{*}[\mathbf{1}_{I_{N}}W_{2}]\|_{\Sp^{2q'} } \leq C_{\theta} \|W_{1} \|_{L_{t}^{2p'}L_{x}^{2q'}}
			\|W_{2} \|_{L_{t}^{2p'}L_{x}^{2q'}}.
		\end{equation}
        Put $\alpha =2q'$ and $\beta=2p'.$  We may rewrite \eqref{PR2} as follows:
		\begin{equation}\label{PR3}
			\|W_{1}\mathbf{1}_{I_{N}}\mathcal{E}_{N} \mathcal{E}_{N}^{*}[\mathbf{1}_{I_{N}}W_{2}]\|_{\Sp^\alpha } \leq C_{\theta} \|W_{1} \|_{L_{t}^{\beta}L_{x}^{\alpha}}
			\|W_{2} \|_{L_{t}^{\beta}L_{x}^{\alpha}},
		\end{equation} 
        for all $\alpha,\beta \geq 1$ such that $\frac{2}{\beta}+\frac{2}{\theta} \cdot \frac{1}{\alpha}=1$ and $0 \leq \frac{1}{\alpha} \leq \frac{1}{2},$ i.e. $2\leq \alpha \leq \infty.$
In view of interpolation theorems in order to prove  \eqref{PR3} for $\alpha \in [2, \infty],$ it is sufficient to prove $\alpha \in [2,  \frac{2(\theta +1)}{\theta}].$ In fact, once this is done,  we may invoke Lemma \ref{PL5}  (as \eqref{PR3} and \eqref{PR1} are equivalent) to get the range of $\alpha$ in    $ [ \frac{2(\theta +1)}{\theta}, \infty ].$ On the other end, we invoke Stein interpolation to get the range of $\alpha$ in $ [2,  \frac{2(\theta +1)}{\theta}].$

We shall first prove \eqref{PR3} for $\alpha = \infty.$ In fact,  in this case, we have 
		\begin{align*}
			\|W_{1}\mathbf{1}_{I_{N}}\mathcal{E}_{N} \mathcal{E}_{N}^{*}[\mathbf{1}_{I_{N}}W_{2}]\|_{\Sp^{\infty}} &=\|W_{1}\mathbf{1}_{I_{N}}\mathcal{E}_{N} \mathcal{E}_{N}^{*}[\mathbf{1}_{I_{N}}W_{2}]\|_{ \cH \to \cH } \\ &=\|W_{1}\mathbf{1}_{I_{N}}\mathcal{E}_{N} (\overline{[\mathbf{1}_{I_{N}}W_{2}]}\mathcal{E}_{N})^{*}\|_{\cH\to \cH} \\
			&\leq \|W_{1}\mathbf{1}_{I_{N}}\mathcal{E}_{N}  \|_{\cH \to \cH} \|\overline{[\mathbf{1}_{I_{N}}W_{2}]}\mathcal{E}_{N}\|_{\cH\to \cH}.
		\end{align*}
Now, let  $\|f \|_{L^{2}(\hM)} \leq 1.$ In view \eqref{d1}, we have 
		\begin{align*}
			\|W_{1}\mathbf{1}_{I_{N}}\mathcal{E}_{N} f\|_{\cH } ^{2}&=\int_{I} \int_{\M} |W_{1}(t,x) \mathbf{1}_{I_{N}}(t,x) \mathcal{E}_{N}f(t,x) |^{2} dx dt \\
			&\leq \int_{I_N} \|W_{1}(t)\|_{L_{x}^{\infty}(\M)}^{2} \int_{\M} |\mathcal{E}_{N}f(t,x) |^{2} dx dt \\
			&\leq \int_{I} \|W_{1}(t)\|_{L_{x}^{\infty}(\M)}^{2} \|\mathcal{E}_{N}f(t) \|_{L_{x}^{2}(\M)}^{2}  dt\\
            & \leq  \|W_{1} \|_{L^{2}(I,L^{\infty}(\M))}
		\end{align*}
	Similarly, $$\|\overline{[\mathbf{1}_{I_{N}} W_{2}]}\mathbf{1}_{I_{N}}\mathcal{E}_{N} \|_{\cH\to \cH} \leq \|W_{2} \|_{L^{2}(I,L^{\infty}(\M))}$$
    Combining the above inequalities, we have 
		$$\|W_{1}\mathbf{1}_{I_{N}}\mathcal{E}_{N} \mathcal{E}_{N}^{*}[\mathbf{1}_{I_{N}}W_{2}]\|_{\Sp^{\infty}} \leq \|W_{1} \|_{L^{2}(I,L^{\infty}(\M))} \|W_{2} \|_{L^{2}(I,L^{\infty}(\M))}.$$
	This proves \eqref{PR3} for  $\alpha=\infty.$ Now we shall prove  \eqref{PR3} for $\alpha=2.$  To this end, for $\epsilon > 0,$ we define a map $T_{N,\epsilon}:\cH  \to \cH $ by
		$$F \mapsto T_{N,\epsilon}F(t,x)=\int_{I \times \M} K_{N,\epsilon} (t - t', x - x') F( t', x') \, dx' dt',$$
		where  
        \begin{align*}
        K_{N,\epsilon}(t,x)& =\mathbf{1}_{\epsilon < |t| <N^{-(\theta-1)}} (t,x) K_{N}(t,x)\\
        &=\mathbf{1}_{\epsilon < |t| <N^{-(\theta-1)}} (t,x)\sum_{\xi=-N}^{N} e^{2 \pi i (x \cdot \xi+ t|\xi|^{\theta})}.    
        \end{align*}
Further, for $z \in \mathbb{C},$ with $\operatorname{Re}(z) \in [-1,\frac{1}{\theta}],$ we define
\[T_{N,\epsilon}^{z}=K_{N,\epsilon}^{z}*\]
where 
		$K_{N,\epsilon}^{z}(t,x)=t^{z}K_{N,\epsilon}(t,x).$
    Since $\Sp^2-$norm is the Hilbert-Schmidt norm, we  first find the kernel of the operator $ W_1 \mathbf{1}_{I_{N}} T^z_{N,\epsilon} (\mathbf{1}_{I_{N}} W_2)$. In fact, for   $f \in \cH $, we have 
 \[
\begin{aligned}
& W_1 \mathbf{1}_{I_{N}} T^z_{N,\epsilon} (\mathbf{1}_{I_{N}} W_2) f(t,x) \\
&= W_1(t,x)\mathbf{1}_{I_{N}}(t,x) K^z_{N,\epsilon} * (\mathbf{1}_{I_{N}} W_2 f)(t,x) \\
&= W_1(t,x) \mathbf{1}_{I_{N}}(t,x) \int_{I_N \times \M} K^z_{N,\epsilon} (t-t', x-x') W_2(t', x') f(t',x') \, dx' \, dt' \\
&= \int_{ I_N \times \M} W_1(t,x) \mathbf{1}_{I_{N}}(t,x)  K^z_{N,\epsilon} (t-t', x-x') W_2(t', x') f(t', x') \, dx' \, dt' \\
&= \int_{I_N \times \M} K_1(t,x, t',x') f(t', x') \, dx' \, dt',
\end{aligned}
\]
where kernel  $ K_1(t,x,t',x') = W_1(t,x) \mathbf{1}_{I_{N}}(t,x) K^z_{N,\epsilon} (t-t',x-x') W_2(x',t').$ By Proposition \ref{l8}, we have 
  \begin{flalign}\label{ds1}
			&\|W_{1}\mathbf{1}_{I_{N}}T_{N,\epsilon}^{z}[\mathbf{1}_{I_{N}}W_{2}]\|_{\Sp^{2}}^{2}  =\left\|K_1\right\|_{\cH \times \cH}^{2} \nonumber \\
            &=\int_{ I_{N} \times \M} \int_{ I_{N} \times \M} |W_{1}(t,x)K_{N,\epsilon}^{z}(t-t',x-x')W_{2}(t', x')|^{2} dx dt dx' dt' \nonumber \\
            &  \lesssim_{\theta} \int_{I_{N} \times \M} \int_{ I_{N} \times \M} |W_{1}(t,x)|^{2}|t-t'|^{2\operatorname{Re}(z)-\frac{2}{\theta}}|W_{2}(t', x')|^{2} dx dt dx' dt' \nonumber \\
            & = \int_{I} \int_{I} \frac{\|W_{1}(t)\|_{L_{x}^{2}(\M)}^{2}\|W_{2}(t')\|_{L_{x}^{2}(\M)}^{2}}{|t-t'|^{\frac{2}{\theta}-2\operatorname{Re}(z)}} dt dt'.
            \end{flalign} 
In order to invoke 1D Hardy–Littlewood–Sobolev inequality that  we imposes  the following the condition
 \[2\operatorname{Re}(z)-\frac{2}{\theta} \in (-1,0]\quad \text{and} \quad \frac{2}{\widetilde{u}}+\bigg(\frac{2}{\theta}-2\operatorname{Re}(z)\bigg)=2.\] 
 Taking  $2\widetilde{u}=u,$ we get 
\begin{eqnarray}\label{c1}
    \frac{1}{u}=\frac{1}{2}+\frac{1}{2}(\operatorname{Re}(z)-\frac{1}{\theta}) \in \left(\frac{1}{4},\frac{1}{2}\right] .
\end{eqnarray}
 Now, for the above ranges of $u$ and $\operatorname{Re}(z),$ it follows that 
\begin{flalign}\label{PR4}
\|W_{1}\mathbf{1}_{I_{N}}T_{N,\epsilon}^{z}[\mathbf{1}_{I_{N}}W_{2}]\|_{\Sp^{2}}^{2} & \lesssim_{\theta} \left\| \|W_{1}\|_{L_{x}^{2}(\M)}^{2} \right\|_{L_{t}^{\widetilde{u}}(I)} \left\| \|W_{2}\|_{L_{x}^{2}(\M)}^{2} \right\|_{L_{t}^{\widetilde{u}}(I)} \nonumber\\
            & = C'_{\theta} \|W_{1} \|^2_{L_{t}^{u}(I,L_{x}^{2}(\M))} \|W_{2} \|^2_{L_{t}^{u}(I,L_{x}^{2}(\M))}.
            \end{flalign} 
Next we claim that for $\operatorname{Re}(z)=-1, \quad T_{N,\epsilon}^{z}: \cH \to \cH,$ holds with some constant depending only on $\theta \text{ and } \text{Im}(z)$ exponentially. In fact, for each $t \in I,$ by Plancherel's theorem, we have 
		\begin{flalign*}
			&\|T_{N,\epsilon}^{z}F(\cdot,t)\|_{L_{x}^{2}(\M)}^{2}=\sum_{m \in \mathbb{Z}}|\widehat{T_{N,\epsilon}^{z}F}(t, m)|^{2} \\
			&=\sum_{m \in \mathbb{Z}}|\widehat{K_{N,\epsilon}^{z}*F}(t, m)|^{2} \\
			&=\sum_{m=-N}^{N}\left|\int_{\epsilon< |t'|<N^{-(\theta-1)}} t'^{-1+i \text{Im}(z)} e^{-2 \pi i (t-t')|m|^{\theta}}\widehat{F(t-t', \cdot)}(m) dt' \right|^{2}\\
			&=\sum_{m=-N}^{N}\left|\int_{\epsilon< |t'|<N^{-(\theta-1)}} t'^{-1+i \text{Im}(z)} G_{m}(t-t') dt' \right|^{2} \\
			&=\sum_{m=-N}^{N}|H_{N,\epsilon}^{z}G_{m}(t)|^{2},
		\end{flalign*}
		where $G_{m}(t)=e^{-2 \pi i t|m|^{\theta}}\widehat{F(t, \cdot)}(m)$ and define 
    \[H_{N,\epsilon}^{z}G_{m}(t)=\int_{\epsilon< |t'|<N^{-(\theta-1)}} t'^{-1+i \text{Im}(z)} G_{m}(t-t') dt'.\]
In fact, we may  rewrite $H_{N,\epsilon}^{z}$ as a Hilbert transform of $h$  up to $i\operatorname{Im}(z):$		
        \begin{equation*}\label{Hilbert transform}
			 H_{N,\epsilon}^{z}G_{m}(t)=\int_{\mathbb{R}}\frac{h(t-t')}{t'^{1-i \text{Im}(z)}} dt',
		\end{equation*}
		where
		\begin{equation*}
			h\left( t\right) =\begin{cases}  G_{m}(t) & if \quad\epsilon  <|t|  <N^{-\left( \theta-1\right) },\\
				0 & otherwise
                \end{cases}.
		\end{equation*}
So it follows (see e.g. \cite{vega1992restriction}) that  $H_{N,\epsilon}^{z}G_{m}$ is a bounded operator on  $L^{2}(\mathbb{R})$  with the operator norm depends only on $\text{Im}(z)$ exponentially. As a consequence, we get 
\begin{equation*}
		 \|T_{N,\epsilon}^{z}F\|_{\cH}^{2}=\sum_{m=-N}^{N} \|H_{N,\epsilon}^{z}G_{m} \|_{L_{t}^{2}(I_{N})}^2.
		\end{equation*}
    Hence, $T_{N,\epsilon}^{z}:\cH \to \cH$ is bounded operator with operator norm depends on $\text{Im}(z)$ exponentially: 
    \begin{eqnarray}\label{eb}
        \|T_{N,\epsilon}^{z}\|_{\cH \to \cH} \leq C (\text{Im}(z)).
    \end{eqnarray}
 		Let $f \in \cH,$ and $\|f\|_{\cH} \leq 1,$ then
		\begin{align*}
			\|W_{1}\mathbf{1}_{I_{N}}T_{N,\epsilon}^{z}[\mathbf{1}_{I_{N}}W_{2}]f\|_{\cH} ^{2} &=\int_{I_{N}} \int_{\mathbb{T}} |W_{1}(t,x) T_{N,\epsilon}^{z}([\mathbf{1}_{I_{N}}W_{2}]f)(t,x)|^{2} dt dx \\
			&\leq \|W_{1}\|_{L_{t}^{\infty}L_{x}^{\infty}(I,\M)}\int_{I_{N}} \int_{\mathbb{T}} |T_{N,\epsilon}^{z}([\mathbf{1}_{I_{N}}W_{2}]f)(t,x)|^{2} dt dx\\
			&=\|W_{1}\|_{L_{t}^{\infty}L_{x}^{\infty}(I,\M)}^{2} \|T_{N,\epsilon}^{z}([\mathbf{1}_{I_{N}}W_{2}]f)\|_{L_{t,x}^{2}(I_{N},\M)}^{2}\\
			&\leq \|W_{1}\|_{L_{t}^{\infty}L_{x}^{\infty}(I,\M)}^{2} \|T_{N,\epsilon}^{z}\|_{\cH \to \cH}^{2} \|[\mathbf{1}_{I_{N}}W_{2}]f\|_{L_{t,x}^{2}(I_{N},\M)}^{2} \\
			& \leq C(\operatorname{Im}(z)) \|W_{1}\|_{L_{t}^{\infty}L_{x}^{\infty}(I,\M)}^{2} \|W_{2}\|_{L_{t}^{\infty}L_{x}^{\infty}(I,\M)}^{2} \|f\|_{\cH} \\
			& \leq C(\operatorname{Im}(z))\|W_{1}\|_{L_{t}^{\infty}L_{x}^{\infty}(I,\M)}^{2} \|W_{2}\|_{L_{t}^{\infty}L_{x}^{\infty}(I,\M)}^{2}.
		\end{align*}       
Thus, for $\operatorname{Re}(z)=-1,$ we have  
		\begin{align}\label{PR6}
			\|W_{1}\mathbf{1}_{I_{N}}T_{N,\epsilon}^{z}[\mathbf{1}_{I_{N}}W_{2}]\|_{\Sp^{\infty}} &=\|W_{1}\mathbf{1}_{I_{N}}T_{N,\epsilon}^{z}[\mathbf{1}_{I_{N}}W_{2}]\|_{\cH \to \cH} \nonumber \\
            & \leq C(\operatorname{Im}(z))\|W_{1}\|_{L_{t}^{\infty}L_{x}^{\infty}(I,\M)} \|W_{2}\|_{L_{t}^{\infty}L_{x}^{\infty}(I,\M)}
		\end{align}
We may choose $ \tau=\operatorname{Re}(z)/(1+ \operatorname{Re}(z)) \in [0,1]$ in Lemma \ref{s-interpolation},  and interpolate between \eqref{PR4} and \eqref{PR6},  to obtain
 \begin{flalign*}
\|W_{1}\mathbf{1}_{I_{N}}T_{N,\epsilon}^{0}[\mathbf{1}_{I_{N}}W_{2}]\|_{\Sp^{2(\operatorname{Re}(z)+1)}}  \lesssim_{\theta}  \|W_{1} \|_{L_{t}^{u(\operatorname{Re}(z)+1)}L_{x}^{2(\operatorname{Re}(z)+1)}(I,\M)} \|W_{2} \|_{L_{t}^{u(\operatorname{Re}(z)+1)}L_{x}^{2(\operatorname{Re}(z)+1)}(I,\M)},
            \end{flalign*} 
 where the constant depends on $\theta$ and independent of $\epsilon.$
Recalling $T^{0}_{N, \epsilon}=T_{N, \epsilon}$ and taking $\epsilon \to 0,$ we get
 \begin{flalign*}
\|W_{1}\mathbf{1}_{I_{N}}T_{N}[\mathbf{1}_{I_{N}}W_{2}]\|_{\Sp^{2(\operatorname{Re}(z)+1)}}  \lesssim_{\theta}  \|W_{1} \|_{L_{t}^{u(\operatorname{Re}(z)+1)}L_{x}^{2(\operatorname{Re}(z)+1)}(I,\M)} \|W_{2} \|_{L_{t}^{u(\operatorname{Re}(z)+1)}L_{x}^{2(\operatorname{Re}(z)+1)}(I,\M)}.
            \end{flalign*}
Put  $\beta=u(\operatorname{Re}(z)+1)$ and  $\alpha=2(\operatorname{Re}(z)+1)$. Note that condition \eqref{c1} is compatible with the following condition 
\begin{eqnarray}\label{c2}
\frac{2}{\beta}+\frac{2}{ \theta \alpha}=1 \quad \text{ and } \quad \frac{\theta}{2 (\theta+1)}\leq \frac{1}{\alpha} \leq \frac{1}{2}    
\end{eqnarray}
for all  $\theta>2.$ Thus the desired inequality  \eqref{PR3} holds as long as $\alpha$ and $\beta$ satisfies \eqref{c2}.
\end{proof}
\begin{proof}[\textbf{Proof of Theorem \ref{ose}}]
Let  $I'_{N}$  be an arbitrary interval whose length is $N^{-(\theta-1)}.$
 Denote $c(I'_{N})$ by  the center of the interval $I'_{N}$. We wish to  shift this interval $I'_N$ to the interval   $I_{N}=[\frac{-1}{2N^{\theta-1}},\frac{1}{2N^{\theta-1}}]$ with centre origin. Thus, by the change of variables, we get 
\begin{equation}\label{ts1}
			\left\|\sum_{j} \lambda_{j}|\mathcal{E}_{N}a_{j}|^{2}\right\|_{L^{p}_{t}L^{q}_{x}( I'_{N} \times \M)}=\left\|\sum_{j} \lambda_{j}|\mathcal{E}_{N}b_{j}|^{2}\right\|_{L^{p}_{t}L^{q}_{x}( I_{N} \times \M)},
		\end{equation}
		where $b_{j}(n)=a_{j}(n)e^{-2 \pi ic(I'_{N}) |n|^{\theta} }.$
Note that if $(a_{j})_{j}$ is orthonormal then $(b_{j})_{j}$ is also orthonormal in $\ell^{2}$. In view of the above identity,  Proposition \ref{oseP} is applicable to any time interval whose length is $1/N^{\theta -1}$.
 Since $I$  is an interval of finite length, we may cover it by taking finitely many intervals of length $1/N,$ say $I\subset \bigcup_{i=1}^{{[N^{\theta-1}]+1}}I_{i}.$ 
        Now taking  Remark \ref{usrev}, \eqref{ts1} and Proposition \ref{oseP} into account, we obtain
\begin{eqnarray*}
   \left\|\sum_{j} \lambda_{j}|\mathcal{E}_{N}a_{j}|^{2}\right\|_{L^{p}_{t}L^{q}_{x}(I \times \M)}^{p}  & \lesssim & \sum_{i=1}^{[N^{\theta-1}]+1}\left\|\sum_{j} \lambda_{j}|\mathcal{E}_{N}a_{j}|^{2}\right\|_{L^{p}_{t}L^{q}_{x}( I_{i} \times \M)}^{p}\\
   & \lesssim & N^{{\theta-1}} \|\lambda\|_{\ell^{\alpha'}}^p.
\end{eqnarray*}
 Thus, we proved that  if $\alpha' \leq  2q/q+1$ then \eqref{IN12} holds true. 
 \end{proof}
 \begin{remark}
     For any other time interval, we can partition it into subintervals of unit length and apply the same argument as in the proof of Theorem \ref{ose}. Hence, it suffices to prove the OSE on the interval $[0,1].$
 \end{remark}
 \begin{proof}[\textbf{Proof of Theorem~\ref{Strichartz} for the case $\theta = 2$}]Here $\M=\mathbb{R}^n \times \mathbb{T}^m.$ Let $I=[0,1]$ and $\cH=L^{2}(I,L^{2}(\M).$ We start by establishing the claim over a shorter time interval \(I_{N}\). After obtaining this localized version, the general case can be derived by summing over a suitable partition of the full time interval. By Remark~\ref{usrev}, it suffices to establish the following estimate
		\begin{equation}\label{theta=2eq}
			\left\| \sum_{j} \lambda_{j} |\mathcal{E}_{N}a_{j} |^{2} \right\|_{L_{t}^{p} (I_N, L_{z}^{q}(\mathbb{R}^n \times \mathbb{T}^m))} \lesssim \| \lambda \|_{\ell^{\alpha'}},
		\end{equation}
		where $I_{N}=[\frac{-1}{2N},\frac{1}{2N}],\alpha'=\frac{2 q}{q+1}$ and $\mathcal{E}_{N}$ is defined in \eqref{FEP} by taking $\eta(\xi)=\mathbf{1}_{S_{d,1}}(\xi)$ i.e.
		\[
		\mathcal{E}_{N} a_j(z) =  \int_{\xi_1 \in [-N,N]^{n}} \sum_{\xi_2 \in [-N,N]^m} e^{2 \pi i z \cdot \xi} e^{2\pi i t (\lvert \xi_1 \rvert^2 + \lvert \xi_2 \rvert^2)}  a_j(\xi) \, d\xi_1.
		\]
We now proceed to prove \eqref{stri} in the same manner as Proposition~\ref{oseP}, and therefore we shall only present the main steps of the argument. It is enough to prove the following inequality:
		\begin{equation}\label{daul for theta=2}
			\|W_{1}\mathbf{1}_{I_{N}}\mathcal{E}_{N} \mathcal{E}_{N}^{*}[\mathbf{1}_{I_{N}}W_{2}]\|_{\Sp^\alpha } \leq C_{\theta} \|W_{1} \|_{L_{t}^{\beta}L_{z}^{\alpha}}
			\|W_{2} \|_{L_{t}^{\beta}L_{z}^{\alpha}},
		\end{equation} 
        for all $\alpha,\beta \geq 1$ such that $\frac{2}{\beta}+ \frac{d}{\alpha}=1$ and $0 \leq \frac{1}{\alpha} < \frac{1}{d+1}.$ Since \eqref{daul for theta=2} holds trivially for $\alpha=\infty$ by Plancherel’s theorem. So we prove \eqref{daul for theta=2} on $\frac{1}{d+2} \leq \alpha <\frac{1}{d+1}.$
        For $\epsilon > 0,$ we define a map $T_{N,\epsilon}:\cH  \to \cH $ by
        $$F \mapsto T_{N,\epsilon}F(t,z)=\int_{I \times \M} K_{N,\epsilon} (t - t', z - z') F( t', z') \, dz' dt',$$
		where  
        \begin{align*}
        K_{N,\epsilon}(t,x,y)& =\mathbf{1}_{\epsilon < |t| <N^{-1}} (t,x,y) K_{N}(t,x,y)\\
        &=\mathbf{1}_{\epsilon < |t| <N^{-1}} (t,x,y) \int_{\xi_1 \in [-N,N]^{n}} e^{2 \pi i(x \cdot \xi_{1}+t \lvert \xi_{1} \rvert^{2})} \, d\xi_{1} \sum_{\xi_2 \in [-N,N]^{m}} e^{2 \pi i(y \cdot \xi_{2}+t \lvert \xi_{2} \rvert^{2})}.
        \end{align*}
Now we observe that if $x=(x'_{1}, \cdots, x'_{n}),\xi_{1}=(\xi'_{1}, \cdots, \xi'_{n}) \in \mathbb{R}^n,$ then by Lemma \ref{gle} \eqref{l1}, we have
\begin{align*}
   \abs{\int_{\xi_1 \in [-N,N]^{n}} e^{2 \pi i(x \cdot \xi_{1}+t \lvert \xi_{1} \rvert^{2})} \, d\xi_{1}}&=\abs{\int_{\xi'_1 \in [-N,N]} e^{2 \pi i(x'_{1} \cdot \xi'_{1}+t \lvert \xi'_{1} \rvert^{2})} \, d\xi'_{1}}^n \\
   &\lesssim \abs{t}^{-\frac{n}{2}}.
\end{align*}
Similarly, if $y=(y'_{1}, \cdots, y'_{m}) \in \mathbb{T}^m,\xi_{2}=(\xi'_{n+1}, \cdots, \xi'_{n+m}) \in \mathbb{Z}^m,$ then by Proposition \ref{l8}, we have
\begin{align*}
   \abs{\sum_{\xi_2 \in [-N,N]^{m}} e^{2 \pi i(y \cdot \xi_{2}+t \lvert \xi_{2} \rvert^{2})}}=\abs{\sum_{\xi'_{n+1} \in [-N,N]} e^{2 \pi i(y'_{1} \cdot \xi'_{n+1}+t \lvert \xi'_{n+1} \rvert^{2})}}^m \lesssim \abs{t}^{-\frac{m}{2}},
   \end{align*}
   if $t \in I_{N}.$
   Using these inequalities, for $t \in I_{N}$ and $z \in \M,$ we get
   \begin{equation}\label{kernel for theta=2 waveguide}
       \abs{K_{N}(t,z)} \lesssim \abs{t}^{-\frac{d}{2}}.
   \end{equation}
Further, for $z_{1} \in \mathbb{C},$ with $\operatorname{Re}(z_{1}) \in [-1,\frac{d}{2}],$ we define
\[T_{N,\epsilon}^{z_{1}}=K_{N,\epsilon}^{z_{1}}*,\]
where $K_{N,\epsilon}^{z_{1}}(t,z)=t^{z_{1}}K_{N,\epsilon}(t,z).$
Form \eqref{kernel for theta=2 waveguide}, we have for $(t,z) \in I_{N} \times \M,$
$$\abs{K_{N,\epsilon}^{z_{1}}(t,z)} \lesssim \abs{t}^{\operatorname{Re}(z_{1})-\frac{d}{2}}.$$ 
If we repeat the calculation given in Proposition~\ref{oseP} for $\Sp^{2}$, then we obtain
\begin{equation}\label{c2 theta=2}
    \|W_{1}\mathbf{1}_{I_{N}}T_{N,\epsilon}^{z_{1}}[\mathbf{1}_{I_{N}}W_{2}]\|_{\Sp^{2}} \lesssim \|W_{1} \|_{L_{t}^{u}(I,L_{z}^{2}(\M))} \|W_{2} \|_{L_{t}^{u}(I,L_{z}^{2}(\M))}
\end{equation}
provided 
\begin{equation}\label{condition theta=2}
    \frac{1}{u}=\frac{1}{2}+\frac{1}{2}(\operatorname{Re}(z_{1})-\frac{d}{2}), \operatorname{Re}(z_{1}) \in \left(\frac{d-1}{2},\frac{d}{2} \right].
\end{equation}
In the case $\operatorname{Re}(z_{1})=-1$, repeating the argument of Proposition~\ref{oseP} and using the boundedness of the Hilbert transform, we deduce
\begin{equation}\label{c infty theta=2}
    \|W_{1}\mathbf{1}_{I_{N}}T_{N,\epsilon}^{z_{1}}[\mathbf{1}_{I_{N}}W_{2}]\|_{\Sp^{\infty}} \leq C(\operatorname{Im}(z_{1})) \|W_{1} \|_{L_{t}^{\infty}(I,L_{z}^{\infty}(\M))} \|W_{2} \|_{L_{t}^{\infty}(I,L_{z}^{\infty}(\M))}.
\end{equation}
Applying complex interpolation between \eqref{c2 theta=2} and \eqref{c infty theta=2}, we get \eqref{theta=2eq} holds as long as 
$$\frac{2}{\beta}+ \frac{d}{\alpha}=1, \quad \frac{1}{d+2} \leq \alpha <\frac{1}{d+1}.$$
Applying the same argument as in the proof of Theorem~\ref{ose}, we obtain a covering 
\([0,1] = \bigcup_{i=1}^{N} I_{i}\), where \(\{I_{i}\}_{i=1}^{N}\) is a collection of disjoint intervals of length \(N^{-1}\), and consequently we get

\begin{eqnarray*}
   \left\|\sum_{j} \lambda_{j}|\mathcal{E}_{N}a_{j}|^{2}\right\|_{L^{p}_{t}L^{q}_{z}(I \times \M)}^{p}  & = & \sum_{i=1}^{N}\left\|\sum_{j} \lambda_{j}|\mathcal{E}_{N}a_{j}|^{2}\right\|_{L^{p}_{t}L^{q}_{z}( I_{i} \times \M)}^{p}\\
   & \lesssim & N \|\lambda\|_{\ell^{\alpha'}}^p.
\end{eqnarray*}

\end{proof}

\subsection{Application of decoupling}
\begin{proof}[\textbf{Proof of Theorem \ref{diagonal or nondiagonal for single stri}}] 
To prove \eqref{arbitrary loss for wave guide for q}, let $I$ be a subset of $\mathbb{R}$ with length $\sim 1.$
   Let $f$ be a function with $f = P_{\le N} f$ and $q= \frac{2(d + 2)}{d}.$ 
\begin{itemize}
    \item Case:  $\theta \geq 2:$
\end{itemize}   
In this case  by Lemma \ref{l2 decoupling wave}, we have
\begin{equation}\label{decomposed}
    \Big\| e^{it(-\Delta)^{\frac{\theta}{2}}} f \Big\|_{L_t^q L_z^q (I \times \M)} 
\lesssim_{\varepsilon} N^{\varepsilon} \Big( \sum_{\nu, k} \Big\| e^{it(-\Delta)^{\frac{\theta}{2}}} f_{\Theta_{\nu, k}} w_{I}\Big\|_{L_t^q L_z^q(\mathbb{R} \times \M)}^2 \Big)^{1/2}
\end{equation}
with $f_{\Theta_{\nu, k}}=e^{2 \pi i y \cdot \nu} \mathcal{F}_{x}^{-1}(\widehat{f}(\cdot,\nu)\phi_{k}).$ By applying H\"older’s inequality in time:

\[
\Big\| e^{it(-\Delta)^{\frac{\theta}{2}}} f_{\Theta_{\nu, k}} w_{I} \Big\|_{L_t^q L_z^q(\mathbb{R} \times \M)} 
\lesssim \Big\| e^{it(-\Delta)^{\frac{\theta}{2}}} f_{\Theta_{\nu, k}}\Big\|_{L_t^p L_x^q(\mathbb{R} \times \mathbb{R}^{n})},
\]
with $p \geq \frac{4q}{n(q - 2)},$ $\frac{\theta}{p}+\frac{n}{q}=\frac{n}{2}$
i.e. $p$ is the time exponent for the Strichartz estimate on $\mathbb{R}^n.$
Applying Lemma \ref{fractional stri on Rd}, namely the Strichartz estimate for the fractional Schrödinger operator (see \cite{guo2014improved,Dinh,cho2011,cho2015}), we obtain:

\[
\Big\| e^{it(-\Delta)^{\frac{\theta}{2}}} f_{\Theta_{\nu, k}} \Big\|_{L_t^p L_x^q(\mathbb{R} \times \mathbb{R}^{n})}
\lesssim \Big\| f_{\Theta_{\nu, k}} \Big\|_{L_{x}^2(\mathbb{R}^n)}.
\]
Substituting this into \eqref{decomposed} and applying Plancherel’s theorem, we have:
\begin{align}\label{q=2(d+2)}
    \Big\| e^{it(-\Delta)^{\frac{\theta}{2}}} f \Big\|_{L_t^q L_z^q (I \times \M)} &\lesssim_{\varepsilon} N^{\varepsilon} \Big( \sum_{\nu, k} \Big\| f_{\Theta_{\nu, k}} \Big\|_{L_{x}^2(\mathbb{R}^n)}^2 \Big)^{1/2} \nonumber \\
    &= N^{\varepsilon} \Big( \sum_{\nu, k} \Big\| \widehat{f}(\cdot, \nu) \phi_k \Big\|_{L^{2}_{\xi_{1}}(\mathbb{R}^n)}^{2} \Big)^{1/2}, \nonumber \\
    &\lesssim N^{\varepsilon} \Big( \sum_{\nu} \Big\| \widehat{f}(\cdot, \nu) \Big\|_{L^{2}_{\xi_{1}}(\mathbb{R}^n)}^{2} \Big)^{1/2} \nonumber \\
    &= N^{\varepsilon} \Big\| f \Big\|_{L_{z}^2(\M)}.
\end{align}
Now, when \(q = 2,\) then by Plancharel's theorem, we have
\begin{equation}\label{q=2}
    \Big\| e^{it(-\Delta)^{\frac{\theta}{2}}}  f \Big\|_{L_t^2 L_z^2 (I \times \M)} \lesssim \Big\| f \Big\|_{L_{z}^2(\M)}.
\end{equation}
And when \(q = \infty,\) then by Bernstein’s inequality, we have
\begin{equation}\label{q=infty}
    \Big\| e^{it(-\Delta)^{\frac{\theta}{2}}}  f \Big\|_{L_t^\infty L_z^\infty (I \times \M)} \lesssim N^{d/2} \Big\| f \Big\|_{L_{z}^2(\M)}.
\end{equation}
Interpolating \eqref{q=2(d+2)} and \eqref{q=2}, we get for $2 \leq q \leq \frac{2(d+2)}{d},$
then
\[
\Big\| e^{it(-\Delta)^{\frac{\theta}{2}}}  f \Big\|_{L_t^q L_z^q (I \times \M)} \lesssim_{\varepsilon} N^{\varepsilon} \Big\| f \Big\|_{L^2(\M)}.
\]
Interpolating \eqref{q=2(d+2)} and \eqref{q=infty}, we get for $\frac{2(d+2)}{d} \leq q \leq \infty ,$
then
\[
\Big\| e^{it(-\Delta)^{\frac{\theta}{2}}} f \Big\|_{L_t^q L_z^q (I \times \M)} \lesssim_{\varepsilon} N^{ \frac{d}{2}-\frac{d+2}{q}+ \varepsilon} \Big\| f \Big\|_{L^2(\M)}.
\]
To prove \eqref{arbitrary loss for wave guide for pq} for $\theta \geq 2,$ we first apply Hölder’s inequality in the time variable to \eqref{arbitrary loss for wave guide for q}. Then, by interpolating with the trivial endpoint \((p, q) = (\infty, 2)\), we arrive at the desired result. 
\begin{itemize}
    \item Case: $1<\theta < 2$:
\end{itemize}
Let $f$ be a function where the spatial variable is radial in $\mathbb{R}^n$ with $f = P_{\le N} f.$ When $q = \frac{2(d + 2)}{d}$, then by Lemma \ref{l2 decoupling wave}, we have
\begin{align*}
    \Big\| e^{it(-\Delta)^{\frac{\theta}{2}}} f \Big\|_{L_t^q L_z^q (I \times \M)} &\leq \Big\| e^{it(-\Delta)^{\frac{\theta}{2}}} f \Big\|_{L_t^q L_z^q (N^{2-\theta}I \times \M)} \\
&\lesssim_{\varepsilon} N^{\varepsilon} \Big( \sum_{\nu, k} \Big\| e^{it(-\Delta)^{\frac{\theta}{2}}} f_{\Theta_{\nu, k}} w_{N^{2-\theta}I}\Big\|_{L_t^q L_z^q(\mathbb{R} \times \M)}^2 \Big)^{1/2},
\end{align*}
with $f_{\Theta_{\nu, k}}=e^{2 \pi i y \cdot \nu} \mathcal{F}_{x}^{-1}(\widehat{f}(\cdot,\nu)\phi_{k})$ and $w_{N^{2-\theta}I}$ is adapted in $N^{2-\theta}I.$
By scaling the time variable \( t \) to \( N^{2 - \theta} t \), we obtain:
\begin{align}\label{decomposed 1< theta <2}
    \Big\| e^{it(-\Delta)^{\frac{\theta}{2}}} f \Big\|_{L_t^q L_z^q (I \times \M)} \lesssim_{\varepsilon} N^{\varepsilon} N^{\frac{2-\theta}{q}} \Big( \sum_{\nu, k} \Big\| e^{iN^{2-\theta}t(-\Delta)^{\frac{\theta}{2}}} f_{\Theta_{\nu, k}} w_{I}\Big\|_{L_t^q L_z^q(\mathbb{R} \times \M)}^2 \Big)^{1/2},
\end{align}
$w_{I}$ is adapted in $I.$
By applying H\"older’s inequality in time:

\[
\Big\| e^{iN^{2-\theta}t(-\Delta)^{\frac{\theta}{2}}} f_{\Theta_{\nu, k}} w_{I} \Big\|_{L_t^q L_z^q(\mathbb{R} \times \M)} 
\lesssim \Big\| e^{iN^{2-\theta}t(-\Delta)^{\frac{\theta}{2}}} f_{\Theta_{\nu, k}}\Big\|_{L_t^{\infty} L_x^q(\mathbb{R} \times \mathbb{R}^{n})},
\]
Since $q>2,$ applying Strichartz estimate for the fractional Schrödinger operator (see \cite[Theorem 1.5(b)]{guo2014improved}), we obtain:

\[
\Big\| e^{iN^{2-\theta}t(-\Delta)^{\frac{\theta}{2}}} f_{\Theta_{\nu, k}} \Big\|_{L_t^{\infty} L_x^q(\mathbb{R} \times \mathbb{R}^{n})}
\lesssim \Big\| f_{\Theta_{\nu, k}} \Big\|_{L_{x}^2(\mathbb{R}^n)}.
\]
Substituting this into \eqref{decomposed 1< theta <2} and applying Plancherel’s theorem, we have:
\begin{align}\label{q=2(d+2) 1< theta <2}
    \Big\| e^{it(-\Delta)^{\frac{\theta}{2}}} f \Big\|_{L_t^q L_z^q (I \times \M)} \lesssim_{\varepsilon} N^{\varepsilon}  N^{\frac{2-\theta}{q}} \Big\| f \Big\|_{L_{z}^2(\M)}.
\end{align}
The remaining part of the proof can be completed in the same way as in the case $\theta \geq 2.$
\end{proof} 
\subsection{Proof of Theorem \ref{Strichartz} for the case $\theta \neq 2$ and Corollary \ref{arbitrary loss ons}}
\begin{remark}[proof strategy]
    \begin{itemize}
    \item First, we will prove \eqref{extension stri} for the case \( (p, q) = (\infty, 1) \) by utilizing the Minkowski inequality and the Plancherel theorem.
    
    \item Next, we will apply the Littlewood-Paley decomposition for Euclidean space $\mathbb{R}^n$ and torus $\mathbb{T}^m.$ When \(\theta > 1\), we will partition the time interval into disjoint segments. However, when \(0 < \theta < 1\), we will not partition the time interval due to Lemma \ref{dispersive on R}.
    
    \item Then we will dualize the problem, using Lemma \ref{PL1}. To get the inequality between \(\alpha = d+2\) and \(\alpha = \infty\), we will use Lemma \ref{PL5}.
    
    \item Then we will prove the dualized inequality for \(\frac{1}{d+2} \leq \frac{1}{\alpha} < \frac{1}{d+1}\) using Lemma \ref{s-interpolation} for \(\Sp^2\)-norm and \(\Sp^\infty\)-norm.
   \item As previously discussed in Remark \ref{cpf}, the phase function \(\phi(\xi) = |\xi_1|^\theta + |\xi_2|^\theta\) exhibits behavior that imposes various restrictions on the parameter \(\theta\). Since the kernel in this case can be expressed as a product of kernels on the Euclidean and toroidal components separately, we can apply the Littlewood–Paley decomposition independently to each component. This approach enables us to utilize Lemma~\ref{dispersive on R}. 
\end{itemize}
\end{remark}
    \begin{proof}[\textbf{Proof of Theorem \ref{Strichartz} for the case $\theta \neq 2$}]
        It is enough to prove for the time intverval $[0,1].$ And by Remark \ref{usrev}, it suffices to establish the following estimate        
\begin{equation}\label{extension stri}
		\left\| \sum_j \lambda_j |\mathcal{E}_{N}a_j|^2 \right\|_{L_t^p([0,1],L_z^q(\mathbb{R}^n \times \mathbb{T}^m))} \lesssim N^{\sigma} \| \lambda \|_{\ell^{\alpha'}},
	\end{equation}
   holds for all orthonormal system $(a_{j})_{j}$ in $\ell_{\xi_{2}}^{2}L^{2}_{\xi_{1}}$ if $\alpha' \leq \frac{2q}{q+1}$ and $\mathcal{E}_{N}$ is defined in \eqref{FEP} by taking $\eta(\xi)=\eta^{d}(\xi).$ 
        We denote $I=[0,1]$ throughout the proof.
        At \((q, p) = (1, \infty)\), we have \(\alpha' = 1\). By  Minkowski inequality  and \eqref{d1}, we have 
\[
\left\| \sum_j \lambda_j |\mathcal{E}_N a_j|^2 \right\|_{L_t^\infty L_z^{1}}
\leq \|\lambda\|_{\ell^1} \sup_j \left\| \mathcal{E}_N a_j \right\|_{L_t^\infty L_z^{2}}^{2} 
\lesssim \|\lambda\|_{\ell^1}.
\]
Let \(\psi\) be a \(C^\infty\) radial function on \(\mathbb{R}^n\) with
\[
\text{supp } \psi \subset \{\xi_{1} \in \mathbb{R}^n: |\xi_{1}| \leq 2\}, \quad \psi(\xi_{1}) = 1 \quad \text{if} \quad |\xi_{1}| \leq 1.
\]
Let \(k_{1} \in \mathbb{N}\) and
\[
\psi_{k_{1}}(\xi_{1}) = \psi(2^{-k_{1}}\xi_{1}) - \psi(2^{-k_{1}+1}\xi_{1}), \quad \xi_{1} \in \mathbb{R}^n.
\]
Then we have
\[
\text{supp } \psi_{k_{1}} \subset \{\xi_{1} \in \mathbb{R}^n: 2^{k_{1}-1} \leq |\xi_{1}| \leq 2^{k_{1}+1}\},
\]
and, with \(\psi_0 = \psi\),
\[
\sum_{k_{1}=0}^\infty \psi_{k_{1}}(\xi_{1}) = 1 \quad \text{if} \quad \xi_{1} \in \mathbb{R}^n.
\]
Then family $ \{\psi_{k_{1}}\}_{k_{1}\geq 0} \subset   C_0^\infty(\mathbb{R}^n)$  forms a Littlewood-Paley decomposition  for the Euclidean space $\mathbb{R}^n$ and each   \( \psi_{k_{1}}\) is supported in \(\{ |s| \approx 2^{k_{1}} \}.\)
We further restrict  \eqref{FEP} by plugging  $\psi_{k_{2}}$ into it, specifically, we introduce  modified Fourier extension  operator 
\begin{eqnarray}\label{extension operator k1}
    \mathcal{E}_{N,k_{1}} a_j(z) = \sum_{\xi_2} \int_{\xi_1} e^{2 \pi i z \cdot \xi} e^{2\pi i t (\lvert \xi_1 \rvert^\theta + \lvert \xi_2 \rvert^\theta)} \eta^{d}\left( \frac{\xi}{N} \right) \psi_{k_{1}}(\xi_1) a_j(\xi) \, d\xi_1, \quad z\in \mathbb R^n \times \mathbb T^m.
\end{eqnarray}
The modified  restriction operator $\mathcal{E}_{N,k_{1}}^*$  can be defined similarly (see \eqref{FRO}). 
Exploiting similar calculations as in  \eqref{P1}, we obtain
\[\mathcal{E}_{N,k_{1}} \circ \mathcal{E}_{N,k_{1}}^* =K_{N,k_{1}} \ast ,\] where
\begin{align*}
     K_{N,k_{1}}(t,x,y)& = \int_{\xi_{1}} e^{2 \pi i(x \cdot \xi_{1}+t \lvert \xi_{1} \rvert^{\theta})} \left(\eta_{1}^{n}(\frac{\xi_{1}}{N})\right)^{2} \left(\psi_{k_{1}}(\xi_{1})\right)^{2} d\xi_{1} \cdot \sum_{\xi_{2}} e^{2 \pi i(y \cdot \xi_{2}+t \lvert \xi_{2} \rvert^{\theta})} \left(\eta_{1}^{m}(\frac{\xi_{2}}{N})\right)^{2} .\\
\end{align*}
To controll the kernel $K_{N,k_{1}}$ for \(\theta > 1\), we may split the interval \(I\) into \(\approx 2^{(\theta - 1) k_{1}}\) short intervals \(\{I_{k_{1},n'}\}_{n'}\) of length \(2^{(1 - \theta) k_{1}}\), i.e. 
\[I \subset \cup_{n'=1}^{2^{(\theta - 1) k_{1}}}I_{k_{1},n'}, \quad  |I_{k_{1}, n'}|\approx 2^{(1 - \theta) k_{1}}.\]
By Minkowski inequality, we have  
\begin{align}\label{strichartz proof for I}
  \left\| \sum_j \lambda_j |\mathcal{E}_N a_j|^2 \right\|_{L_t^p (I, L_z^q(\M))}&\leq \left(\sum_{k_{1} \geq 0} \left\| \sum_j \lambda_j |\mathcal{E}_{N,k_{1}} a_j|^2 \right\|_{L_t^p(I, L_z^q(\M))}^{\frac{1}{2}}\right)^{2} \nonumber \\
     &\leq \left(\sum_{2^{k_{1}} \lesssim N} \left( \sum_{n'=1}^{2^{(\theta-1)k_{1}}} \left\| \sum_j \lambda_j |\mathcal{E}_{N,k_{1}} a_j|^2 \right\|_{L_t^p(I_{k_{1},n'}, L_z^q(\M))}^{p}\right)^{\frac{1}{2p}}\right)^{2} \nonumber \\
    &\leq \resizebox{0.58\textwidth}{!}{$   \left(\sum_{2^{k_{1}} \lesssim N} 2^{(\theta-1)\frac{k_{1}}{2p}} \max\limits_{n'} \left\| \sum_j \lambda_j |\mathcal{E}_{N,k_{1}} a_j|^2 \right\|_{L_t^p(I_{k_{1},n'}, L_z^q(\M))}^{\frac{1}{2}}\right)^{2}.
$}
\end{align}
Now we will prove the following inequality:
\begin{equation}\label{ons for small time k1}
   \left\| \sum_j \lambda_j |\mathcal{E}_{N,k_{1}} a_j|^2 \right\|_{L_t^p L_z^q(I_{k_{1},n'}\times \M)} \lesssim \begin{cases}
   N^{\frac{1}{p}}2^{\frac{(2-\theta)k_{1}}{p}} \| \lambda \|_{\ell^{\alpha'}}    & for \quad k_{1} \geq 0,\theta < 2 \\
N^{\frac{1}{p}}2^{\frac{(2-\theta)k_{1}}{p}} \| \lambda \|_{\ell^{\alpha'}}    & for \quad k_{1}>0,\theta>2 \\
 
  N^{(1+\frac{(\theta-2)n}{m+n})\frac{1}{p}} \| \lambda \|_{\ell^{\alpha'}}  & for \quad k_{1}=0,\theta>2
\end{cases}.
\end{equation}
In order to prove \eqref{ons for small time k1}, again we use Littlewood-Paley decomposition. Let  family $ \{\phi_{k_{2}}\}_{k_{2}\geq 0} \subset   C_0^\infty(\mathbb{R}^{m})$  forms a Littlewood-Paley decomposition  for the torus,  i.e. 
 \(\sum_{k \geq 0} \phi_k(s) = 1\)  and each   \( \phi_k\) is supported in \(\{ |s| \approx 2^k \}.\) We further restrict  \eqref{extension operator k1} by plugging  $\phi_{k_{2}}$ into it, so modified Fourier extension  operator 
\[
\mathcal{E}_{N,k_{1},k_{2}} a_j(z) = \sum_{\xi_2} \int_{\xi_1} e^{2 \pi i z \cdot \xi} e^{2\pi i t (\lvert \xi_1 \rvert^\theta + \lvert \xi_2 \rvert^\theta)} \eta^{d}\left( \frac{\xi}{N} \right) \phi_{k_{2}}(\xi_2) \psi_{k_{1}}(\xi_{1}) a_j(\xi) \, d\xi_1, \quad z\in \mathbb R^n \times \mathbb T^m.
\]
The modified  restriction operator $\mathcal{E}_{N,k_{1},k_{2}}^*$  can be defined similarly (see \eqref{FRO}). 
Exploiting similar calculations as in  \eqref{P1}, we obtain
\[\mathcal{E}_{N,k_{1},k_{2}} \circ \mathcal{E}_{N,k_{1},k_{2}}^* =K_{N,k_{1},k_{2}} \ast ,\] where
\begin{align*}
     K_{N,k_{1},k_{2}}(t,x,y)& = \resizebox{0.79\textwidth}{!}{$ \int_{\xi_{1}} e^{2 \pi i(x \cdot \xi_{1}+t \lvert \xi_{1} \rvert^{\theta})} \left(\eta_{1}^{n}(\frac{\xi_{1}}{N})\right)^{2} \left(\psi_{k_{1}}(\xi_{1})\right)^{2} d\xi_{1} \cdot \sum_{\xi_{2}} e^{2 \pi i(y \cdot \xi_{2}+t \lvert \xi_{2} \rvert^{\theta})} \left(\eta_{1}^{m}(\frac{\xi_{2}}{N})\right)^{2} \left(\phi_{k_{2}}(\xi_{2})\right)^{2}    $}\\
     &:=J_{k_{1}}(t,x)\cdot J_{k_{2}}(t,y).
\end{align*}
For all $k_{1}\geq 1,k_{2}>0,$ taking  \(\widehat{f}(\xi_1) = \eta_1^n \left( \frac{\xi_1}{N} \right)\) and \(\varphi(|\xi_1|) = \eta_1^n \left( \frac{\xi_1}{hN} \right) \left( \psi_{k_{1}}\left( \frac{\xi_1}{h} \right) \right)^2\)  with
\(h = 2^{-(k_{1}+k_{2})}\) in Lemma \ref{dispersive on R},  we obtain 
\begin{align*}
    \|J_{k_{1}}\|_{L^{\infty}(\mathbb R^n)}  & \lesssim 2^{n (k_{1}+k_{2})} (1 + |t| 2^{(k_{1}+k_{2})\theta})^{-n/2} \, \|f\|_{L^1(\mathbb{R}^n)} \lesssim  2^{(2 - \theta)\frac{n (k_{1}+k_{2})}{2}} \, |t|^{-n/2},
\end{align*}
 for every \(t \in \left[ -2^{(1 - \theta){(k_{1}}+k_{2})} t_1, 2^{(1 - \theta){(k_{1}}+k_{2})} t_1 \right]\)  with some $t_1>0.$
  To estimate $J_{k_{2}},$  choose $f$ on $\mathbb T^m$ such that \(\widehat{f}(\xi_2) = \eta_1^m\left( \frac{\xi_2}{N} \right)\). By Poisson summation formula, we get 
   \[\|f\|_{L^1(\mathbb{T}^m)} \leq \left\| \left( \eta_1^m \left( \frac{\cdot}{N} \right) \right)^\vee \right\|_{L^1(\mathbb{R}^m)} 
 = c_2<\infty, \]
 where $c_2$ is some constant   independent of $N$.
Taking  this $f$, and \(\varphi(|\xi_2|) = \eta_1^m\left( \frac{\xi_2}{hN} \right)\left( \phi_k\left( \frac{\xi_2}{h} \right) \right)^2\)
with
\(h = 2^{-(k_{1}+k_{2}})\) in Lemma \ref{dispersive on R},  we obtain 
\begin{align*}
    \|J_{k_{2}}\|_{L^{\infty}(\mathbb T^m)}  & \lesssim 2^{m (k_{1}+k_{2})} (1 + |t| 2^{(k_{1}+k_{2})\theta})^{-m/2} \, \|f\|_{L^1(\mathbb T^m)} \lesssim  2^{(2 - \theta)\frac{m (k_{1}+k_{2})}{2}} \, |t|^{-m/2}
\end{align*}
for every \(t \in \left[ -2^{(1 - \theta){(k_{1}+k_{2}})} t_2, 2^{(1 - \theta){(k_{1}+k_{2}})} t_2 \right]\)  with some $t_2>0.$ \\
Combining the above  inequalities, for $\abs{t} \lesssim 2^{(1-\theta)(k_{1}+k_{2})}$ and $(x,y) \in \mathbb{R}^n \times \mathbb{T}^m(= \M), d=\dim(\M)=m+n$, we  have 
\begin{eqnarray}\label{kernel for waveguide}
    \abs{K_{N,k_{1},k_{2}}(t,x,y)} \lesssim 2^{(2-\theta)\frac{d(k_{1}+k_{2})}{2}} |t|^{-d/2}, \quad \text{ for all } k_{1}>0,k_{2} \geq 0.
\end{eqnarray}
For $k_{1}=0,$ since $\psi_{0}=\psi,$ which is nice compactly supported function in $\text{supp } \psi \subset \{\xi_{1} \in \mathbb{R}^n: |\xi_{1}| \leq 2\}$, so we cannot apply Lemma \ref{dispersive on R} on $J_0,$ but we have
$$\|J_{0}\|_{L^{\infty}(\mathbb R^n)}  \lesssim 1,$$
then 
\begin{eqnarray*}
    \abs{K_{N,0,k_{2}}(t,x,y)} &\lesssim 2^{(2 - \theta)\frac{m k_{2}}{2}} \, |t|^{-m/2}. 
\end{eqnarray*}
 Since $\abs{t} \leq 1,$ if $\theta < 2$ then we take 
\begin{eqnarray}\label{kernel for waveguide 2}
    \abs{K_{N,0,k_{2}}(t,x,y)} \lesssim 2^{(2-\theta)\frac{d k_{2}}{2}} |t|^{-d/2},
\end{eqnarray}
for other $\theta,$ we take 
\begin{eqnarray}\label{kernel for waveguide 3}
    \abs{K_{N,0,k_{2}}(t,x,y)} &\lesssim 2^{(2 - \theta)\frac{m k_{2}}{2}} \, |t|^{-d/2}. 
\end{eqnarray}
In order to take the advantage of dispersive estimate \eqref{kernel for waveguide} for each $k_{2},$ we need to make the time interval $I_{k_{1},n'}$ sufficiently small.  To this end,  for \(\theta > 1\), we may split the interval \(I_{k_{1},n'}\) into \(\approx 2^{(\theta - 1) k_{2}}\) short intervals \(\{I_{k_{1},k_{2},n',n''}\}_{n''}\) of length \(2^{(1 - \theta)(k_{1}+ k_{2})}\), i.e. 
\[I_{k_{1},n'} \subset \cup_{n''=1}^{2^{(\theta - 1) k_{2}}}I_{k_{1},k_{2},n',n''}, \quad  |I_{k_{1},k_{2},n',n''}|\approx 2^{(1 - \theta) (k_{1}+k_{2})}.\]
Let us denote $L^p_{t}L_{z}^{q}=L_t^{p}{(I_{k_{1},k_{2},n',n''},L^{q}_{z}}(\M)),$ and by Minkowski inequality, we have  
\begin{align}\label{strichartz proof}
    \left\| \sum_j \lambda_j |\mathcal{E}_{N,k_{1}} a_j|^2 \right\|_{L_t^p (I_{k_{1},n'}, L_z^q(\M))}&\leq \left(\sum_{k_{2} \geq 0} \left\| \sum_j \lambda_j |\mathcal{E}_{N,k_{1},k_{2}} a_j|^2 \right\|_{L_t^p(I_{k_{1},n'}, L_z^q(\M))}^{\frac{1}{2}}\right)^{2} \nonumber \\
     &\leq \left(\sum_{2^{k_{2}} \lesssim N} \left( \sum_{n''=1}^{2^{(\theta-1)k_{2}}} \left\| \sum_j \lambda_j |\mathcal{E}_{N,k_{1},k_{2}} a_j|^2 \right\|_{L_t^p L_z^q}^{p}\right)^{\frac{1}{2p}}\right)^{2} \nonumber \\
    &\leq \left(\sum_{2^{k_{2}} \lesssim N} 2^{(\theta-1)\frac{k_{2}}{2p}} \max\limits_{n''} \left\| \sum_j \lambda_j |\mathcal{E}_{N,k_{1},k_{2}} a_j|^2 \right\|_{L_t^pL_z^q}^{\frac{1}{2}}\right)^{2}.
\end{align}
Now we will prove the following inequality:
\begin{eqnarray}\label{ons for small time}
\left\| \sum_j \lambda_j |\mathcal{E}_{N,k_{1},k_{2}} a_j|^2 \right\|_{L_t^p L_z^q} \lesssim
\begin{cases}
 2^{\frac{(2-\theta)(k_{1}+k_{2})}{p}} \| \lambda \|_{\ell^{\alpha'}}    & for \quad k_{1},k_{2} \geq 0, \theta < 2 \\
 2^{\frac{(2-\theta)(k_{1}+k_{2})}{p}} \| \lambda \|_{\ell^{\alpha'}} & for \quad k_{1} \geq 1,k_{2} \geq 0,\theta > 2 \\
  2^{\frac{(2-\theta)mk_{2}}{p(m+n)}} \| \lambda \|_{\ell^{\alpha'}}  & for \quad k_{1}=0, \theta >2
\end{cases}
\end{eqnarray}
The key ideas to prove \eqref{ons for small time} are similar to the proof of Proposition \ref{oseP}. Put  $\cH= L^2(I_{k_{1},k_{2},n',n''}, L^2(\M))$,  $\alpha=2q' \geq 1$ and $\beta=2p' \geq 1$.  By duality principle Lemma \ref{PL1}, to prove  the desired estimate \eqref{ons for small time} for all $(\frac{1}{q},\frac{1}{p}) \in (A,B]$ (see Figure \ref{fig:enter-label}), it suffices to show that 
		\begin{equation}\label{dual onse for short time}
        \|W_{1}\mathcal{E}_{N,k_{1},k_{2}} \mathcal{E}_{N,k_{1},k_{2}}^{*}W_{2}\|_{\Sp^\alpha } \lesssim
        \begin{cases}
            2^{\frac{(2-\theta)(k_{1}+k_{2})}{p}} \|W_{1} \|_{L_{t}^{\beta}L_{z}^{\alpha}}
			\|W_{2} \|_{L_{t}^{\beta}L_{z}^{\alpha}} & for \quad k_{1},k_{2} \geq 0, \theta < 2 \\
            2^{\frac{(2-\theta)(k_{1}+k_{2})}{p}} \|W_{1} \|_{L_{t}^{\beta}L_{z}^{\alpha}}
			\|W_{2} \|_{L_{t}^{\beta}L_{z}^{\alpha}} &  for \quad k_{1} \geq 1,k_{2} \geq 0,\theta > 2  \\
            2^{\frac{(2-\theta)mk_{2}}{p(m+n)}} \|W_{1} \|_{L_{t}^{\beta}L_{z}^{\alpha}}
			\|W_{2} \|_{L_{t}^{\beta}L_{z}^{\alpha}} & for \quad \quad k_{1}=0, \theta >2
        \end{cases}.
		\end{equation}
By the hypothesis, we have
\[\frac{2}{\beta}+\frac{d}{\alpha}=1 \quad \text{and} \quad d+1< \alpha \leq \infty.\]
 In order to prove \eqref{dual onse for short time} for the given range, it is sufficient to prove it for \(\frac{1}{d+2} \leq \frac{1}{\alpha} < \frac{1}{d+1}\).  As the  remaining range   will follow from Lemma \ref{PL5} by interpolation between \(\alpha = d+2\) and \(\alpha = \infty.\) Now we shall briefly sketch the proof for the range   \(\frac{1}{d+2} \leq \frac{1}{\alpha} < \frac{1}{d+1}\).

 For any small $\varepsilon >0,$ we define an analytic family of operators $T^{ z'}_{N,k_{1},k_{2},\varepsilon}:\cH \to \cH  $ on the strip $\{z' \in \mathbb{C}: -1 \leq \operatorname{Re}(z')\leq \frac{d}{2}\}$ as follows:
$$T^{ z'}_{N,k_{1},k_{2},\varepsilon}:=K^{ z'}_{N,k_{1},k_{2},\varepsilon} \ast $$
where $K^{z'}_{N,k_{1},k_{2},\varepsilon}(t,z)=\mathbf{1}_{\varepsilon<\lvert t \rvert <2^{(1-\theta)(k_{1}+k_{2})}}  t ^{z'}K_{N,k_{1},k_{2}}(t,z)$. 
 By \eqref{kernel for waveguide},\eqref{kernel for waveguide 2} and \eqref{kernel for waveguide 3}, for $\abs{t-s} \lesssim 2^{(1-\theta)(k_{1}+k_{2})},$ we have 
\begin{equation*}
\lvert K^{\epsilon}_{z',N,k_{1},k_{2}}(t-s,x-x',y-y') \rvert \lesssim
\begin{cases}
    \lvert t-s \rvert^{-\frac{d-2\operatorname{Re}(z')}{2}} 2^{(2-\theta)\frac{d (k_{1}+k_{2})}{2}} & for \quad k_{1},k_{2} \geq 0, \theta < 2,  \\
    \lvert t-s \rvert^{-\frac{d-2\operatorname{Re}(z')}{2}} 2^{(2-\theta)\frac{d(k_{1}+k_{2})}{2}} & for \quad k_{1} \geq 1,k_{2} \geq 0,\theta > 2 \\
     \lvert t-s \rvert^{-\frac{d-2\operatorname{Re}(z')}{2}} \, 2^{(2 - \theta)\frac{m k_{2}}{2}}  & for \quad k_{1}=0, \theta >2
\end{cases}.
\end{equation*}
In view of this  and following a similar steps as in \eqref{ds1} and \eqref{PR4}, we obtain
\begin{equation}\label{C^2}
\|W_1T^{z'}_{N,k_{1},k_{2},\varepsilon} W_2\|_{\Sp^{2}} \lesssim
\begin{cases}
    \|W_1\|_{L_t^{u}L^{2}_{z}}\|W_2\|_{L_t^{u}L^{2}_{z}} 2^{(2-\theta)\frac{d (k_{1}+k_{2})}{2}} & for \quad k_{1},k_{2} \geq 0, \theta < 2,  \\
    \|W_1\|_{L_t^{u}L^{2}_{z}}\|W_2\|_{L_t^{u}L^{2}_{z}} 2^{(2-\theta)\frac{d (k_{1}+k_{2})}{2}} & for \quad k_{1} \geq 1,k_{2} \geq 0,\theta > 2
    \\
    \|W_1\|_{L_t^{u}L^{2}_{z}}\|W_2\|_{L_t^{u}L^{2}_{z}} 2^{(2 - \theta)\frac{m k_{2}}{2}} & for \quad k_{1}=0, \theta >2
\end{cases} 
\end{equation}
where 
\begin{eqnarray}\label{c2wave}
    2\operatorname{Re}(z')-d \in (-1,0]\quad \text{and} \quad \frac{1}{u}=\frac{1}{2}+\frac{1}{2}(d-\operatorname{Re}(z')) \in \left(\frac{1}{4},\frac{1}{2}\right].
\end{eqnarray}
Next, for  $\operatorname{Re}(z')=-1,$  performing similar  calculations as in \eqref{eb} and   \eqref{PR6}, we have   
		\begin{align}\label{C^infty}
			\|W_1T_{N,k_{1},k_{2},\epsilon}^{z'}W_2\|_{\Sp^{\infty}} &=\|W_{1}T_{N,k_{1},k_{2},\epsilon}^{z'}W_{2}\|_{\cH \to \cH} \nonumber \\
			& \leq C(\operatorname{Im}(z'))\|W_{1}\|_{L_{t}^{\infty}L_{z}^{\infty}} \|W_{2}\|_{L_{t}^{\infty}L_{z}^{\infty}}.
		\end{align}
We may choose $ \tau=\operatorname{Re}(z')/(1+ \operatorname{Re}(z')) \in [0,1]$ in Lemma \ref{s-interpolation},  and interpolate between \eqref{C^2} and \eqref{C^infty},  to obtain for $k_{1},k_{2} \geq 0, \theta < 2$ and $\quad k_{1} \geq 1,k_{2} \geq 0,\theta > 2,$
		\begin{flalign*}
			& \|W_{1}T_{N,k_{1},k_{2},\epsilon}^{0}W_{2}\|_{\Sp^{2(\operatorname{Re}(z')+1)}} \\
            &\lesssim 2^{\frac{1}{p}(2-\theta)(k_{1}+k_{2})} \|W_{1} \|_{L_{t}^{u(\operatorname{Re}(z')+1)}L_{z}^{2(\operatorname{Re}(z')+1)}} \|W_{2} \|_{L_{t}^{u(\operatorname{Re}(z')+1)}L_{z}^{2(\operatorname{Re}(z')+1)}},
		\end{flalign*} 
        and for $ k_{1}=0, \theta >2,$ we get 
        \begin{flalign*}
	&
    \|W_{1}T_{N,k_{1},k_{2},\epsilon}^{0}W_{2}\|_{\Sp^{2(\operatorname{Re}(z')+1)}} \\
            &\lesssim 2^{\frac{m}{pd}(2-\theta)k_{2}} \|W_{1} \|_{L_{t}^{u(\operatorname{Re}(z')+1)}L_{z}^{2(\operatorname{Re}(z')+1)}} \|W_{2} \|_{L_{t}^{u(\operatorname{Re}(z')+1)}L_{z}^{2(\operatorname{Re}(z')+1)}},
		\end{flalign*} 
		where the constants are independent of $\epsilon.$
		Recalling $T^{0}_{N,k_{1},k_{2}, \epsilon}=T_{N,k_{1},k_{2}, \epsilon}$ and taking $\epsilon \to 0,$ for $k_{1},k_{2} \geq 0, \theta < 2$ and $ k_{1} \geq 1,k_{2} \geq 0,\theta > 2,$ we get
		\begin{flalign*}
			& \|W_{1}T_{N,k_{1},k_{2}}W_{2}\|_{\Sp^{2(\operatorname{Re}(z')+1)}} \\ & \lesssim 2^{\frac{1}{p}(2-\theta)(k_{1}+k_{2})} \|W_{1} \|_{L_{t}^{u(\operatorname{Re}(z')+1)}L_{z}^{2(\operatorname{Re}(z')+1)}} \|W_{2} \|_{L_{t}^{u(\operatorname{Re}(z')+1)}L_{z}^{2(\operatorname{Re}(z')+1)}},
		\end{flalign*}
         and for $ k_{1}=0, \theta >2,$ we get 
\begin{flalign*}
			& \|W_{1}T_{N,k_{1},k_{2}}W_{2}\|_{\Sp^{2(\operatorname{Re}(z')+1)}} \\ & \lesssim 2^{\frac{1}{p}(2-\theta)\frac{mk_{2}}{m+n}} \|W_{1} \|_{L_{t}^{u(\operatorname{Re}(z')+1)}L_{z}^{2(\operatorname{Re}(z')+1)}} \|W_{2} \|_{L_{t}^{u(\operatorname{Re}(z')+1)}L_{z}^{2(\operatorname{Re}(z')+1)}},
		\end{flalign*}

		Put  $\beta=u(\operatorname{Re}(z')+1)$ and  $\alpha=2(\operatorname{Re}(z')+1)$. Note that condition \eqref{c2wave}  is compatible with the following condition 
		\begin{eqnarray}\label{c3}
			\frac{2}{\beta}+\frac{d}{ \alpha}=1 \quad \text{ and } \quad \frac{1}{d+2} \leq \frac{1}{\alpha} < \frac{1}{d+1},
		\end{eqnarray}
		for all $\theta >1.$ Thus the desired inequality  \eqref{dual onse for short time} holds as long as $\alpha$ and $\beta$ satisfies \eqref{c3}.
        Thus for all $(\frac{1}{q},\frac{1}{p}) \in (A,B]$ inequality \eqref{dual onse for short time} implies \eqref{ons for small time}. 
Using the inequality \eqref{ons for small time} in \eqref{strichartz proof}, we get 
\begin{eqnarray*}
\left\| \sum_j \lambda_j |\mathcal{E}_{N,k_{1}} a_j|^2 \right\|_{L_t^p L_z^q(I_{k_{1},n'}\times \M)} \lesssim
\begin{cases}
 N^{\frac{1}{p}}2^{\frac{(2-\theta)k_{1}}{p}} \| \lambda \|_{\ell^{\alpha'}}    & for \quad k_{1} \geq 0, \theta < 2 \\
 N^{\frac{1}{p}} 2^{\frac{(2-\theta)k_{1}}{p}} \| \lambda \|_{\ell^{\alpha'}} & for \quad k_{1} \geq 1 ,\theta > 2 \\
 N^{\frac{1}{p}(1+\frac{n(\theta-2)}{m+n})}  \| \lambda \|_{\ell^{\alpha'}}  & for \quad k_{1}=0, \theta >2
\end{cases}.
\end{eqnarray*}
Using this inequality in \eqref{strichartz proof for I}, we get
\begin{equation*}
\left\| \sum_j \lambda_j |\mathcal{E}_N a_j|^2 \right\|_{L_t^p (I, L_z^q(\M))}\lesssim \begin{cases}
    N^{\frac{2}{p}} \| \lambda \|_{\ell^{\alpha'}} & for \quad 1 <\theta < 2\\
    \left(N^{(1+\frac{n(\theta-2)}{m+n})\frac{1}{2p}}+N^{\frac{1}{p}} \right)^{2} \| \lambda \|_{\ell^{\alpha'}} &   \theta >2
\end{cases}.
\end{equation*}
When \(1 < \theta \leq 3+\tfrac{m}{n}\), we have 
\[
\left(1+\frac{n(\theta-2)}{m+n}\right)\frac{1}{2p} \leq \frac{1}{p}.
\]
On the other hand, if \(\theta > 3+\tfrac{m}{n}\), then 
\[
\left(1+\frac{n(\theta-2)}{m+n}\right)\frac{1}{2p} > \frac{1}{p}.
\]
Therefore, we obtain the desired inequality \eqref{extension stri}.
\\
When $\theta \in (0,1)$ we can get the inequality by directly adding up the estimate from \eqref{ons for small time} over all $k_{1}$ and $k_{2}$ because there is no need to divide the interval $I.$ \\
    \end{proof}

\begin{proof}[\textbf{Proof of Corollary \ref{arbitrary loss ons}}] Let $I$ be a subset of $\mathbb{R}$ with length $\sim 1.$ When $\theta > 2, 1 \leq q \leq \frac{d+2}{d} $ and $\frac{2}{p}+\frac{d}{q}=d.$ Then 
    by using Theorem \ref{diagonal or nondiagonal for single stri} and Minkowski inequality, we get 
    \begin{equation*}
        \Big\|\sum_{j}\lambda_{j}|e^{it(-\Delta)^{\frac{\theta}{2}}}f_{j}|^{2}\Big\|_{L^{p}_{t}L^{q}_{x}(I \times \M)} \lesssim_{\varepsilon} N^{\varepsilon}\|\lambda\|_{\ell^{1}},
    \end{equation*}
    and by Theorem \ref{Strichartz}, for $$\sigma_{1}=\begin{cases}
(1+\frac{n(\theta-2)}{m+n})\frac{1}{p} & if \quad \theta>3+\frac{m}{n} \\
    \frac{2}{p} & if \quad 2< \theta \leq 3+\frac{m}{n}
\end{cases},$$
we have
    \begin{equation*}
        \left\| \sum_j \lambda_j |e^{it(-\Delta)^\frac{\theta}{2}} P_{\leq N}f_j|^2 \right\|_{L_t^p L_z^q (I \times \M)}\lesssim N^{\sigma_{1}} \| \lambda \|_{\ell^{\alpha'}},
    \end{equation*}
    where $\alpha' \leq \frac{2q}{q+1}.$ 
    By interpolating these two above inequalities, we get the desired result.
\end{proof}
\section{Proof of Theorem  \ref{TIN7}} \ 
\begin{remark}[proof strategy]
\begin{itemize}
    \item First, we eliminate the factor $N^{\sigma}$ by applying the vector-valued version of the Littlewood–Paley theorem together with the triangle inequality, which yields \eqref{WL2}.
    \item Next, using \eqref{WL2}, we establish Proposition \ref{inhomogeneous estimate}, which provides control over the inhomogeneous component of the Duhamel formula for \eqref{HEInP}.
    \item After that, we construct a complete metric space and define a mapping such that any fixed point of this map corresponds to a solution of \eqref{HEInP}. We then show that this mapping is a contraction on the chosen space, making use of \eqref{WL2} and \eqref{WL4}. Finally, we invoke the Banach fixed-point theorem to conclude the proof.
    \end{itemize}
\end{remark}

Throughout this section we shall assume that $p, q, s$ and $\sigma$ be as in  Theorem \ref{TIN7}. We note that the conclusion of Theorems \ref{ose} and \ref{Strichartz} (by the same proof) also holds true by replacing  $P_{\leq N}$ by $P_N$ for any $N \in 2^\mathbb{N}.$  As a consequence, the following estimates 
\begin{eqnarray}\label{WL1}
    \bigg\| \sum_{j}\lambda_j |e^{it(-\Delta)^{\frac{\theta}{2}}}P_{N} f_j|^2 \bigg\|_{L^{p}_{t}L^q_z(I \times \M)} \leq C_{\abs{I},\theta, \sigma} N^{\sigma} \|\lambda\|_{\ell^{\alpha'}}.
\end{eqnarray}
holds true 
for any $N \in 2^\mathbb{N}.$  Next we claim that  for any $\epsilon>0$ any ONS $(f_j)_j$ in $L^2(\M)$,  the following estimate
	\begin{equation}\label{WL2}
    \bigg\| \sum_{j}\lambda_j |e^{it(-\Delta)^{\frac{\theta}{2}}}\langle D \rangle^{-\frac\sigma2 -\varepsilon} f_j|^2 \bigg\|_{L^{p}_{t}L^q_z(I \times \M)} \leq C_{\abs{I},\theta, \sigma, \varepsilon} \|\lambda\|_{\ell^{\alpha'}},
	\end{equation} 
holds true. In fact, by the vector-valued version of the Littlewood-Paley theorem (e.g.  \cite[Lemma 1]{sabin2016littlewood}), triangle inequality, Remark \ref{vvr}, and \eqref{WL1}, we have 
 \begin{flalign*}
     &\bigg\| \sum_{j}\lambda_j |e^{it( - \Delta )^{\frac{\theta}{2}}}\langle D \rangle^{-\frac\sigma2 -\varepsilon} f_j|^2 \bigg\|_{L^{p}_{t}L^q_z(I \times \M)} \\ &\lesssim \bigg\| \sum_{j}\lambda_j |\langle D \rangle^{-\frac\sigma2 -\varepsilon}e^{it( - \Delta )^{\frac{\theta}{2}}} P_1 f_j|^2 \bigg\|_{L^{p}_{t}L^q_z(I \times \M)}+\sum_{N \in 2^{\mathbb{N}} \setminus \{1\}} \bigg\| \sum_{j}\lambda_j |\langle D \rangle^{-\frac\sigma2 -\varepsilon}e^{it( - \Delta )^{\frac{\theta}{2}}} P_{N} f_j|^2 \bigg\|_{L^{p}_{t}L^q_z(I \times \M)} \\ 
     &\lesssim_{\abs{I},\theta, \sigma, \varepsilon} \|\lambda\|_{\ell^{\alpha'}}+\sum_{N \in 2^{\mathbb{N}} \setminus \{1\}} N^{-(\sigma+\varepsilon)} \bigg\| \sum_{j}\lambda_j |e^{it( - \Delta )^{\frac{\theta}{2}}} P_{N} f_j|^2 \bigg\|_{L^{p}_{t}L^q_z(I \times \M)} \lesssim_{\abs{I},\theta, \sigma, \varepsilon} \|\lambda\|_{\ell^{\alpha'}}.
 \end{flalign*}
This proves \eqref{WL2}.

\begin{remark}\label{vvr}
     We recall the vector-valued extension of Bernstein's inequality: for any $\rho \in \mathbb{R},$ $1< q < \infty,$ and any $(g_{j})_{j} \in L^{q}(\ell^2)$ (not necessarily orthonormal), 
 $$\bigg\|\left(\sum_{j}|P_{N}\langle D\rangle^{\rho}g_{j}|^2\right)^{\frac{1}{2}}\bigg\|_{L^{q}(\M)} \lesssim N^{\rho} \bigg\|\left(\sum_{j}|P_{N}g_{j}|^2\right)^{\frac{1}{2}}\bigg\|_{L^{q}(\M)}.$$
For $\M = \mathbb{R}^d$ or $\mathbb{T}^d$, this follows from Hörmander–Mikhlin’s Fourier multiplier theorem, see \cite[Propositions 1 and 4]{zimmermann1989vector} and \cite{nakamura2020orthonormal}. \\

By taking $Tg= P_{N} \langle D \rangle^\rho g$ and $g_{1}=g$ and $g_{j}=0, j>1$ using these values in $L^q$ multiplier on $\mathbb{R}^d,$ so $$\|Tg\|_{L^q(\mathbb{R}^d)} \lesssim N^\rho \|g\|_{L^q(\mathbb{R}^d)}$$ 
Thus $\|T\| \lesssim N^\rho$ on $\mathbb{R}^d.$ From Lemma \ref{Lp multiplier}, we get
$$\|T\| \lesssim N^\rho$$ on $\mathbb{R}^n \times \mathbb{T}^m.$ 
Now applying Proposition \ref{l2 valued extension}, we get the result.

On the other hand, case $\M = \mathbb{R}^n \times \mathbb{T}^m$ follows from Lemma \ref{Lp multiplier} together with Proposition \ref{l2 valued extension}. 
\end{remark}
  
\begin{proposition}\label{inhomogeneous estimate} \ 
		\begin{enumerate}
			\item \label{ss1} If \eqref{WL2} holds, then for all $V \in L_t^{p'} L_z^{q'} (I \times \M) \text{ and } s=\frac{\sigma}{2}+\varepsilon,$
			\begin{equation}\label{WL3}
				\left\| \int_{I} e^{i t ( - \Delta )^{\frac{\theta}{2}}} \langle D \rangle^{-s} V(t,z) \langle D \rangle^{-s} e^{-i t ( - \Delta )^{\frac{\theta}{2}}} dt \right\|_{\Sp^{\alpha} (L_z^{2}(\M))} \leq C_{\abs{I},\theta, \sigma, \varepsilon} \| V \|_{L_t^{p'} L_z^{q'} (I \times \M)}.
			\end{equation}
			\item \label{ss2} (Inhomogeneous estimate) Let \( R(t') : L^{2}(\M) \to L^{2}(\M) \) be self-adjoint for each \( t^{'} \in I \) and define
			\begin{equation*}
				\gamma(t) = \int_{0}^{t} e^{-i (t - t') ( - \Delta )^{\frac{\theta}{2}}} R(t') e^{i (t - t') ( - \Delta )^{\frac{\theta}{2}}} dt' \text{ for all } t \in I.
			\end{equation*}
			If \eqref{WL2} holds,  then we have
			\begin{equation}\label{WL4}
				\left\|\rho_{\langle D \rangle^{-s} \gamma{(t)} \langle D \rangle^{-s}} \right\|_{L_t^{p} L_{z}^{q} (I \times \M)}\leq C_{\abs{I},\theta, \sigma, \varepsilon} \left\| \int_{I} e^{i s ( - \Delta )^{\frac{\theta}{2}}} |R(s)| e^{-i s ( - \Delta )^{\frac{\theta}{2}}} ds \right\|_{\Sp^{\alpha'}(L_{z}^{2}(\M))}.
			\end{equation}
		\end{enumerate}
	\end{proposition} 
	    \begin{proof}  The method of proof follows similar arguments as in  \cite[Section 2.3]{frank2014strichartz} and \cite[Proposition 11]{nealbez2021}. For the convenience  of the reader, we briefly outline it here.\\
        \eqref{ss1}
            Let \( \gamma_{0} = \sum_{j} \lambda_j \lvert f_j \rangle \langle f_j \rvert \), \( (f_j)_{j} \) is the ONS in \( L_z^{2} (\M) \) such that
			$
			\| \gamma_{0} \|_{\Sp^{\alpha'} (L_z^{2}(\M))} = \| \lambda \|_{\ell^{\alpha'}} \leq 1.
			$
          We recall  property of density function $ \rho_{\gamma}$, namely $\int V(z) \rho_{\gamma}(z)dz=\operatorname{Tr}(V\gamma).$
            By the cyclic property of trace and H\"older's inequality, we have
			\begin{flalign*}
				&\left| \operatorname{Tr}_{L^{2}_{z}} \left( \gamma_{0} \int_{I} e^{i t ( - \Delta )^{\frac{\theta}{2}}} \langle D \rangle^{-s} V(t,z) \langle D \rangle^{-s} e^{-i t ( - \Delta )^{\frac{\theta}{2}}} dt \right) \right| \\
				&=\left| \int_{I} \operatorname{Tr}_{L^{2}_{z}} \left( V e^{-i t ( - \Delta )^{\frac{\theta}{2}}} \langle D \rangle^{-s} \gamma_{0} \langle D \rangle^{-s} e^{i t ( - \Delta )^{\frac{\theta}{2}}} \right) dt \right| \\
				&= \left| \int_{I} \int_{\M} V(z) \rho_{e^{-i t ( - \Delta )^{\frac{\theta}{2}}} \langle D \rangle^{-s} \gamma_{0} \langle D \rangle^{-s} e^{i t ( - \Delta )^{\frac{\theta}{2}}}}(z) dz \, dt \right| \\
				&\leq \| V \|_{L_t^{p'} L_z^{q'} (I \times \M)} \left\| \rho_{e^{-i t ( - \Delta )^{\frac{\theta}{2}}} \langle D \rangle^{-s} \gamma_{0} \langle D \rangle^{-s} e^{i t ( - \Delta )^{\frac{\theta}{2}}}}\right\|_{L_t^{p} L_{z}^{q} (I \times \M)} \\
				&= \| V \|_{L_t^{p'} L_z^{q'} (I \times \M)}\left\|\sum_{j} \lambda_j \left| e^{i t ( - \Delta )^{\frac{\theta}{2}}} \langle D \rangle^{-s} f_j \right|^2\right\|_{L_t^{p} L_{z}^{q}(I \times \M)},
			\end{flalign*}
where density function  $\rho_{e^{-i t ( - \Delta )^{\frac{\theta}{2}}} \langle D \rangle^{-s} \gamma_{0} \langle D \rangle^{-s} e^{i t ( - \Delta )^{\frac{\theta}{2}}}}=\sum_{j} \lambda_j \left| e^{i t ( - \Delta )^{\frac{\theta}{2}}} \langle D \rangle^{-s} f_j \right|^2.$ 
Now by invoking the duality argument and \eqref{WL2}, we get  \eqref{WL2}.\\

 \noindent 
 \eqref{ss2} We recall that  $\abs{\operatorname{Tr}(AB)} \leq \operatorname{Tr}(|A||B|)$ for self-adjoint operators $A,B.$ Denote   \[A(t)=e^{i t ( - \Delta )^{\frac{\theta}{2}}} \langle D \rangle^{-s} |V(t,z)| \langle D \rangle^{-s} e^{-i t ( - \Delta )^{\frac{\theta}{2}}}, \quad  B(t')= e^{i t' ( - \Delta )^{\frac{\theta}{2}}} |R(t')| e^{-i t' ( - \Delta )^{\frac{\theta}{2}}}.\]
  Let  \(\| V \|_{L_t^{p'} L_z^{q'}(I \times \M)} \leq 1\).  By \eqref{Holder for Schatten}, we have 
            \begin{align*}
                &\left| \int_{I} \int_{\M} V(t,z) \rho_{\langle D \rangle^{-s} \gamma{(t)} \langle D \rangle^{-s}}(z) \, dz \, dt \right| \\
				&= \left| \int_{I} \operatorname{Tr}_{L_z^2} \left( V(t,z) \langle D \rangle^{-s} \gamma{(t)} \langle D \rangle^{-s} \right) \, dt \right| \\
				&=\left| \int_{I} \int_{0}^{t} \operatorname{Tr}_{L^{2}_{z}} \left( e^{i t ( - \Delta )^{\frac{\theta}{2}}} \langle D \rangle^{-s} V(t,z) \langle D \rangle^{-s} e^{-i t ( - \Delta )^{\frac{\theta}{2}}} e^{i t' ( - \Delta )^{\frac{\theta}{2}}} R(t') e^{-i t' ( - \Delta )^{\frac{\theta}{2}}} \right) dt dt' \right| \\
				&\leq  \int_{I} \int_{0}^{t} \operatorname{Tr}_{L^{2}_{z}} \left( A(t)B(t') \right) dt dt'  \\
				&\leq \operatorname{Tr}_{L^{2}_{z}} \left( \left( \int_{I} A(t) \, dt \right) \left(\int_{I} B(t') dt' \right) \right) \\
				&\leq \left \| \int_{I} A(t) \, dt \right\|_{\Sp^{\alpha}(L^{2}_{z}(\M))} \left \| \int_{I} B(t') \, dt' \right\|_{\Sp^{\alpha'}(L^{2}_{z}(\M))}\\
				&\leq C_{\abs{I},\theta, \sigma, \varepsilon} \| V \|_{L_t^{p'} L_z^{q'} (I \times \M)} \left \| \int_{I} B(t') \, dt' \right\|_{\Sp^{\alpha'}(L^{2}_{z}(\M))}.
            \end{align*}
Invoking duality argument, we get \eqref{WL2}.
	\end{proof}
    \begin{proof}[\textbf{Proof of Theorem \ref{TIN7} \eqref{TIN71}}] Denote $\Sp^{\alpha',s}=\Sp^{\alpha',s}(L^2_z(\M)).$
    Let \(R > 0\) such that 
		$\| \gamma_{0} \|_{\Sp^{\alpha, s} } < R,$ and 
		$T = T(R)>0$ to be chosen later. Set 
		\[
		X_T = \left\{ (\gamma, \rho) \in C ([0, T], \Sp^{\alpha', s} ) \times L_{t}^{p} L_{z}^{q} ([0, T] \times \M) : \| (\gamma, \rho) \|_{X_T} \leq C^* R \right\},
		\]
		where 
		\[
		\| (\gamma, \rho) \|_{X_T} = \| \gamma \|_{C ([0, T], \Sp^{\alpha', s} )} + \| \rho \|_{L_{t}^{p} L_{z}^{q} ([0, T] \times \M)}.
		\]
Here \(C^* > 0\) is to  chosen later independent of \(R\). We define
		\[
		\Phi(\gamma, \rho) = \left( \Phi_1(\gamma, \rho), \rho \left[ \Phi_1(\gamma, \rho) \right] \right),
		\]
\[ \Phi_1(\gamma, \rho)(t) = S(t) + \int_{0}^{t} e^{-i(t-t')( - \Delta )^{\frac{\theta}{2}}} \left[ w \ast \rho(t'), \gamma(t') \right] e^{i(t-t')( - \Delta )^{\frac{\theta}{2}}} \, dt' \]
where   $S(t)= e^{-it( - \Delta )^{\frac{\theta}{2}}} \gamma_{0}  e^{it( - \Delta )^{\frac{\theta}{2}}}$ and    \(\rho[\gamma] = \rho _{\gamma}\) (the density function of $\gamma$). The solution of \eqref{HEInP} is fixed point of \(\Phi.\) i.e.
		\( \Phi(\gamma, \rho _{\gamma}) = (\gamma, \rho _{\gamma}).\)
        Fix any \(t \in [0, T]\) and 
        \begin{align*}
           &\| \Phi_1(\gamma, \rho)(t) \|_{\Sp^{\alpha', s}} \\ 
           &\leq \| S(t) \|_{\Sp^{\alpha', s}}+ \int_{0}^{T} \| e^{-i(t-t')( - \Delta )^{\frac{\theta}{2}}} \left[ w \ast \rho(t'), \gamma(t') \right] e^{i(t-t')( - \Delta )^{\frac{\theta}{2}}} \|_{\Sp^{\alpha', s}} \, dt'.
        \end{align*}
Since $(f_{j})_{j}$ is ONS in $L^{2}_{z}(\M),$ then $(e^{-it(-\Delta)^{\frac{\theta}{2}}}f_{j})_{j}$ is as well for each $t$ and so
		\[
		\| e^{-it( - \Delta )^{\frac{\theta}{2}}} \gamma_{0} e^{it( - \Delta )^{\frac{\theta}{2}}} \|_{\Sp^{\alpha', s}} = \| \gamma_{0} \|_{\Sp^{\alpha', s}} < R,
		\]
		For the second term, again we use the inequality, we get
		\begin{align*}
			&\| e^{-i(t-t')( - \Delta )^{\frac{\theta}{2}}} \left[ w * \rho(t'), \gamma(t') \right] e^{i(t-t')( - \Delta )^{\frac{\theta}{2}}} \|_{\Sp^{\alpha', s}}\\
			&\leq \big\{ \left\| \langle D \rangle^s w * \rho(t') \langle D \rangle^{-s} \right\|_{\Sp^{\infty} }
			+
			\left\| \langle D \rangle^{-s} w \ast \rho(t') \langle D \rangle^{s} \right\|_{\Sp^{\infty} } \big\}\| \gamma(t') \|_{\Sp^{\alpha', s}}.
		\end{align*}
	The estimate we employ to evaluate the above nonlinear term is the following (see in \cite[Corollary on p. 205]{HansII}), where the inequality was proved for the \(\mathbb{R}^d\) case, but the same proof is applicable for the \(\M= \mathbb{T}^d\) or \(\mathbb{R}^n \times \mathbb{T}^m\) case:
		\[
		\| f \cdot g \|_{H^{r}(\M)} \leq C_{s, \delta} \| f \|_{B_{\infty, \infty}^{|r|+\delta}(\M)} \| g \|_{H^{r}(\M)},
		\]
		where \(r \in \mathbb{R}\) and \(\delta > 0\) are arbitrary.
		From this estimate and Young's inequality
		\begin{align*}
			\left\| \langle D \rangle^s w \ast \rho(t') \langle D \rangle^{-s} \right\|_{\Sp^{\infty} } 
			&\leq C'_{s, \delta} \| w \ast \rho(t') \|_{B^{s+\delta}_{\infty, \infty}(\M)}\\
			&\leq C'_{s, \delta} \| w \|_{B^{s+\delta}_{q', \infty}(\M)} \| \rho(t') \|_{L^q_z(\M)}.
		\end{align*}
		Similarly,
		\[
		\| \langle D \rangle^{-s} w \ast \rho(t') \langle D \rangle^s \|_{\Sp^{\infty}} \leq C'_{-s, \delta} \| w \|_{B^{s+\delta}_{q', \infty}} \| \rho(t') \|_{L^q_z(\M)}.
		\]
        In total from \((\gamma, \rho) \in X_T\), we estimate
		$$\resizebox{\textwidth}{!}{$\int_0^T \| e^{-i(t-t')(-\Delta)} D^{\theta} [ w \ast \rho(t'), \gamma(t')] e^{i(t-t')( - \Delta )^{\frac{\theta}{2}}} \|_{\Sp^{\alpha', s}} \, dt'\\  \leq  C_{s, \delta} \| w \|_{B^{s+\delta}_{q', \infty}(\M)} T^{1 / p'} (C^*R)^2.$}$$
		where \(C_{s, \delta} = C'_{s, \delta} + C'_{-s, \delta}.\)
        Thus we have
  \begin{eqnarray}\label{WL5}
      \| \Phi_{1} (\gamma, \rho) \|_{C \left([0, T] ;\Sp^{\alpha', s} \right)} \leq R+c_{s,\delta}T^{\frac{1}{p'}} \| w \|_{B^{s+\delta}_{q', \infty}(\M)} (C^*R)^2.
  \end{eqnarray}
Similarly, by using Proposition \ref{inhomogeneous estimate}, we get
\begin{align*}
    &T^{-\frac{1}{p}} C^{-1}_{\theta, \sigma, \epsilon} \|\rho[ \Phi_{1} (\gamma, \rho) ] \|_{L_t^{p} L_{z}^{q} ([0, T] \times \M)} \\
    &\leq \| \langle D \rangle^s \gamma_{0} \langle D \rangle^s \|_{\Sp^{\alpha'} }+\left\| \int_0^T e^{i t' ( - \Delta )^{\frac{\theta}{2}}} \langle D \rangle^s \left| [ w \ast \rho(t') , \gamma(t')] \right| \langle D \rangle^s e^{-i t' ( - \Delta )^{\frac{\theta}{2}}} \, dt' \right \|_{\Sp^{\alpha'} }
\end{align*}
		For the first term,
		\[
		\| \langle D \rangle^s \gamma_{0} \langle D \rangle^s \|_{\Sp^{\alpha'} } = \| \gamma_{0} \|_{\Sp^{\alpha', s} } \leq R.
		\]
		For the second term, we may employ the same argument as for
		\[
		\| \Phi_{1} (\gamma, \rho) \|_{C \left([0, T] ;\Sp^{\alpha', s} \right)},
		\]
		and we get
		\begin{eqnarray}\label{WL6}
		    \| \rho [ \Phi_1 (\gamma, \rho) ] \|_{L_t^{p} L_{z}^{q} ([0, T] \times \M)} \leq C_{\theta, \sigma, \epsilon} T^{\frac{1}{p}} \left\{R + C_{s, \delta} T^{\frac{1}{p'}} \| w \|_{B^{s+\delta}_{q', \infty}(\M)} (C^* R)^2 \right\}.
		\end{eqnarray}
		These two estimates that 
		\[
		\|\Phi(\gamma, \rho) \|_{X_{T}} \leq (1+C_{\theta, \sigma, \epsilon} T^{1/p}) \left\{R + C_{s, \delta} T^{\frac{1}{p'}} \| w \|_{B^{s+\delta}_{q', \infty}(\M)} (C^* R)^2 \right\}
		\]
		Thus $\Phi :X_{T} \to X_{T}$ if we take $C^*$ large and $T$ (Specifically, $T$ relies on $\| w \|_{B^{s+\delta}_{q', \infty}(\M)}$, not $\| w \|_{B^{s}_{q', \infty}(\M)}$, but this distinction is inconsequential since $s=\sigma+\varepsilon$ and both $\varepsilon$ and $\delta$ can be arbitrarily small) small enough. Similarly we can prove $\Phi$ is a contraction on $X_{T}$ and has a fixed point. Uniqueness can be proved with similar estimate.
    \end{proof}
\begin{proof}[\textbf{Proof of Theorem \ref{TIN7} \eqref{TIN72}}]
    Fix any $T > 0.$ From the estimates \eqref{WL5} and \eqref{WL6}, which we have already proven. These two estimates gives us that 
  \begin{eqnarray*}
      \|\Phi(\gamma, \rho) \|_{X_{T}} \leq (1+C_{\theta, \sigma, \epsilon} T^{1/p}) \left\{\| \gamma_{0} \|_{\Sp^{\alpha', s}} + C_{s, \delta} T^{\frac{1}{p'}} \| w \|_{B^{s+\delta}_{q', \infty}(\M)} \| (\gamma, \rho) \|_{X_T}^2 \right\}.
  \end{eqnarray*}
  With this in mind, we choose $R_{T}=R(\| w \|_{B^{s+\delta}_{q', \infty}(\M)})$ small enough (precisely $R_{T}$ depends on $\| w \|_{B^{s+\delta}_{q', \infty}(\M)},$ not $\| w \|_{B^{s}_{q', \infty}(\M)},$ but this is harmless since $s=\sigma+\varepsilon$ and $\varepsilon,\delta$ are arbitrarily small) so that we can find $M>0$ such that for any $y \in [0,M],$ it holds that $$(1+C_{\theta, \sigma, \epsilon} T^{1/p}) \left\{\| \gamma_{0} \|_{\Sp^{\alpha', s}} + C_{s, \delta} T^{\frac{1}{p'}} \| w \|_{B^{s+\delta}_{q', \infty}(\M)} y^2 \right\} \leq M,$$ as long as $\| \gamma_{0} \|_{\Sp^{\alpha', s}} \leq R_{T}.$\\
  Define the space $$X_{T,M} :=\left\{ (\gamma, \rho) \in X_{T} \mid \| (\gamma, \rho) \|_{X_{T}} \leq M \right\},$$ then by similar way we see that $\Phi : X_{T,M} \to X_{T,M}$ is contraction map by choosing sufficiently small $R_{T},$ and hence from fixed point theorem we find a solution $$\gamma \in C ([0, T]; \Sp^{\frac{2 q}{q+1}, s} ),$$ satisfying \eqref{HEInP} on $[0,T] \times \M,$ and $\rho_{\gamma} \in L_{t}^{p} L_{z}^{q} ([0, T] \times \M).$
\end{proof}
  
\section*{Acknowledgement}
The first author wishes to thank the National Board for Higher Mathematics (NBHM) for the research fellowship and Indian Institute of Science Education and Research, Pune for the support provided during the period of this work. The author also thanks S. S. Mondal for helpful discussions.

    
	\bibliographystyle{siam}
	\bibliography{main.bib}
\end{document}